\documentclass[11pt, makeidx]{amsart}
\usepackage{amssymb,amsfonts,amsthm}
\usepackage{srcltx}
\usepackage{float}

\usepackage{color}
\makeindex

\makeatletter
\def\sssub{\@startsection{paragraph}{4}}

\renewcommand\paragraph{\@startsection{paragraph}{4}{\z@}{1.25ex}{0.0001pt}{\normalfont\normalsize\em}}

\makeatother

\numberwithin{paragraph}{subsubsection}

\newcommand{\R}{{\mathbb R}}
\newcommand{\HH}{{\mathbb H}}
\newcommand{\HNN}{\mathrm{HNN}}

\newcommand{\Z}{{\mathbb Z}}
\newcommand{\SL}{{\mathrm{SL}}}

\newcommand{\GL}{{\mathrm{GL}}}
\newcommand{\SO}{{\mathrm{SO}}}

\newcommand{\e }{\varepsilon }

\newcommand{\N}{{\mathbb{N}}}

\newcommand{\la}{\langle}
\newcommand{\ra}{\rangle}
\newcommand{\g}{{\mathfrak g}}
\newcommand{\Con}{{\mathrm{Con}}}
\newcommand{\bbb}{{\mathcal B}}
\newcommand{\dist}{{\mathrm{dist}}}

\newcommand{\cost}{{\mathrm{time}}}

\newtheorem{theorem}{Theorem}[section]
\newtheorem{ex}[theorem]{Example}
\newtheorem{exercise}[theorem]{Exercise}
\newtheorem{example}[theorem]{Example}
\newtheorem{lemma}[theorem]{Lemma}

\newtheorem{cy}[theorem]{Corollary}

\newtheorem{prop}[theorem]{Proposition}
\theoremstyle{definition}
\newtheorem{df}[theorem]{Definition}

\newtheorem{rk}[theorem]{Remark}
\newtheorem{prob}[theorem]{Problem}

\newtheorem{ctt}[theorem]{Quasi-Theorem}

\newcommand{\area}{\mathrm{Area}}

\newcommand{\ttt}{{\mathcal T}}

\newcommand{\vk}{van Kampen }

\newcommand{\ccc}{{\mathcal C}}

\newcommand{\cee}{\alpha}

\newcommand{\dol}{\omega}
\newcommand{\iv}{^{-1}}

\newcommand{\pp}{{\mathcal P} }

\newcommand{\sss}{{\mathcal S} }

\begin{document}

\def\thesubsubsection       {\thesubsection.\Alph{subsubsection}}

\def\theparagraph       {\thesubsubsection\arabic{paragraph}}

\title[Asymptotic invariants and complexity of groups]{Asymptotic invariants, complexity of groups and related problems}
\author{Mark Sapir}\thanks{The research was supported in part by NSF grant DMS-0700811.}
\address{Department of Mathematics, Vanderbilt University, Nashville, TN 37240, U.S.A.}
\email{m.sapir@vanderbilt.edu}

\begin{abstract} We survey results about computational complexity of the word problem in groups, Dehn functions of groups and related problems.
\end{abstract}

\maketitle
\tableofcontents

\section{Introduction}\label{I}

The word and conjugacy problems are the most classical algorithmic problem for groups going back to the work of Dehn and Tietze at the beginning of the 20th century. These very basic problems are interested by themselves and because of applications to other areas of mathematics, especially topology (where the analogs of these problems are problems of existence of pointed or free homotopy between loops). There are very many papers devoted to these problems for various classes of groups. Different aspects of these problems are discussed in many books and surveys (see, for example, \cite{BCL, Br, Ch, DrutuIJAC,Eautomatic,Gi, KS, LS,  Mi, OlSoros, OSsur, Ro, SaICM}). Still there are very many new results and ideas that appeared after the last of these surveys was published. We can mention, for example, the construction of finitely presented groups with polynomial-non-recursive and even quadratic-non-recursive Dehn functions \cite{OSnq}, finding a nilpotent finitely generated group with Dehn function not of the form $n^\alpha$ for any $\alpha$ \cite{Wen2}, the proof that all groups $\SL_n(\Z)$, $n\ge 5$, the R.Thompson group $F$, extensions of finitely generated free groups by cyclic groups,  have quadratic Dehn functions \cite{Young,GubaF,BrGr}, the description of all FFFL (space) functions of finitely presented groups and the algebraic characterization of groups with PSPACE word problem \cite{OlFFFL}, solving the isomorphism problem for arbitrary hyperbolic groups \cite{DaG} and many others. Several new methods are used in the proofs of these results and our goal in this survey is to give as gentle as possible an introduction to these results and methods. To this end, we often present not the results in their full generality but their easier to explain approximations. In this regard, this survey is similar to our survey \cite{KS} with Olga Kharlampovich.

Another feature of this survey is the emphasis on ``related problems". For example we consider hyperbolic groups (Section \ref{hgac}). This is the class of groups with the smallest possible (linear) Dehn functions. The  word problem in a hyperbolic group can be solved in linear time (in fact by a real time Turing machine \cite{Hol}). We mention that this class is very large: almost all finitely presented groups are hyperbolic. This leads us to the discussion of various probabilistic models used to clarify the words ``almost all", to small cancelation conditions (including coarse small cancelation conditions over hyperbolic group), to different ways of constructing hyperbolic groups,  Gromov-Olshanskii theory of quotients of hyperbolic groups and various combination theorems, to various monsters constructed as limits of hyperbolic groups and to Gromov's random groups. We also discuss various weakening of the linear isoperimetric inequality including Wenger's result from \cite{Wen1} and Cartan-Hadamard local-to-global type theorems of Gromov \cite{GrHyp} and others. Similarly, when we consider groups with quadratic Dehn functions (Section \ref{dfaapog}), we describe the very diverse zoo of examples of these groups, present ideas of the proofs of results from \cite{Young}, \cite{GubaF} and others. Then we notice that all known groups with quadratic Dehn functions (except for automatic groups) are known to have solvable conjugacy problem although the solvability of conjugacy problem was proved by very  different methods in different classes of groups. So we formulate and discuss a general problem (due to Rips).

Most of the paper is devoted to the Dehn functions which measure the time complexity of the word problem. One can argue that after the \label{pa1}{\em \index{Growth function}growth function} (first introduced by A.~S.~\v Svarc in \cite{Sv}, and later by Milnor \cite{M}) which counts the number of different elements of the group that can be represented by products of generators of length $\le n$, this is possibly the most important and basic geometric invariant of a group.\footnote{Dehn functions are possibly even more important than growth functions because the set of Dehn functions is much more diverse than the set of growth functions (see Section \ref{dfaapog}), so Dehn functions capture more information about a group.} We also discuss the space complexity and the recent results of Olshanskii \cite{OlFFFL}. We present an algebraic characterization of groups with word problem in NP from \cite{BORS} and groups with word problem in PSPACE from \cite{OlFFFL}, discuss examples of groups with NP-complete \cite{SBR} and  coNP-complete \cite{Bir1} word problems.

Note that we could not survey all aspects of complexity of the word and conjugacy problems in groups. For example we do not talk about the average case complexity, generic solutions of algorithmic problems,  search problems and other complexity related issues used in group based cryptography \cite{MSU}. Fortunately, there is a nice recent survey by Shpilrain \cite{Shpil} where at least some of these topics are discussed.

{\bf Acknowledgement.} I am grateful to Efim Zelmanov who inspired me to write this survey. Several people contributed with suggestions, comments and even pieces of text. I am especially grateful to Martin Bridson, Fran\c cois Dahmani, Daniel Groves, Victor Guba, Sergei Ivanov,  Bruce Kleiner, Igor Lys\"enok, Alexander Olshanskii, Denis Osin,  Tim Riley, Stefan Wenger and Robert Young.

\section{Algorithmic problems in groups}

Let $X$ be a set. A \label{pa2}{\em \index{Word}word} over $X$ is a sequence of elements of $X$. A \label{pa3} {\em \index{Group word}group word} over $X$ is a sequence of elements of $X$ and their inverses, i.e. symbols $x\iv$, $x\in X$. The length of a word $W$ is denoted by $|W|$. A group word is called \label{pa4}{\em \index{Group word!reduced}reduced} if it does not contain subwords of the form $xx\iv$ and $x\iv x$, $x\in X$. Every group word can be made reduced by removing all subwords of these forms. The set $F(X)$ of all reduced group words over $X$ is equipped with the binary operation: the product $UV$ of two reduced group words $U, V$ is the result of reducing the concatenation of $U$ and $V$. With this operation, the set $F(X)$ turns into a group which is called the \label{pa5}{\em \index{Group!free}free group} over $X$. The identity element of that group is the empty word denoted by $\emptyset$. For every group $G$, any map $X\to G$ extends uniquely to a homomorphism $F(X)\to G$: any word is mapped to the product of the images of its letters. In particular, every group $G$ generated by at most $|X|$ elements is a homomorphic image of $F(X)$ so that the image of $X$ generates $G$. Let $\phi$ be one of these (surjective) homomorphisms.

The \label{pa6}{\em \index{Word problem}word problem} in $G$ (related to $\phi$) is the following

\begin{prob}\rm{(}The word problem\rm{)}\label{p1}.

{\bf Input:} A reduced group word $W$ over $X$.

{\bf Output:} ``Yes" if $\phi(W)=1$ in $G$ and ``No" otherwise.

\end{prob}

Let us also introduce a ``close cousin" of the word problem, the \label{pa7}{\em conjugacy problem}. We say that two elements $a,b$ of a group $G$ are \label{pa8}{\em conjugate} if there exists an element $t$ (called a \label{pa9}{\em conjugator}) such that $tat\iv=b$.

The \label{pa10}{\em \index{Conjugacy problem}conjugacy problem} in $G$ (related to $\phi$) is the following

\begin{prob}\rm{(}The conjugacy problem\rm{)}\label{p2}.

{\bf Input:} Two reduced group words $U, V$ over $X$.

{\bf Output:} ``Yes" if $\phi(U)$ and $\phi(V)$ are conjugate in $G$ and ``No" otherwise.

\end{prob}

In principle, the existence of an algorithm to solve Problem \ref{p1} or \ref{p2} depends not only on $G$ but also on $\phi$. There is, however, an important case when $\phi$ does not matter. Suppose that $X$ is finite. Then $G=\phi(F(X))$ is called \label{pa11}{\em \index{Group!finitely generated}finitely generated}. Let $N$ be the kernel of $\phi$. It is a normal subgroup of $F(X)$. Suppose that $N$ is generated as a normal subgroup by a finite set $R$ (``generated as a normal subgroup" means that every element of $N$ is a product of conjugates of elements of $R$ and their inverses). Then we say that $G$ has a finite presentation $\la X\mid R\ra$. For {\em \index{Group!finitely presented} finitely presented} groups the solvability of the word or conjugacy problem (i.e. the existence of the needed algorithm) does not depend on the choice of $\phi$ or even on the choice of the finite set $X$ (as long as $\phi$ is surjective).

Note that the conjugacy problem is stronger than the word problem, that is if the conjugacy problem is solvable in $G$, then the word problem is also solvable. Indeed, only the identity element can be a conjugate of the identity element.

\subsection{Van Kampen diagrams}\label{s:vkd}

\subsubsection{The definition}\label{td}
Let $G$ be a finitely presented group given by a finite set of generators $X$ and a finite set of relators $R$, $\phi$ is a homomorphism from $F(X)$ onto $G$, $N$ is the kernel of this homomorphism. Then words in $N$ are precisely products of generators of $G$ and their inverses that are equal to 1 in $G$. The {\em word problem} for $G$ is then the membership problem in $N$ for the elements of $F(X)$. Note that if the product $uv$ is in $N$ then $vu$ is also in $N$ (as a conjugate of $uv$). The word $vu$ is called a \label{pa12}{\em \index{Word!cyclic shift of}cyclic shift} of $uv$. We shall always assume (without loss of generality) that $R$ consists of {\em \index{Word!cyclically reduced}cyclically reduced} words, that is group words all of whose cyclic shifts are reduced.
Since $N$ is generated by $R$ as a normal subgroup, a reduced group word $W$ from $F(X)$ is in $N$ if and only if $W$ can be represented in the free group as a product of conjugates of elements of $R$ and their inverses:


\begin{equation}\label{e1}
W=\prod_{i=1}^m s_i r_i^{\pm 1}s_i\iv
\end{equation}

\noindent where $s_i\in F(X)$, $r_i\in R$.

The minimal number $m$ in all representations (\ref{e1}) of $W$ is called the \label{pa13}{\em
\index{Word!area of}area} of $W$ for the reasons explained below.
For every representation (\ref{e1}), we can draw a planar {\em diagram}, a bouquet  of ``lollipops" which is a planar labeled graph. Each ``lollipop" corresponds to one of the factors $s_ir_is_i\iv$, it has a stem, a path labeled by $s_i$ (i.e. the stem is subdivided into edges labeled by the letters of $s_i$), and a candy, a cycle path labeled by $r_i$ (see Figure \ref{f1}).
\begin{center}
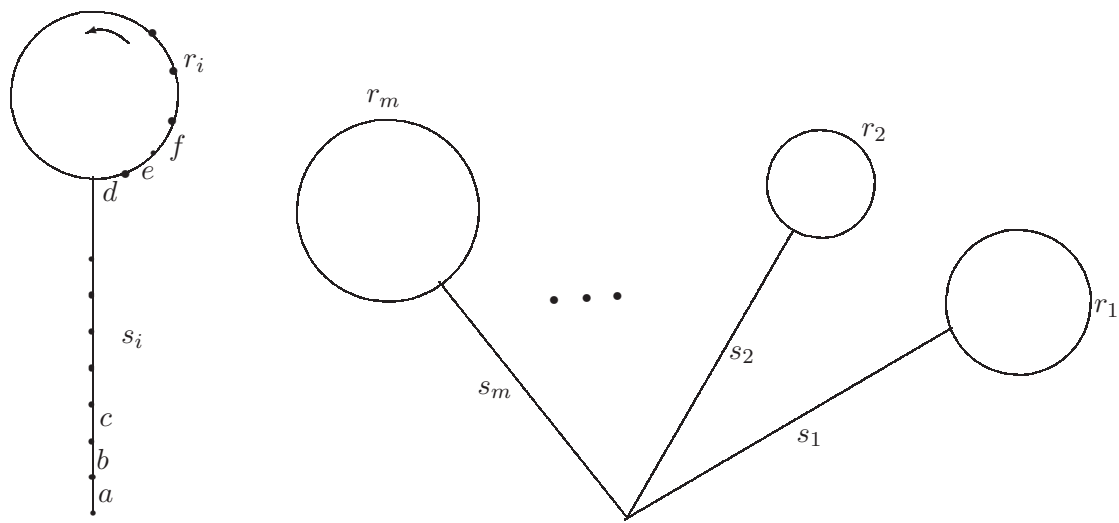
\begin{figure}[ht]
\unitlength .75mm 
\linethickness{0.4pt}
\ifx\plotpoint\undefined\newsavebox{\plotpoint}\fi 
\begin{picture}(197.75,99.576)(0,0)
\put(33.076,84.75){\line(0,1){.7411}}
\put(33.058,85.491){\line(0,1){.7392}}
\put(33.002,86.23){\line(0,1){.7355}}
\multiput(32.91,86.966)(-.03227,.18249){4}{\line(0,1){.18249}}
\multiput(32.78,87.696)(-.033081,.144521){5}{\line(0,1){.144521}}
\multiput(32.615,88.418)(-.033552,.118906){6}{\line(0,1){.118906}}
\multiput(32.414,89.132)(-.02959,.08781){8}{\line(0,1){.08781}}
\multiput(32.177,89.834)(-.030171,.076641){9}{\line(0,1){.076641}}
\multiput(31.906,90.524)(-.030568,.067534){10}{\line(0,1){.067534}}
\multiput(31.6,91.199)(-.030823,.059928){11}{\line(0,1){.059928}}
\multiput(31.261,91.859)(-.030965,.053454){12}{\line(0,1){.053454}}
\multiput(30.889,92.5)(-.033598,.051839){12}{\line(0,1){.051839}}
\multiput(30.486,93.122)(-.0333665,.0462414){13}{\line(0,1){.0462414}}
\multiput(30.052,93.723)(-.0330906,.0413361){14}{\line(0,1){.0413361}}
\multiput(29.589,94.302)(-.0327744,.0369884){15}{\line(0,1){.0369884}}
\multiput(29.097,94.857)(-.0324208,.0330975){16}{\line(0,1){.0330975}}
\multiput(28.579,95.386)(-.0363036,.0335313){15}{\line(-1,0){.0363036}}
\multiput(28.034,95.889)(-.0379343,.0316748){15}{\line(-1,0){.0379343}}
\multiput(27.465,96.364)(-.0422894,.0318633){14}{\line(-1,0){.0422894}}
\multiput(26.873,96.811)(-.0472006,.0319951){13}{\line(-1,0){.0472006}}
\multiput(26.259,97.226)(-.052803,.032062){12}{\line(-1,0){.052803}}
\multiput(25.626,97.611)(-.059279,.032054){11}{\line(-1,0){.059279}}
\multiput(24.974,97.964)(-.066888,.031956){10}{\line(-1,0){.066888}}
\multiput(24.305,98.283)(-.076002,.031747){9}{\line(-1,0){.076002}}
\multiput(23.621,98.569)(-.08718,.031398){8}{\line(-1,0){.08718}}
\multiput(22.923,98.82)(-.101304,.030858){7}{\line(-1,0){.101304}}
\multiput(22.214,99.036)(-.119839,.030048){6}{\line(-1,0){.119839}}
\multiput(21.495,99.217)(-.14543,.028825){5}{\line(-1,0){.14543}}
\put(20.768,99.361){\line(-1,0){.7334}}
\put(20.035,99.468){\line(-1,0){.7379}}
\put(19.297,99.539){\line(-1,0){.7405}}
\put(18.556,99.573){\line(-1,0){.7413}}
\put(17.815,99.57){\line(-1,0){.7402}}
\put(17.075,99.529){\line(-1,0){.7373}}
\put(16.337,99.452){\line(-1,0){.7325}}
\multiput(15.605,99.338)(-.145174,-.030089){5}{\line(-1,0){.145174}}
\multiput(14.879,99.188)(-.119574,-.03109){6}{\line(-1,0){.119574}}
\multiput(14.162,99.001)(-.101031,-.031738){7}{\line(-1,0){.101031}}
\multiput(13.454,98.779)(-.086904,-.032155){8}{\line(-1,0){.086904}}
\multiput(12.759,98.522)(-.075723,-.032407){9}{\line(-1,0){.075723}}
\multiput(12.078,98.23)(-.066607,-.032536){10}{\line(-1,0){.066607}}
\multiput(11.412,97.905)(-.058998,-.032568){11}{\line(-1,0){.058998}}
\multiput(10.763,97.547)(-.052522,-.03252){12}{\line(-1,0){.052522}}
\multiput(10.132,97.156)(-.0469206,-.0324044){13}{\line(-1,0){.0469206}}
\multiput(9.522,96.735)(-.0420107,-.0322299){14}{\line(-1,0){.0420107}}
\multiput(8.934,96.284)(-.0376574,-.0320035){15}{\line(-1,0){.0376574}}
\multiput(8.369,95.804)(-.03376,-.0317304){16}{\line(-1,0){.03376}}
\multiput(7.829,95.296)(-.0321318,-.0333782){16}{\line(0,-1){.0333782}}
\multiput(7.315,94.762)(-.0324515,-.037272){15}{\line(0,-1){.037272}}
\multiput(6.828,94.203)(-.0327299,-.0416223){14}{\line(0,-1){.0416223}}
\multiput(6.37,93.62)(-.032963,-.0465298){13}{\line(0,-1){.0465298}}
\multiput(5.942,93.015)(-.033146,-.052129){12}{\line(0,-1){.052129}}
\multiput(5.544,92.39)(-.033271,-.058604){11}{\line(0,-1){.058604}}
\multiput(5.178,91.745)(-.033331,-.066214){10}{\line(0,-1){.066214}}
\multiput(4.845,91.083)(-.03331,-.07533){9}{\line(0,-1){.07533}}
\multiput(4.545,90.405)(-.033191,-.086513){8}{\line(0,-1){.086513}}
\multiput(4.279,89.713)(-.032944,-.100645){7}{\line(0,-1){.100645}}
\multiput(4.049,89.008)(-.032517,-.119193){6}{\line(0,-1){.119193}}
\multiput(3.854,88.293)(-.031822,-.144803){5}{\line(0,-1){.144803}}
\multiput(3.694,87.569)(-.03068,-.18277){4}{\line(0,-1){.18277}}
\put(3.572,86.838){\line(0,-1){.7363}}
\put(3.486,86.102){\line(0,-1){.7397}}
\put(3.437,85.362){\line(0,-1){.7412}}
\put(3.424,84.621){\line(0,-1){.7409}}
\put(3.449,83.88){\line(0,-1){.7387}}
\put(3.511,83.141){\line(0,-1){.7347}}
\multiput(3.61,82.407)(.027084,-.145764){5}{\line(0,-1){.145764}}
\multiput(3.746,81.678)(.028613,-.12019){6}{\line(0,-1){.12019}}
\multiput(3.917,80.957)(.029644,-.101665){7}{\line(0,-1){.101665}}
\multiput(4.125,80.245)(.030353,-.08755){8}{\line(0,-1){.08755}}
\multiput(4.368,79.545)(.030836,-.076376){9}{\line(0,-1){.076376}}
\multiput(4.645,78.857)(.031154,-.067265){10}{\line(0,-1){.067265}}
\multiput(4.957,78.185)(.031343,-.059658){11}{\line(0,-1){.059658}}
\multiput(5.302,77.529)(.031429,-.053182){12}{\line(0,-1){.053182}}
\multiput(5.679,76.89)(.0314284,-.0475798){13}{\line(0,-1){.0475798}}
\multiput(6.087,76.272)(.0313554,-.0426673){14}{\line(0,-1){.0426673}}
\multiput(6.526,75.674)(.0334489,-.0410467){14}{\line(0,-1){.0410467}}
\multiput(6.995,75.1)(.0330948,-.036702){15}{\line(0,-1){.036702}}
\multiput(7.491,74.549)(.0327075,-.0328143){16}{\line(0,-1){.0328143}}
\multiput(8.014,74.024)(.0365939,-.0332143){15}{\line(1,0){.0365939}}
\multiput(8.563,73.526)(.0409375,-.0335825){14}{\line(1,0){.0409375}}
\multiput(9.136,73.056)(.0425649,-.0314943){14}{\line(1,0){.0425649}}
\multiput(9.732,72.615)(.0474771,-.0315834){13}{\line(1,0){.0474771}}
\multiput(10.349,72.204)(.053079,-.031602){12}{\line(1,0){.053079}}
\multiput(10.986,71.825)(.059556,-.031537){11}{\line(1,0){.059556}}
\multiput(11.641,71.478)(.067163,-.031373){10}{\line(1,0){.067163}}
\multiput(12.313,71.165)(.076275,-.031085){9}{\line(1,0){.076275}}
\multiput(13,70.885)(.08745,-.030638){8}{\line(1,0){.08745}}
\multiput(13.699,70.64)(.101568,-.029976){7}{\line(1,0){.101568}}
\multiput(14.41,70.43)(.120096,-.029005){6}{\line(1,0){.120096}}
\multiput(15.131,70.256)(.145675,-.027559){5}{\line(1,0){.145675}}
\put(15.859,70.118){\line(1,0){.7344}}
\put(16.593,70.017){\line(1,0){.7385}}
\put(17.332,69.952){\line(1,0){.7408}}
\put(18.073,69.925){\line(1,0){.7412}}
\put(18.814,69.935){\line(1,0){.7398}}
\put(19.554,69.981){\line(1,0){.7366}}
\multiput(20.29,70.065)(.18287,.03009){4}{\line(1,0){.18287}}
\multiput(21.022,70.185)(.144906,.03135){5}{\line(1,0){.144906}}
\multiput(21.746,70.342)(.119299,.032128){6}{\line(1,0){.119299}}
\multiput(22.462,70.535)(.100752,.032615){7}{\line(1,0){.100752}}
\multiput(23.167,70.763)(.086621,.032909){8}{\line(1,0){.086621}}
\multiput(23.86,71.026)(.075438,.033065){9}{\line(1,0){.075438}}
\multiput(24.539,71.324)(.066322,.033115){10}{\line(1,0){.066322}}
\multiput(25.203,71.655)(.058713,.03308){11}{\line(1,0){.058713}}
\multiput(25.848,72.019)(.052237,.032976){12}{\line(1,0){.052237}}
\multiput(26.475,72.415)(.046637,.0328112){13}{\line(1,0){.046637}}
\multiput(27.081,72.841)(.0417288,.032594){14}{\line(1,0){.0417288}}
\multiput(27.666,73.298)(.0373776,.0323298){15}{\line(1,0){.0373776}}
\multiput(28.226,73.783)(.0334827,.0320228){16}{\line(1,0){.0334827}}
\multiput(28.762,74.295)(.0318403,.0336564){16}{\line(0,1){.0336564}}
\multiput(29.272,74.833)(.0321261,.0375528){15}{\line(0,1){.0375528}}
\multiput(29.753,75.397)(.0323667,.0419054){14}{\line(0,1){.0419054}}
\multiput(30.207,75.983)(.0325571,.0468147){13}{\line(0,1){.0468147}}
\multiput(30.63,76.592)(.032691,.052415){12}{\line(0,1){.052415}}
\multiput(31.022,77.221)(.03276,.058892){11}{\line(0,1){.058892}}
\multiput(31.382,77.869)(.032753,.066501){10}{\line(0,1){.066501}}
\multiput(31.71,78.534)(.032654,.075617){9}{\line(0,1){.075617}}
\multiput(32.004,79.214)(.032438,.086799){8}{\line(0,1){.086799}}
\multiput(32.263,79.909)(.032067,.100927){7}{\line(0,1){.100927}}
\multiput(32.488,80.615)(.031479,.119472){6}{\line(0,1){.119472}}
\multiput(32.677,81.332)(.030562,.145075){5}{\line(0,1){.145075}}
\put(32.83,82.057){\line(0,1){.7321}}
\put(32.946,82.79){\line(0,1){.737}}
\put(33.026,83.527){\line(0,1){1.2235}}
\put(18,70.25){\line(0,-1){59.5}}
\put(18,10.75){\circle*{.707}}
\put(17.75,17){\circle*{.707}}
\put(17.75,23.5){\circle*{.707}}
\put(17.75,30){\circle*{.707}}
\put(17.75,36.5){\circle*{.707}}
\put(17.75,43){\circle*{.707}}
\put(17.75,49.5){\circle*{.707}}
\put(18,17){\circle*{1.118}}
\put(17.75,23.25){\circle*{1.118}}
\put(17.75,29.75){\circle*{1.118}}
\put(17.75,36.25){\circle*{1.118}}
\put(17.75,42.75){\circle*{1.118}}
\put(17.75,49.25){\circle*{1.118}}
\put(17.75,55.75){\circle*{1.118}}
\put(20.25,13.75){\makebox(0,0)[cc]{$a$}}
\put(19.75,20.25){\makebox(0,0)[cc]{$b$}}
\put(20.25,27.25){\makebox(0,0)[cc]{$c$}}
\put(25,41.5){\makebox(0,0)[cc]{$s_i$}}
\put(23.75,70.75){\circle*{1.5}}
\put(28.75,74.5){\circle*{1}}
\put(32,80.25){\circle*{1.5}}
\put(32.25,89){\circle*{1.581}}
\put(28.5,95.75){\circle*{1.581}}
\put(21,67.75){\makebox(0,0)[cc]{$d$}}
\put(27.75,71){\makebox(0,0)[cc]{$e$}}
\put(33,75.5){\makebox(0,0)[cc]{$f$}}
\put(36,90.5){\makebox(0,0)[cc]{$r_i$}}
\put(16.75,95.75){\vector(-3,-1){.07}}\qbezier(24.25,94)(20.5,97.625)(16.75,95.75)
\multiput(112.25,9.5)(.05753968254,.03373015873){1008}{\line(1,0){.05753968254}}
\multiput(112.5,9.5)(.03373015873,.05839002268){882}{\line(0,1){.05839002268}}
\multiput(112.5,10)(-.03372210953,.04234279919){986}{\line(0,1){.04234279919}}
\put(194.838,48){\line(0,1){.6626}}
\put(194.821,48.663){\line(0,1){.6609}}
\put(194.77,49.324){\line(0,1){.6573}}
\put(194.684,49.981){\line(0,1){.6521}}
\multiput(194.565,50.633)(-.03053,.129009){5}{\line(0,1){.129009}}
\multiput(194.412,51.278)(-.030957,.106051){6}{\line(0,1){.106051}}
\multiput(194.227,51.914)(-.031191,.08941){7}{\line(0,1){.08941}}
\multiput(194.008,52.54)(-.031294,.076721){8}{\line(0,1){.076721}}
\multiput(193.758,53.154)(-.0313,.06667){9}{\line(0,1){.06667}}
\multiput(193.476,53.754)(-.031229,.058469){10}{\line(0,1){.058469}}
\multiput(193.164,54.339)(-.031096,.051617){11}{\line(0,1){.051617}}
\multiput(192.822,54.906)(-.033719,.049943){11}{\line(0,1){.049943}}
\multiput(192.451,55.456)(-.033231,.044125){12}{\line(0,1){.044125}}
\multiput(192.052,55.985)(-.0327358,.0390931){13}{\line(0,1){.0390931}}
\multiput(191.627,56.493)(-.0322307,.0346834){14}{\line(0,1){.0346834}}
\multiput(191.175,56.979)(-.033978,.0329735){14}{\line(-1,0){.033978}}
\multiput(190.7,57.441)(-.0383758,.0335739){13}{\line(-1,0){.0383758}}
\multiput(190.201,57.877)(-.0400576,.0315483){13}{\line(-1,0){.0400576}}
\multiput(189.68,58.287)(-.045102,.031892){12}{\line(-1,0){.045102}}
\multiput(189.139,58.67)(-.050932,.032205){11}{\line(-1,0){.050932}}
\multiput(188.579,59.024)(-.057779,.032487){10}{\line(-1,0){.057779}}
\multiput(188.001,59.349)(-.065977,.032735){9}{\line(-1,0){.065977}}
\multiput(187.407,59.644)(-.076026,.032946){8}{\line(-1,0){.076026}}
\multiput(186.799,59.907)(-.088714,.033118){7}{\line(-1,0){.088714}}
\multiput(186.178,60.139)(-.105356,.033244){6}{\line(-1,0){.105356}}
\multiput(185.546,60.339)(-.128318,.033314){5}{\line(-1,0){.128318}}
\multiput(184.904,60.505)(-.16233,.03331){4}{\line(-1,0){.16233}}
\put(184.255,60.638){\line(-1,0){.6553}}
\put(183.599,60.738){\line(-1,0){.6596}}
\put(182.94,60.803){\line(-1,0){.6621}}
\put(182.278,60.835){\line(-1,0){.6629}}
\put(181.615,60.832){\line(-1,0){.6618}}
\put(180.953,60.795){\line(-1,0){.659}}
\put(180.294,60.724){\line(-1,0){.6545}}
\multiput(179.64,60.619)(-.129639,-.027732){5}{\line(-1,0){.129639}}
\multiput(178.991,60.48)(-.106696,-.028655){6}{\line(-1,0){.106696}}
\multiput(178.351,60.308)(-.090064,-.029249){7}{\line(-1,0){.090064}}
\multiput(177.721,60.104)(-.07738,-.029627){8}{\line(-1,0){.07738}}
\multiput(177.102,59.867)(-.075747,-.033581){8}{\line(-1,0){.075747}}
\multiput(176.496,59.598)(-.065701,-.033286){9}{\line(-1,0){.065701}}
\multiput(175.904,59.298)(-.057505,-.032969){10}{\line(-1,0){.057505}}
\multiput(175.329,58.969)(-.050661,-.03263){11}{\line(-1,0){.050661}}
\multiput(174.772,58.61)(-.044833,-.032268){12}{\line(-1,0){.044833}}
\multiput(174.234,58.223)(-.0397921,-.0318825){13}{\line(-1,0){.0397921}}
\multiput(173.717,57.808)(-.0353725,-.0314729){14}{\line(-1,0){.0353725}}
\multiput(173.222,57.368)(-.0337008,-.0332567){14}{\line(-1,0){.0337008}}
\multiput(172.75,56.902)(-.0319394,-.0349519){14}{\line(0,-1){.0349519}}
\multiput(172.303,56.413)(-.0324075,-.0393657){13}{\line(0,-1){.0393657}}
\multiput(171.881,55.901)(-.03286,-.044401){12}{\line(0,-1){.044401}}
\multiput(171.487,55.368)(-.033299,-.050224){11}{\line(0,-1){.050224}}
\multiput(171.121,54.816)(-.033729,-.057063){10}{\line(0,-1){.057063}}
\multiput(170.783,54.245)(-.030739,-.058728){10}{\line(0,-1){.058728}}
\multiput(170.476,53.658)(-.030741,-.066929){9}{\line(0,-1){.066929}}
\multiput(170.199,53.055)(-.030651,-.07698){8}{\line(0,-1){.07698}}
\multiput(169.954,52.439)(-.030442,-.089668){7}{\line(0,-1){.089668}}
\multiput(169.741,51.812)(-.030068,-.106306){6}{\line(0,-1){.106306}}
\multiput(169.561,51.174)(-.029449,-.12926){5}{\line(0,-1){.12926}}
\put(169.413,50.528){\line(0,-1){.653}}
\put(169.3,49.875){\line(0,-1){.658}}
\put(169.22,49.217){\line(0,-1){.6613}}
\put(169.174,48.555){\line(0,-1){.6628}}
\put(169.163,47.893){\line(0,-1){.6625}}
\put(169.185,47.23){\line(0,-1){.6604}}
\put(169.242,46.57){\line(0,-1){.6566}}
\multiput(169.333,45.913)(.03115,-.16276){4}{\line(0,-1){.16276}}
\multiput(169.457,45.262)(.031608,-.128749){5}{\line(0,-1){.128749}}
\multiput(169.615,44.618)(.031843,-.105788){6}{\line(0,-1){.105788}}
\multiput(169.807,43.984)(.031938,-.089146){7}{\line(0,-1){.089146}}
\multiput(170.03,43.36)(.031935,-.076456){8}{\line(0,-1){.076456}}
\multiput(170.286,42.748)(.031856,-.066405){9}{\line(0,-1){.066405}}
\multiput(170.572,42.15)(.031717,-.058205){10}{\line(0,-1){.058205}}
\multiput(170.889,41.568)(.031527,-.051355){11}{\line(0,-1){.051355}}
\multiput(171.236,41.003)(.031291,-.045521){12}{\line(0,-1){.045521}}
\multiput(171.612,40.457)(.033599,-.043845){12}{\line(0,-1){.043845}}
\multiput(172.015,39.931)(.0330619,-.0388178){13}{\line(0,-1){.0388178}}
\multiput(172.445,39.426)(.0325199,-.0344124){14}{\line(0,-1){.0344124}}
\multiput(172.9,38.944)(.0342528,-.032688){14}{\line(1,0){.0342528}}
\multiput(173.38,38.487)(.0386554,-.0332516){13}{\line(1,0){.0386554}}
\multiput(173.882,38.055)(.0403202,-.031212){13}{\line(1,0){.0403202}}
\multiput(174.406,37.649)(.045367,-.031513){12}{\line(1,0){.045367}}
\multiput(174.951,37.271)(.0512,-.031778){11}{\line(1,0){.0512}}
\multiput(175.514,36.921)(.058049,-.032002){10}{\line(1,0){.058049}}
\multiput(176.094,36.601)(.066249,-.032181){9}{\line(1,0){.066249}}
\multiput(176.691,36.311)(.076299,-.032309){8}{\line(1,0){.076299}}
\multiput(177.301,36.053)(.088988,-.032374){7}{\line(1,0){.088988}}
\multiput(177.924,35.826)(.105631,-.032361){6}{\line(1,0){.105631}}
\multiput(178.558,35.632)(.128593,-.032239){5}{\line(1,0){.128593}}
\multiput(179.201,35.471)(.16261,-.03195){4}{\line(1,0){.16261}}
\put(179.851,35.343){\line(1,0){.6562}}
\put(180.507,35.249){\line(1,0){.6601}}
\put(181.167,35.189){\line(1,0){.6624}}
\put(181.83,35.163){\line(1,0){.6628}}
\put(182.493,35.172){\line(1,0){.6615}}
\put(183.154,35.214){\line(1,0){.6584}}
\put(183.812,35.291){\line(1,0){.6536}}
\multiput(184.466,35.401)(.129403,.028816){5}{\line(1,0){.129403}}
\multiput(185.113,35.545)(.106452,.029547){6}{\line(1,0){.106452}}
\multiput(185.752,35.723)(.089816,.030002){7}{\line(1,0){.089816}}
\multiput(186.38,35.933)(.077129,.030273){8}{\line(1,0){.077129}}
\multiput(186.997,36.175)(.067079,.030412){9}{\line(1,0){.067079}}
\multiput(187.601,36.448)(.058878,.030451){10}{\line(1,0){.058878}}
\multiput(188.19,36.753)(.057228,.033449){10}{\line(1,0){.057228}}
\multiput(188.762,37.087)(.050386,.033053){11}{\line(1,0){.050386}}
\multiput(189.316,37.451)(.044562,.032642){12}{\line(1,0){.044562}}
\multiput(189.851,37.843)(.0395239,.0322144){13}{\line(1,0){.0395239}}
\multiput(190.365,38.261)(.0351079,.0317678){14}{\line(1,0){.0351079}}
\multiput(190.857,38.706)(.0334213,.0335376){14}{\line(0,1){.0335376}}
\multiput(191.324,39.176)(.0316457,.035218){14}{\line(0,1){.035218}}
\multiput(191.767,39.669)(.032077,.0396355){13}{\line(0,1){.0396355}}
\multiput(192.184,40.184)(.032487,.044675){12}{\line(0,1){.044675}}
\multiput(192.574,40.72)(.032878,.050501){11}{\line(0,1){.050501}}
\multiput(192.936,41.276)(.03325,.057343){10}{\line(0,1){.057343}}
\multiput(193.269,41.849)(.033607,.065537){9}{\line(0,1){.065537}}
\multiput(193.571,42.439)(.030179,.067184){9}{\line(0,1){.067184}}
\multiput(193.843,43.044)(.030005,.077234){8}{\line(0,1){.077234}}
\multiput(194.083,43.661)(.02969,.08992){7}{\line(0,1){.08992}}
\multiput(194.29,44.291)(.029177,.106554){6}{\line(0,1){.106554}}
\multiput(194.466,44.93)(.028366,.129502){5}{\line(0,1){.129502}}
\put(194.607,45.578){\line(0,1){.654}}
\put(194.716,46.232){\line(0,1){.6587}}
\put(194.79,46.89){\line(0,1){1.1096}}
\put(156.552,69){\line(0,1){.5212}}
\put(156.538,69.521){\line(0,1){.5196}}
\put(156.496,70.041){\line(0,1){.5165}}
\put(156.425,70.557){\line(0,1){.5119}}
\multiput(156.326,71.069)(-.0317,.12643){4}{\line(0,1){.12643}}
\multiput(156.199,71.575)(-.030838,.09961){5}{\line(0,1){.09961}}
\multiput(156.045,72.073)(-.030189,.081483){6}{\line(0,1){.081483}}
\multiput(155.864,72.562)(-.029648,.068327){7}{\line(0,1){.068327}}
\multiput(155.656,73.04)(-.033332,.066607){7}{\line(0,1){.066607}}
\multiput(155.423,73.506)(-.032302,.056603){8}{\line(0,1){.056603}}
\multiput(155.164,73.959)(-.031415,.048673){9}{\line(0,1){.048673}}
\multiput(154.882,74.397)(-.030621,.042197){10}{\line(0,1){.042197}}
\multiput(154.575,74.819)(-.032878,.040464){10}{\line(0,1){.040464}}
\multiput(154.247,75.224)(-.031852,.0351){11}{\line(0,1){.0351}}
\multiput(153.896,75.61)(-.033719,.03331){11}{\line(-1,0){.033719}}
\multiput(153.525,75.976)(-.035486,.03142){11}{\line(-1,0){.035486}}
\multiput(153.135,76.322)(-.040863,.032381){10}{\line(-1,0){.040863}}
\multiput(152.726,76.646)(-.047298,.033449){9}{\line(-1,0){.047298}}
\multiput(152.301,76.947)(-.049053,.030818){9}{\line(-1,0){.049053}}
\multiput(151.859,77.224)(-.056994,.031608){8}{\line(-1,0){.056994}}
\multiput(151.403,77.477)(-.067009,.032516){7}{\line(-1,0){.067009}}
\multiput(150.934,77.705)(-.080131,.033613){6}{\line(-1,0){.080131}}
\multiput(150.453,77.906)(-.081846,.029191){6}{\line(-1,0){.081846}}
\multiput(149.962,78.082)(-.09998,.029619){5}{\line(-1,0){.09998}}
\multiput(149.462,78.23)(-.12681,.03015){4}{\line(-1,0){.12681}}
\put(148.955,78.35){\line(-1,0){.5131}}
\put(148.442,78.443){\line(-1,0){.5174}}
\put(147.925,78.508){\line(-1,0){.5201}}
\put(147.405,78.544){\line(-1,0){.5213}}
\put(146.883,78.552){\line(-1,0){.521}}
\put(146.362,78.531){\line(-1,0){.5191}}
\put(145.843,78.482){\line(-1,0){.5156}}
\put(145.328,78.405){\line(-1,0){.5106}}
\multiput(144.817,78.3)(-.12603,-.03324){4}{\line(-1,0){.12603}}
\multiput(144.313,78.167)(-.099226,-.032052){5}{\line(-1,0){.099226}}
\multiput(143.817,78.006)(-.081108,-.031182){6}{\line(-1,0){.081108}}
\multiput(143.33,77.819)(-.067959,-.03048){7}{\line(-1,0){.067959}}
\multiput(142.854,77.606)(-.057921,-.029875){8}{\line(-1,0){.057921}}
\multiput(142.391,77.367)(-.056205,-.032991){8}{\line(-1,0){.056205}}
\multiput(141.941,77.103)(-.048285,-.032007){9}{\line(-1,0){.048285}}
\multiput(141.507,76.815)(-.04182,-.031134){10}{\line(-1,0){.04182}}
\multiput(141.089,76.504)(-.040059,-.03337){10}{\line(-1,0){.040059}}
\multiput(140.688,76.17)(-.034708,-.032278){11}{\line(-1,0){.034708}}
\multiput(140.306,75.815)(-.032895,-.034123){11}{\line(0,-1){.034123}}
\multiput(139.944,75.44)(-.030985,-.035867){11}{\line(0,-1){.035867}}
\multiput(139.604,75.045)(-.03188,-.041255){10}{\line(0,-1){.041255}}
\multiput(139.285,74.633)(-.032868,-.047703){9}{\line(0,-1){.047703}}
\multiput(138.989,74.203)(-.030217,-.049425){9}{\line(0,-1){.049425}}
\multiput(138.717,73.758)(-.030909,-.057375){8}{\line(0,-1){.057375}}
\multiput(138.47,73.299)(-.031695,-.067402){7}{\line(0,-1){.067402}}
\multiput(138.248,72.828)(-.032632,-.080536){6}{\line(0,-1){.080536}}
\multiput(138.052,72.344)(-.028189,-.082196){6}{\line(0,-1){.082196}}
\multiput(137.883,71.851)(-.028395,-.100334){5}{\line(0,-1){.100334}}
\put(137.741,71.349){\line(0,-1){.5087}}
\put(137.627,70.841){\line(0,-1){.5142}}
\put(137.54,70.327){\line(0,-1){.5181}}
\put(137.482,69.809){\line(0,-1){2.081}}
\put(137.533,67.728){\line(0,-1){.5146}}
\put(137.616,67.213){\line(0,-1){.5093}}
\multiput(137.728,66.704)(.02782,-.100495){5}{\line(0,-1){.100495}}
\multiput(137.867,66.201)(.033262,-.098827){5}{\line(0,-1){.098827}}
\multiput(138.033,65.707)(.03217,-.080721){6}{\line(0,-1){.080721}}
\multiput(138.226,65.223)(.031308,-.067582){7}{\line(0,-1){.067582}}
\multiput(138.445,64.75)(.03058,-.057552){8}{\line(0,-1){.057552}}
\multiput(138.69,64.289)(.033675,-.055797){8}{\line(0,-1){.055797}}
\multiput(138.959,63.843)(.032594,-.047891){9}{\line(0,-1){.047891}}
\multiput(139.253,63.412)(.031643,-.041437){10}{\line(0,-1){.041437}}
\multiput(139.569,62.997)(.030779,-.036044){11}{\line(0,-1){.036044}}
\multiput(139.908,62.601)(.032699,-.034311){11}{\line(0,-1){.034311}}
\multiput(140.267,62.224)(.034523,-.032476){11}{\line(1,0){.034523}}
\multiput(140.647,61.866)(.039868,-.033599){10}{\line(1,0){.039868}}
\multiput(141.046,61.53)(.041641,-.031374){10}{\line(1,0){.041641}}
\multiput(141.462,61.217)(.048101,-.032283){9}{\line(1,0){.048101}}
\multiput(141.895,60.926)(.056015,-.033312){8}{\line(1,0){.056015}}
\multiput(142.343,60.66)(.057749,-.030206){8}{\line(1,0){.057749}}
\multiput(142.805,60.418)(.067784,-.030869){7}{\line(1,0){.067784}}
\multiput(143.28,60.202)(.080928,-.031646){6}{\line(1,0){.080928}}
\multiput(143.765,60.012)(.099041,-.03262){5}{\line(1,0){.099041}}
\multiput(144.26,59.849)(.100673,-.027168){5}{\line(1,0){.100673}}
\put(144.764,59.713){\line(1,0){.51}}
\put(145.274,59.605){\line(1,0){.5152}}
\put(145.789,59.525){\line(1,0){.5188}}
\put(146.308,59.473){\line(1,0){.5208}}
\put(146.829,59.449){\line(1,0){.5213}}
\put(147.35,59.454){\line(1,0){.5203}}
\put(147.87,59.487){\line(1,0){.5177}}
\put(148.388,59.549){\line(1,0){.5136}}
\put(148.902,59.639){\line(1,0){.5079}}
\multiput(149.409,59.756)(.100148,.029045){5}{\line(1,0){.100148}}
\multiput(149.91,59.902)(.082011,.028722){6}{\line(1,0){.082011}}
\multiput(150.402,60.074)(.080322,.033154){6}{\line(1,0){.080322}}
\multiput(150.884,60.273)(.067195,.032131){7}{\line(1,0){.067195}}
\multiput(151.355,60.498)(.057174,.031281){8}{\line(1,0){.057174}}
\multiput(151.812,60.748)(.049228,.030537){9}{\line(1,0){.049228}}
\multiput(152.255,61.023)(.047489,.033177){9}{\line(1,0){.047489}}
\multiput(152.682,61.321)(.041047,.032147){10}{\line(1,0){.041047}}
\multiput(153.093,61.643)(.035666,.031217){11}{\line(1,0){.035666}}
\multiput(153.485,61.986)(.033909,.033116){11}{\line(1,0){.033909}}
\multiput(153.858,62.351)(.032052,.034917){11}{\line(0,1){.034917}}
\multiput(154.211,62.735)(.033109,.040275){10}{\line(0,1){.040275}}
\multiput(154.542,63.137)(.030863,.042021){10}{\line(0,1){.042021}}
\multiput(154.851,63.558)(.031693,.048492){9}{\line(0,1){.048492}}
\multiput(155.136,63.994)(.032626,.056417){8}{\line(0,1){.056417}}
\multiput(155.397,64.445)(.033713,.066415){7}{\line(0,1){.066415}}
\multiput(155.633,64.91)(.030039,.068156){7}{\line(0,1){.068156}}
\multiput(155.843,65.387)(.030655,.081309){6}{\line(0,1){.081309}}
\multiput(156.027,65.875)(.031408,.099432){5}{\line(0,1){.099432}}
\multiput(156.184,66.372)(.03242,.12625){4}{\line(0,1){.12625}}
\put(156.314,66.877){\line(0,1){.5113}}
\put(156.416,67.389){\line(0,1){.5161}}
\put(156.489,67.905){\line(0,1){.5194}}
\put(156.535,68.424){\line(0,1){.5758}}
\put(86.384,64.25){\line(0,1){.79}}
\put(86.365,65.04){\line(0,1){.7881}}
\put(86.307,65.828){\line(0,1){.7843}}
\multiput(86.21,66.612)(-.03371,.19467){4}{\line(0,1){.19467}}
\multiput(86.075,67.391)(-.028798,.128521){6}{\line(0,1){.128521}}
\multiput(85.903,68.162)(-.030048,.10882){7}{\line(0,1){.10882}}
\multiput(85.692,68.924)(-.030923,.093816){8}{\line(0,1){.093816}}
\multiput(85.445,69.675)(-.031538,.081946){9}{\line(0,1){.081946}}
\multiput(85.161,70.412)(-.031961,.072273){10}{\line(0,1){.072273}}
\multiput(84.842,71.135)(-.032238,.064201){11}{\line(0,1){.064201}}
\multiput(84.487,71.841)(-.032398,.057334){12}{\line(0,1){.057334}}
\multiput(84.098,72.529)(-.032461,.0513957){13}{\line(0,1){.0513957}}
\multiput(83.676,73.197)(-.032443,.0461914){14}{\line(0,1){.0461914}}
\multiput(83.222,73.844)(-.0323548,.0415776){15}{\line(0,1){.0415776}}
\multiput(82.737,74.467)(-.0322049,.037447){16}{\line(0,1){.037447}}
\multiput(82.221,75.067)(-.0319998,.0337178){17}{\line(0,1){.0337178}}
\multiput(81.677,75.64)(-.0336124,.0321105){17}{\line(-1,0){.0336124}}
\multiput(81.106,76.186)(-.0373409,.0323278){16}{\line(-1,0){.0373409}}
\multiput(80.508,76.703)(-.041471,.0324913){15}{\line(-1,0){.041471}}
\multiput(79.886,77.19)(-.0460845,.0325947){14}{\line(-1,0){.0460845}}
\multiput(79.241,77.647)(-.0512887,.0326297){13}{\line(-1,0){.0512887}}
\multiput(78.574,78.071)(-.057227,.032586){12}{\line(-1,0){.057227}}
\multiput(77.888,78.462)(-.064095,.032449){11}{\line(-1,0){.064095}}
\multiput(77.183,78.819)(-.072168,.032198){10}{\line(-1,0){.072168}}
\multiput(76.461,79.141)(-.081842,.031807){9}{\line(-1,0){.081842}}
\multiput(75.724,79.427)(-.093714,.031231){8}{\line(-1,0){.093714}}
\multiput(74.975,79.677)(-.108721,.030406){7}{\line(-1,0){.108721}}
\multiput(74.214,79.89)(-.128426,.02922){6}{\line(-1,0){.128426}}
\multiput(73.443,80.065)(-.155643,.027476){5}{\line(-1,0){.155643}}
\put(72.665,80.202){\line(-1,0){.784}}
\put(71.881,80.302){\line(-1,0){.7879}}
\put(71.093,80.362){\line(-1,0){1.58}}
\put(69.513,80.367){\line(-1,0){.7883}}
\put(68.725,80.312){\line(-1,0){.7846}}
\multiput(67.94,80.218)(-.19478,-.03307){4}{\line(-1,0){.19478}}
\multiput(67.161,80.086)(-.128615,-.028375){6}{\line(-1,0){.128615}}
\multiput(66.389,79.915)(-.108919,-.029691){7}{\line(-1,0){.108919}}
\multiput(65.627,79.708)(-.093917,-.030615){8}{\line(-1,0){.093917}}
\multiput(64.875,79.463)(-.082049,-.031268){9}{\line(-1,0){.082049}}
\multiput(64.137,79.181)(-.072378,-.031723){10}{\line(-1,0){.072378}}
\multiput(63.413,78.864)(-.064307,-.032027){11}{\line(-1,0){.064307}}
\multiput(62.706,78.512)(-.05744,-.032209){12}{\line(-1,0){.05744}}
\multiput(62.017,78.125)(-.0515021,-.0322919){13}{\line(-1,0){.0515021}}
\multiput(61.347,77.705)(-.0462978,-.032291){14}{\line(-1,0){.0462978}}
\multiput(60.699,77.253)(-.0416837,-.032218){15}{\line(-1,0){.0416837}}
\multiput(60.074,76.77)(-.0375527,-.0320816){16}{\line(-1,0){.0375527}}
\multiput(59.473,76.257)(-.0338228,-.0318888){17}{\line(-1,0){.0338228}}
\multiput(58.898,75.715)(-.0322208,-.0335067){17}{\line(0,-1){.0335067}}
\multiput(58.35,75.145)(-.0324503,-.0372345){16}{\line(0,-1){.0372345}}
\multiput(57.831,74.549)(-.0326274,-.041364){15}{\line(0,-1){.041364}}
\multiput(57.341,73.929)(-.032746,-.0459771){14}{\line(0,-1){.0459771}}
\multiput(56.883,73.285)(-.0327981,-.0511812){13}{\line(0,-1){.0511812}}
\multiput(56.457,72.62)(-.032774,-.05712){12}{\line(0,-1){.05712}}
\multiput(56.063,71.934)(-.032659,-.063988){11}{\line(0,-1){.063988}}
\multiput(55.704,71.231)(-.032436,-.072062){10}{\line(0,-1){.072062}}
\multiput(55.38,70.51)(-.032076,-.081737){9}{\line(0,-1){.081737}}
\multiput(55.091,69.774)(-.031539,-.093611){8}{\line(0,-1){.093611}}
\multiput(54.839,69.025)(-.030763,-.10862){7}{\line(0,-1){.10862}}
\multiput(54.623,68.265)(-.029642,-.128329){6}{\line(0,-1){.128329}}
\multiput(54.446,67.495)(-.027988,-.155552){5}{\line(0,-1){.155552}}
\put(54.306,66.717){\line(0,-1){.7837}}
\put(54.204,65.934){\line(0,-1){.7877}}
\put(54.141,65.146){\line(0,-1){.7899}}
\put(54.116,64.356){\line(0,-1){.7901}}
\put(54.13,63.566){\line(0,-1){.7885}}
\put(54.183,62.777){\line(0,-1){.7849}}
\multiput(54.275,61.993)(.03243,-.19488){4}{\line(0,-1){.19488}}
\multiput(54.404,61.213)(.033543,-.154449){5}{\line(0,-1){.154449}}
\multiput(54.572,60.441)(.029332,-.109016){7}{\line(0,-1){.109016}}
\multiput(54.777,59.678)(.030306,-.094017){8}{\line(0,-1){.094017}}
\multiput(55.02,58.925)(.030998,-.082152){9}{\line(0,-1){.082152}}
\multiput(55.299,58.186)(.031485,-.072482){10}{\line(0,-1){.072482}}
\multiput(55.614,57.461)(.031815,-.064412){11}{\line(0,-1){.064412}}
\multiput(55.964,56.753)(.03202,-.057546){12}{\line(0,-1){.057546}}
\multiput(56.348,56.062)(.0321224,-.051608){13}{\line(0,-1){.051608}}
\multiput(56.765,55.391)(.0321387,-.0464037){14}{\line(0,-1){.0464037}}
\multiput(57.215,54.742)(.0320808,-.0417894){15}{\line(0,-1){.0417894}}
\multiput(57.696,54.115)(.031958,-.0376579){16}{\line(0,-1){.0376579}}
\multiput(58.208,53.512)(.0317775,-.0339274){17}{\line(0,-1){.0339274}}
\multiput(58.748,52.936)(.0334006,-.0323308){17}{\line(1,0){.0334006}}
\multiput(59.316,52.386)(.0371276,-.0325725){16}{\line(1,0){.0371276}}
\multiput(59.91,51.865)(.0412565,-.0327632){15}{\line(1,0){.0412565}}
\multiput(60.529,51.373)(.0458693,-.0328969){14}{\line(1,0){.0458693}}
\multiput(61.171,50.913)(.0510731,-.0329662){13}{\line(1,0){.0510731}}
\multiput(61.835,50.484)(.057012,-.032961){12}{\line(1,0){.057012}}
\multiput(62.519,50.089)(.06388,-.032869){11}{\line(1,0){.06388}}
\multiput(63.222,49.727)(.071955,-.032672){10}{\line(1,0){.071955}}
\multiput(63.941,49.4)(.081631,-.032344){9}{\line(1,0){.081631}}
\multiput(64.676,49.109)(.093507,-.031847){8}{\line(1,0){.093507}}
\multiput(65.424,48.854)(.108519,-.03112){7}{\line(1,0){.108519}}
\multiput(66.184,48.637)(.128231,-.030064){6}{\line(1,0){.128231}}
\multiput(66.953,48.456)(.155459,-.028499){5}{\line(1,0){.155459}}
\put(67.73,48.314){\line(1,0){.7833}}
\put(68.514,48.21){\line(1,0){.7875}}
\put(69.301,48.144){\line(1,0){.7898}}
\put(70.091,48.117){\line(1,0){.7902}}
\put(70.881,48.128){\line(1,0){.7887}}
\put(71.67,48.178){\line(1,0){.7852}}
\multiput(72.455,48.267)(.19499,.03178){4}{\line(1,0){.19499}}
\multiput(73.235,48.394)(.154559,.033035){5}{\line(1,0){.154559}}
\multiput(74.008,48.559)(.109111,.028974){7}{\line(1,0){.109111}}
\multiput(74.771,48.762)(.094117,.029997){8}{\line(1,0){.094117}}
\multiput(75.524,49.002)(.082253,.030728){9}{\line(1,0){.082253}}
\multiput(76.265,49.279)(.072585,.031247){10}{\line(1,0){.072585}}
\multiput(76.991,49.591)(.064516,.031603){11}{\line(1,0){.064516}}
\multiput(77.7,49.939)(.05765,.031831){12}{\line(1,0){.05765}}
\multiput(78.392,50.321)(.0517133,.0319526){13}{\line(1,0){.0517133}}
\multiput(79.064,50.736)(.0465091,.031986){14}{\line(1,0){.0465091}}
\multiput(79.715,51.184)(.0418946,.0319433){15}{\line(1,0){.0418946}}
\multiput(80.344,51.663)(.0377627,.031834){16}{\line(1,0){.0377627}}
\multiput(80.948,52.173)(.0361587,.0336449){16}{\line(1,0){.0361587}}
\multiput(81.527,52.711)(.0324404,.0332942){17}{\line(0,1){.0332942}}
\multiput(82.078,53.277)(.0326944,.0370204){16}{\line(0,1){.0370204}}
\multiput(82.601,53.869)(.0328986,.0411486){15}{\line(0,1){.0411486}}
\multiput(83.095,54.486)(.0330475,.0457609){14}{\line(0,1){.0457609}}
\multiput(83.557,55.127)(.0331339,.0509645){13}{\line(0,1){.0509645}}
\multiput(83.988,55.79)(.033149,.056903){12}{\line(0,1){.056903}}
\multiput(84.386,56.472)(.033079,.063772){11}{\line(0,1){.063772}}
\multiput(84.75,57.174)(.032909,.071847){10}{\line(0,1){.071847}}
\multiput(85.079,57.892)(.032612,.081525){9}{\line(0,1){.081525}}
\multiput(85.372,58.626)(.032154,.093402){8}{\line(0,1){.093402}}
\multiput(85.63,59.373)(.031477,.108416){7}{\line(0,1){.108416}}
\multiput(85.85,60.132)(.030485,.128132){6}{\line(0,1){.128132}}
\multiput(86.033,60.901)(.02901,.155365){5}{\line(0,1){.155365}}
\put(86.178,61.678){\line(0,1){.783}}
\put(86.285,62.461){\line(0,1){.7873}}
\put(86.353,63.248){\line(0,1){1.0018}}
\put(99.75,48.5){\circle*{1.581}}
\put(105.5,48.75){\circle*{1.581}}
\put(111,49.25){\circle*{1.581}}
\put(145,24){\makebox(0,0)[cc]{$s_1$}}
\put(133,38.75){\makebox(0,0)[cc]{$s_2$}}
\put(89.25,32.5){\makebox(0,0)[cc]{$s_m$}}
\put(197.75,47){\makebox(0,0)[cc]{$r_1$}}
\put(156.5,78.25){\makebox(0,0)[cc]{$r_2$}}
\put(69.5,84.25){\makebox(0,0)[cc]{$r_m$}}
\end{picture}
\caption{Using lollipops to build a van Kampen diagram} \label{f1}
\end{figure}

\end{center}

Going counterclockwise around the ``lollipop" starting and ending at the tip of the stem, we read $s_ir_is_i\iv$.
Thus going counterclockwise around the diagram which is the bouquet of ``lollipops", we read the word which is the right hand side of (\ref{e1}).

In order to make the word $W$ from this word, we need to {\em reduce} the boundary of the bouquet of ``lollipops" (the boundary is traced counterclockwise): every time we see a pair of consecutive edges on the boundary of the diagram, which have the same label and the same initial or terminal vertex (see Figure \ref{f8}), we identify these two edges (if the edges have both vertices in common, we identify the two edges and remove the whole subgraph bounded by them on the plane). This amounts to removing a subwords $xx\iv$ and $x\iv x$ from the right hand side of (\ref{e1}). The resulting picture is a {\em \index{Van Kampen diagram} van Kampen diagram for} $W$ over the presentation $\la X\mid R\ra$, that is a planar graph with edges labeled by elements of $X$,  the boundary of each {\em \index{Van Kampen diagram!cell of}cell} (i.e. the closure of a bounded connected components of the plane minus the graph) labeled by words from $R^{\pm 1}$, and the \index{Van Kampen diagram!boundary of}boundary of the whole graph (i.e. the boundary of the infinite component of the plane minus the graph) is labeled by $W$ (see \cite{LS} for a slightly different definition and \cite{Olbook} for another one).

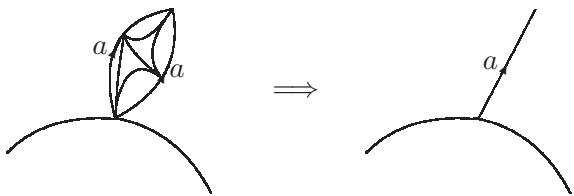
\begin{figure}[H]
\unitlength .5mm 
\linethickness{0.4pt}
\ifx\plotpoint\undefined\newsavebox{\plotpoint}\fi 
\begin{picture}(159,54)(0,0)
\qbezier(9.25,15.75)(19.375,26.25)(38,24.75)
\qbezier(104.75,15.75)(114.875,26.25)(133.5,24.75)
\qbezier(38,24.75)(31.375,50.125)(53.25,54)
\qbezier(53.25,54)(56.875,35)(38,25)
\qbezier(38,25)(54.75,22)(63.5,5)
\qbezier(133.5,25)(150.25,22)(159,5)
\put(38.25,43.625){\vector(1,2){.07}}\multiput(37.25,41.5)(.03333333,.07083333){60}{\line(0,1){.07083333}}
\put(51.25,37){\vector(1,2){.07}}\multiput(50.25,35.25)(.03333333,.05833333){60}{\line(0,1){.05833333}}
\put(34,43.75){\makebox(0,0)[cc]{$a$}}
\put(54.5,37.5){\makebox(0,0)[cc]{$a$}}
\qbezier(40.25,47.25)(45.375,42.5)(53,53.75)
\qbezier(53,53.75)(44.25,42.375)(50.5,35.5)
\qbezier(50.5,35.5)(44.375,40.625)(40.75,35.25)
\qbezier(40.75,35.25)(38.125,31)(38,24.75)
\qbezier(40.25,46.75)(38.375,35.625)(38,25)
\qbezier(40.5,46.75)(41.625,42.75)(50.25,35.75)
\put(142.125,39.375){\vector(1,2){.07}}\multiput(134.5,24.75)(.0337389381,.0647123894){452}{\line(0,1){.0647123894}}
\put(137.75,39.75){\makebox(0,0)[cc]{$a$}}
\put(86,32){\makebox(0,0)[cc]{$\Longrightarrow$}}
\end{picture}
\caption{Reducing the boundary.}\label{f8}
\end{figure}

Figure \ref{f2} shows how a typical van Kampen diagram may look like. One can see that the cells may be of different sizes and shapes, a cell can touch itself, etc. It is important also that a diagram itself may not be an embedded disc: several pieces as on Figure \ref{f2} can be connected by paths to form a tree of discs (as the bouquet of lollipops in Fig \ref{f1}). Nevertheless it is a planar graph, and there are various general methods to study van Kampen diagrams (the whole book \cite{Olbook} is essentially about the study of van Kampen diagrams used to construct groups with extreme properties such as infinite bounded torsion groups, Tarski monsters, etc.).

\begin{figure}[H]

\unitlength .6mm 
\linethickness{0.4pt}
\ifx\plotpoint\undefined\newsavebox{\plotpoint}\fi 

\begin{center}
\begin{picture}(143.25,66)(0,0)
\qbezier(7.25,40.25)(36.375,66)(71,63.75)
\qbezier(71,63.75)(130.5,61)(136,45.25)
\qbezier(136,45.25)(143.25,22.75)(85.5,12.25)
\qbezier(85.5,12.25)(15.75,1.25)(12,18.25)
\qbezier(12,18.25)(9.625,29.75)(10.75,28.25)
\qbezier(10.75,28.25)(-3,31.25)(7.25,40.25)
\qbezier(11.75,44)(60.125,43)(11,28)
\qbezier(35.75,39.5)(74.5,41.75)(66.25,64)
\qbezier(66.25,64)(125.25,35.875)(65.25,22.25)
\qbezier(65.25,22.25)(94.125,46.875)(58.5,44)
\qbezier(58.5,44)(36.25,34.125)(44,27.75)
\qbezier(44,27.75)(70.5,40)(62,36.25)
\qbezier(62,36.25)(35.875,23.5)(44.25,27.75)
\qbezier(44.25,27.75)(18.625,15.5)(65.5,22.25)
\qbezier(60.75,21.5)(63.25,13.125)(71.75,10.25)
\qbezier(71.75,10.25)(111.875,17.5)(103.5,27.75)
\qbezier(103.5,27.75)(98.25,34)(93,34.25)
\qbezier(93,34.25)(129.75,40.75)(133.5,49.25)
\end{picture}
\end{center}
\caption{A van Kampen diagram without labels.} \label{f2}
\end{figure}
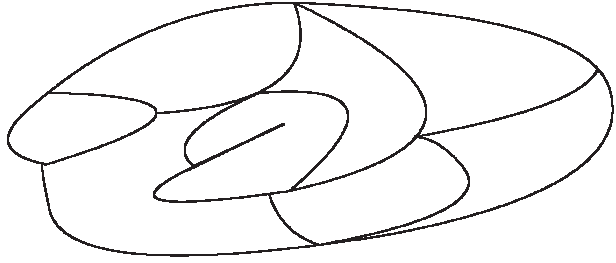

Thus we have shown

\begin{lemma}[The first half of the van Kampen lemma]\label{lvk} If a reduced group word $W$ over the alphabet $X$  is equal to 1 in $G$, then there exists a van Kampen diagram over the presentation of $G$ with boundary label $W$.
\end{lemma}
This is a half of the so called {\em \index{Van Kampen lemma}van Kampen lemma} \cite{LS, Olbook, Br}. The converse statement (that the boundary label of every van Kampen diagram is equal to 1 in $G$) constitutes the second half. In order to prove it, we have to ``undo" the construction of van Kampen diagram above, and, given a van Kampen diagram, produce an equality of the form (\ref{e1}). The proof is by cutting a van Kampen diagram along the edges, to produce a tree of ``lollipops". It is obvious that this can be done somehow. In \cite{OSsur}, we presented a useful and economical way to cut a van Kampen diagram (so that every edge is used at most four times). Since cutting van Kampen diagrams into nice pieces is the main ingredient of many proofs involving van Kampen diagrams (see Section \ref{pnrdf} below, for example), we reproduce the proof here.

Note that if a word $W$ can be represented in the free group in the form $$u_1r_1u_2\ldots u_mr_mu_{m+1}$$ where $u_1u_2\ldots u_{m+1}=1$ in the free group, then $W$ is equal (in the free group) to $$\begin{array}{l} (u_1r_1u_1\iv)(u_1u_2 r_2 u_2\iv u_1\iv) \ldots  (u_1\ldots u_mr_mu_m\iv\ldots u_1\iv) u_1\ldots u_mu_{m+1}\\ =(u_1r_1u_1\iv)(u_1u_2 r_2 u_2\iv u_1\iv) \ldots  (u_1\ldots u_mr_mu_m\iv\ldots u_1\iv)\end{array},$$ and so $W\in N$. The converse statement is obvious.

\begin{prop}[The second half of the van Kampen lemma]\label{prop1} Let $\Delta$ be a \vk diagram over a presentation $\langle X\
|\ R\rangle$ where $X=X\iv$, $R$ is closed under cyclic shifts and inverses.
Let $W$ be the boundary label of $\Delta$.  Then $W$ is equal in the free
group to a word of the form $u_1r_1u_2r_2\ldots u_mr_du_{m+1}$ where:
\begin{enumerate} \item Each $r_i$ is a cyclic shift of a word from $R^{\pm 1}$;
\item $u_1u_2\ldots u_{m+1}=1$ in the free group; \item
$\sum_{i=1}^{m+1}|u_i|\le 4e$ where $e$ is the number of edges of $\Delta$.
\end{enumerate} \label{lmm1}
In particular, $W$ is in the normal subgroup $N$ and is equal to $1$ in $G$.
\end{prop}

\proof If $\Delta$ has an internal edge (i.e. an edge which belongs to
the boundaries of two cells) then it has an internal edge $f$ one of whose
vertices belongs to the boundary. Let us  cut $\Delta$ along $f$ leaving the
second vertex of $f$ untouched. We can repeat this operation until we get a
diagram $\Delta_1$ which does not have internal edges. It is easy to see
that the boundary label of $\Delta_1$ is equal to $W$ in the free group. The
number of edges of $\Delta_1$ which do not belong to contours of cells (let us call
them {\em edges of type 1}) is the
same as the number of such edges in $\Delta$ and the
number of edges which belong
to contours of cells in $\Delta_1$ ({\em edges of type 2})
is at most twice the number of such edges of
$\Delta$ (we cut each edge from a contour of a cell
at most once, after the cut we get two external
edges instead of one internal edge).

Suppose that a cell $\Pi$ in $\Delta_1$ has more
than one edge which has a
common vertex with $\Pi$ but does not belong to the contour of $\Pi$.

\begin{center}
\begin{figure}[H]
\unitlength 1mm 
\linethickness{0.4pt}
\ifx\plotpoint\undefined\newsavebox{\plotpoint}\fi 
\begin{picture}(73.75,56.25)(0,0)
\put(53.14,25.5){\line(0,1){.6949}}
\put(53.123,26.195){\line(0,1){.6931}}
\put(53.07,26.888){\line(0,1){.6895}}
\multiput(52.981,27.578)(-.03084,.17103){4}{\line(0,1){.17103}}
\multiput(52.858,28.262)(-.031607,.135386){5}{\line(0,1){.135386}}
\multiput(52.7,28.939)(-.032052,.111333){6}{\line(0,1){.111333}}
\multiput(52.508,29.607)(-.032299,.093905){7}{\line(0,1){.093905}}
\multiput(52.282,30.264)(-.032411,.08062){8}{\line(0,1){.08062}}
\multiput(52.022,30.909)(-.032423,.070102){9}{\line(0,1){.070102}}
\multiput(51.73,31.54)(-.032357,.061523){10}{\line(0,1){.061523}}
\multiput(51.407,32.155)(-.032227,.054359){11}{\line(0,1){.054359}}
\multiput(51.052,32.753)(-.032041,.048259){12}{\line(0,1){.048259}}
\multiput(50.668,33.332)(-.0318076,.0429824){13}{\line(0,1){.0429824}}
\multiput(50.254,33.891)(-.0315306,.0383557){14}{\line(0,1){.0383557}}
\multiput(49.813,34.428)(-.0334436,.0366996){14}{\line(0,1){.0366996}}
\multiput(49.345,34.942)(-.0329185,.0326183){15}{\line(-1,0){.0329185}}
\multiput(48.851,35.431)(-.0370045,.033106){14}{\line(-1,0){.0370045}}
\multiput(48.333,35.894)(-.0416155,.0335761){13}{\line(-1,0){.0416155}}
\multiput(47.792,36.331)(-.043272,.0314125){13}{\line(-1,0){.043272}}
\multiput(47.229,36.739)(-.048551,.031598){12}{\line(-1,0){.048551}}
\multiput(46.647,37.118)(-.054652,.031727){11}{\line(-1,0){.054652}}
\multiput(46.046,37.467)(-.061817,.031792){10}{\line(-1,0){.061817}}
\multiput(45.427,37.785)(-.070396,.03178){9}{\line(-1,0){.070396}}
\multiput(44.794,38.071)(-.080914,.031671){8}{\line(-1,0){.080914}}
\multiput(44.147,38.325)(-.094197,.031438){7}{\line(-1,0){.094197}}
\multiput(43.487,38.545)(-.111622,.031031){6}{\line(-1,0){.111622}}
\multiput(42.817,38.731)(-.13567,.030365){5}{\line(-1,0){.13567}}
\put(42.139,38.883){\line(-1,0){.6852}}
\put(41.454,39){\line(-1,0){.6903}}
\put(40.764,39.082){\line(-1,0){.6936}}
\put(40.07,39.129){\line(-1,0){.695}}
\put(39.375,39.14){\line(-1,0){.6947}}
\put(38.68,39.116){\line(-1,0){.6926}}
\put(37.988,39.056){\line(-1,0){.6887}}
\multiput(37.299,38.962)(-.17074,-.0324){4}{\line(-1,0){.17074}}
\multiput(36.616,38.832)(-.13509,-.032846){5}{\line(-1,0){.13509}}
\multiput(35.941,38.668)(-.111035,-.033071){6}{\line(-1,0){.111035}}
\multiput(35.274,38.469)(-.093605,-.033158){7}{\line(-1,0){.093605}}
\multiput(34.619,38.237)(-.08032,-.033148){8}{\line(-1,0){.08032}}
\multiput(33.977,37.972)(-.069802,-.033064){9}{\line(-1,0){.069802}}
\multiput(33.348,37.675)(-.061224,-.032919){10}{\line(-1,0){.061224}}
\multiput(32.736,37.345)(-.054061,-.032723){11}{\line(-1,0){.054061}}
\multiput(32.142,36.985)(-.047964,-.032482){12}{\line(-1,0){.047964}}
\multiput(31.566,36.596)(-.0426892,-.0322001){13}{\line(-1,0){.0426892}}
\multiput(31.011,36.177)(-.0380652,-.0318807){14}{\line(-1,0){.0380652}}
\multiput(30.478,35.731)(-.0339656,-.0315266){15}{\line(-1,0){.0339656}}
\multiput(29.969,35.258)(-.0323153,-.033216){15}{\line(0,-1){.033216}}
\multiput(29.484,34.76)(-.0327656,-.0373062){14}{\line(0,-1){.0373062}}
\multiput(29.025,34.237)(-.0331935,-.0419214){13}{\line(0,-1){.0419214}}
\multiput(28.594,33.692)(-.033599,-.047188){12}{\line(0,-1){.047188}}
\multiput(28.19,33.126)(-.031152,-.048838){12}{\line(0,-1){.048838}}
\multiput(27.817,32.54)(-.031225,-.05494){11}{\line(0,-1){.05494}}
\multiput(27.473,31.936)(-.031224,-.062106){10}{\line(0,-1){.062106}}
\multiput(27.161,31.315)(-.031133,-.070684){9}{\line(0,-1){.070684}}
\multiput(26.881,30.678)(-.030928,-.081201){8}{\line(0,-1){.081201}}
\multiput(26.633,30.029)(-.030573,-.094481){7}{\line(0,-1){.094481}}
\multiput(26.419,29.368)(-.030007,-.111902){6}{\line(0,-1){.111902}}
\multiput(26.239,28.696)(-.029121,-.135942){5}{\line(0,-1){.135942}}
\put(26.094,28.016){\line(0,-1){.6862}}
\put(25.983,27.33){\line(0,-1){.691}}
\put(25.907,26.639){\line(0,-1){1.3891}}
\put(25.862,25.25){\line(0,-1){.6945}}
\put(25.892,24.556){\line(0,-1){.692}}
\put(25.958,23.864){\line(0,-1){.6878}}
\multiput(26.059,23.176)(.027171,-.136345){5}{\line(0,-1){.136345}}
\multiput(26.195,22.494)(.028402,-.11232){6}{\line(0,-1){.11232}}
\multiput(26.365,21.82)(.029217,-.094909){7}{\line(0,-1){.094909}}
\multiput(26.57,21.156)(.029763,-.081635){8}{\line(0,-1){.081635}}
\multiput(26.808,20.503)(.030118,-.071123){9}{\line(0,-1){.071123}}
\multiput(27.079,19.863)(.033702,-.069496){9}{\line(0,-1){.069496}}
\multiput(27.382,19.237)(.033479,-.06092){10}{\line(0,-1){.06092}}
\multiput(27.717,18.628)(.033217,-.053759){11}{\line(0,-1){.053759}}
\multiput(28.082,18.037)(.03292,-.047664){12}{\line(0,-1){.047664}}
\multiput(28.477,17.465)(.0325898,-.0423924){13}{\line(0,-1){.0423924}}
\multiput(28.901,16.914)(.0322281,-.0377716){14}{\line(0,-1){.0377716}}
\multiput(29.352,16.385)(.0318364,-.0336753){15}{\line(0,-1){.0336753}}
\multiput(29.83,15.88)(.0335107,-.0320097){15}{\line(1,0){.0335107}}
\multiput(30.333,15.399)(.0376049,-.0324224){14}{\line(1,0){.0376049}}
\multiput(30.859,14.946)(.0422237,-.032808){13}{\line(1,0){.0422237}}
\multiput(31.408,14.519)(.047494,-.033166){12}{\line(1,0){.047494}}
\multiput(31.978,14.121)(.053587,-.033494){11}{\line(1,0){.053587}}
\multiput(32.567,13.753)(.055224,-.030721){11}{\line(1,0){.055224}}
\multiput(33.175,13.415)(.062389,-.030654){10}{\line(1,0){.062389}}
\multiput(33.799,13.108)(.070966,-.030484){9}{\line(1,0){.070966}}
\multiput(34.437,12.834)(.081481,-.030183){8}{\line(1,0){.081481}}
\multiput(35.089,12.592)(.094757,-.029706){7}{\line(1,0){.094757}}
\multiput(35.753,12.384)(.112172,-.028981){6}{\line(1,0){.112172}}
\multiput(36.426,12.211)(.136203,-.027874){5}{\line(1,0){.136203}}
\put(37.107,12.071){\line(1,0){.6872}}
\put(37.794,11.967){\line(1,0){.6917}}
\put(38.485,11.897){\line(1,0){.6943}}
\put(39.18,11.863){\line(1,0){.6951}}
\put(39.875,11.865){\line(1,0){.6942}}
\put(40.569,11.901){\line(1,0){.6914}}
\put(41.26,11.974){\line(1,0){.6868}}
\multiput(41.947,12.081)(.13609,.028419){5}{\line(1,0){.13609}}
\multiput(42.628,12.223)(.112055,.02943){6}{\line(1,0){.112055}}
\multiput(43.3,12.4)(.094637,.030086){7}{\line(1,0){.094637}}
\multiput(43.962,12.61)(.081359,.030509){8}{\line(1,0){.081359}}
\multiput(44.613,12.854)(.070844,.030768){9}{\line(1,0){.070844}}
\multiput(45.251,13.131)(.062266,.030904){10}{\line(1,0){.062266}}
\multiput(45.874,13.44)(.0551,.030942){11}{\line(1,0){.0551}}
\multiput(46.48,13.781)(.053453,.033708){11}{\line(1,0){.053453}}
\multiput(47.068,14.151)(.04736,.033355){12}{\line(1,0){.04736}}
\multiput(47.636,14.552)(.042092,.0329768){13}{\line(1,0){.042092}}
\multiput(48.183,14.98)(.0374747,.0325728){14}{\line(1,0){.0374747}}
\multiput(48.708,15.436)(.0333822,.0321436){15}{\line(1,0){.0333822}}
\multiput(49.209,15.918)(.0317013,.0338025){15}{\line(0,1){.0338025}}
\multiput(49.684,16.425)(.0320765,.0379003){14}{\line(0,1){.0379003}}
\multiput(50.133,16.956)(.0324198,.0425226){13}{\line(0,1){.0425226}}
\multiput(50.555,17.509)(.032729,.047796){12}{\line(0,1){.047796}}
\multiput(50.947,18.082)(.033002,.053892){11}{\line(0,1){.053892}}
\multiput(51.31,18.675)(.033235,.061053){10}{\line(0,1){.061053}}
\multiput(51.643,19.286)(.033424,.06963){9}{\line(0,1){.06963}}
\multiput(51.944,19.912)(.033562,.080148){8}{\line(0,1){.080148}}
\multiput(52.212,20.554)(.033641,.093433){7}{\line(0,1){.093433}}
\multiput(52.448,21.208)(.033643,.110863){6}{\line(0,1){.110863}}
\multiput(52.649,21.873)(.033542,.134919){5}{\line(0,1){.134919}}
\multiput(52.817,22.547)(.03328,.17057){4}{\line(0,1){.17057}}
\put(52.95,23.23){\line(0,1){.6882}}
\put(53.048,23.918){\line(0,1){.6923}}
\put(53.111,24.61){\line(0,1){.8899}}
\put(10.5,25){\line(1,0){15.25}}
\put(53,25){\line(1,0){20.75}}
\put(41.75,39){\line(0,1){17.25}}
\put(37,24.25){\makebox(0,0)[cc]{$\Pi$}}
\qbezier(53,25)(44.25,42.625)(38.5,35.75)
\qbezier(38.5,35.75)(39.875,1.5)(52.75,25.25)
\put(45.5,25.5){\makebox(0,0)[cc]{$\Pi'$}}
\put(55.25,21.5){\makebox(0,0)[cc]{$O$}}
\put(53,25.25){\circle*{1.581}}
\end{picture}
\caption{Cutting $\Pi'$ inside $\Pi$.} \label{f3}
\end{figure}
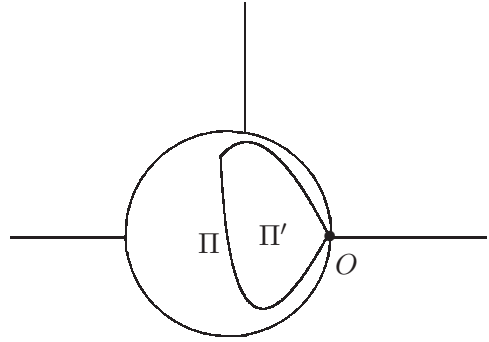
\end{center}
Take any point $O$ on\footnote{We use $\partial$ to denote boundaries of cells, diagrams, subgraphs, etc. .} $\partial(\Pi)$ which belongs to one of the edges not on $\partial(\Pi)$. Let $p$ be the boundary path of $\Delta_1$
starting at $O$ and let $q$ be the boundary path of $\Pi$ starting at $O$.
Consider the path $qq\iv p$. The subpath $q\iv p$ bounds a subdiagram of
$\Delta_1$ containing all cells but $\Pi$. Replace the path $q$ in $qq\iv p$
by a loop $q'$ with the same label starting at $O$ and lying inside the
cell $\Pi$. Let the region inside $q'$ be a new cell $\Pi'$. Then the path
$q'q\iv p$ bounds a diagram whose boundary label is equal to $W$ in the free group $F(X)$.
Notice that $\Pi'$ has exactly one edge having a common vertex with $\Pi'$ and not belonging to the contour of $\Pi'$ (see Figure \ref{f3}). Thus this operation reduces
the number of cells $\Pi$ such that more
than one edge of the diagram has a
common vertex with $\Pi$ but does not belong to the contour of $\Pi$.

After a number of such transformations
we shall have a diagram $\Delta_2$
which has the form of a tree $T$ with cells hanging like leaves (each has exactly
one common vertex with the tree).

The number of
edges of type 1 in $\Delta_2$ cannot be bigger
than the number of all edges in $\Delta_1$, so it cannot be more
than two times
bigger than the total number of edges in $\Delta$.

The boundary label of $\Delta_2$ is equal to $W$ in the free group, and it has the form
$$u_1r_{i_1}u_2r_{i_2}\ldots u_mr_{i_m}u_{m+1}$$ where $m$ is the number of cells
in $\Delta$, $u_1u_2\ldots u_{m+1}$ is the boundary label of a tree (traced counterclockwise), so
$u_1u_2\ldots u_{m+1}=1$ in the free group. The sum of lengths
of $u_i$ is
at most four times the number of edges in $\Delta$ because
the word $u_1u_2\ldots u_{m+1}$  is written on the tree $T$, and when we
travel along the tree,  we pass through each edge twice.\endproof

\subsubsection{The $0$-cells}\label{t0c}

One can modify van Kampen diagrams to make all cells embedded discs (so that the boundary of a cell does not touch itself) and to make the whole diagram an embedded disc by introducing the so-called {\em \index{$0$-cell}$0$-cells} \cite{Olbook}, i.e. cells corresponding to relations $1\cdot 1=1$ and $1\cdot a\cdot 1\cdot a\iv=1$ (see Figure \ref{f7}).

\begin{figure}[H]
\unitlength .6mm 
\linethickness{0.4pt}
\ifx\plotpoint\undefined\newsavebox{\plotpoint}\fi 
\begin{picture}(121.5,34)(0,0)
\multiput(22.25,5.75)(.0337150127,.0623409669){393}{\line(0,1){.0623409669}}
\multiput(35.5,30.25)(.0337209302,-.0569767442){430}{\line(0,-1){.0569767442}}
\put(50,5.75){\line(-1,0){27.5}}
\put(69,29.75){\line(0,-1){24.25}}
\put(69,5.5){\line(1,0){49}}
\put(118,5.5){\line(0,1){25.25}}
\put(118,30.75){\line(-1,0){49}}
\multiput(69,30.75)(.03125,-.28125){8}{\line(0,-1){.28125}}
\put(26.5,18.25){\makebox(0,0)[cc]{1}}
\put(45.25,18.75){\makebox(0,0)[cc]{1}}
\put(35.75,2.5){\makebox(0,0)[cc]{1}}
\put(65.25,18){\makebox(0,0)[cc]{1}}
\put(92.75,2.5){\makebox(0,0)[cc]{a}}
\put(92.75,34){\makebox(0,0)[cc]{a}}
\put(121.5,18.75){\makebox(0,0)[cc]{1}}
\put(94.875,30.75){\vector(1,0){.07}}\put(89.25,30.75){\line(1,0){11.25}}
\put(93.5,5.5){\vector(1,0){.07}}\put(88.75,5.5){\line(1,0){9.5}}
\end{picture}
\caption{$0$-cells} \label{f7}
\end{figure}

The edges labeled by 1 are called {\em \index{$0$-edge}$0$-edges}. Indeed, adding these relations does not change the group. Now, for every cell in the van Kampen diagram (the cell is drawn with thick lines on Figure \ref{f6}), we draw an embedded disc inside the cell, label its boundary by the same word which labels the boundary of the original cell, and connect every vertex of the boundary of the new cell with the corresponding vertex of the boundary of the original cell by a 0-edge. Thus we replace every cell of the original diagram by a new cell with the same boundary label and several 0-cells. The new van Kampen diagram will have all cells embedded discs. Similarly, if a diagram is not an embedded disc, say, it consists of  two disc subdiagrams connected by a path which is a part of the boundary of the diagram, we can replace that path by a sequence of 0-cells. As a result, we can make the whole diagram an embedded disc. When we count the number of cells (edges) in a diagram we usually do not count 0-cells (0-edges). In particular $0$-edges do not affect the calculation of the length of a path in a van Kampen diagram.

\begin{figure}[H]
\unitlength 1mm 
\linethickness{2 pt}
\ifx\plotpoint\undefined\newsavebox{\plotpoint}\fi 
\begin{picture}(85.25,77.25)(0,0)
\qbezier(63.75,52.5)(11.25,77.25)(13.75,27)
\qbezier(13.75,27)(15.25,-5.125)(33.75,6.25)
\qbezier(33.75,6.25)(85.25,42.125)(63.75,52.5)
\multiput(63.75,52.5)(-.0337136929,-.0357883817){482}{\line(0,-1){.0357883817}}
\linethickness{0.4 pt}
\qbezier(41.5,29.75)(53.75,50.75)(48,54.75)
\qbezier(48,54.75)(21.625,60.625)(18.75,35)
\qbezier(18.75,35)(17.5,24.375)(23.25,8.25)
\qbezier(23.25,8.25)(27.25,-1.5)(66.25,40.75)
\qbezier(66.25,40.75)(75,50.125)(41.75,30)

\multiput(29.5,53.25)(-.0333333,.225){30}{\line(0,1){.225}}
\multiput(39.5,55.75)(.0333333,.2833333){15}{\line(0,1){.2833333}}
\multiput(47,55)(.03333333,.07777778){45}{\line(0,1){.07777778}}
\multiput(49.75,50.5)(.033653846,.036858974){156}{\line(0,1){.036858974}}
\multiput(49.25,45.75)(.07253886,.033678756){193}{\line(1,0){.07253886}}
\put(67.25,43.75){\line(0,1){.5}}
\multiput(46.5,40.5)(.4333333,.0333333){15}{\line(1,0){.4333333}}
\multiput(55.25,37.75)(-.03358209,.05223881){67}{\line(0,1){.05223881}}
\multiput(62.5,41.75)(.0333333,.3416667){30}{\line(0,1){.3416667}}
\multiput(67.75,43)(.03333333,.06666667){60}{\line(0,1){.06666667}}
\multiput(58.5,32)(.65625,-.03125){8}{\line(1,0){.65625}}
\multiput(45.25,36)(.1833333,-.0333333){15}{\line(1,0){.1833333}}
\multiput(47.75,35.25)(.0333333,-.1166667){15}{\line(0,-1){.1166667}}
\put(42.25,30.25){\line(1,1){5.5}}
\thicklines
\multiput(49.5,37.25)(.0336879433,.036643026){423}{\line(0,1){.036643026}}
\end{picture}
\caption{Inscribing an embedded cell inside a non-embedded cell of a van Kampen diagram.}\label{f6}
\end{figure}
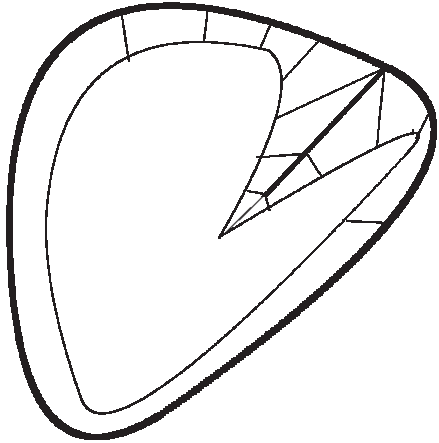

The $0$-cells were introduced by Olshanskii \cite{Olbook} in order to make his proofs cleaner, but they proved to be a very useful technical tool. Let us mention just two applications. First, 0-cells were  used in solving quadratic equations in free groups in \cite{Olsurf} (see \ref{ss:tcosqe}) and subsequent papers (there are no non-0-relations in the standard presentation of the free group). Second, a natural generalization of 0-cells was a crucial tool in constructing a finitely presented non-amenable group  without free subgroups \cite{OSamenab} (see \ref{s:afpnagwfs}).

\subsubsection{Van Kampen diagrams and tilings. An elementary school problem and its non-elementary
solution}\label{vkat}

As an easy application of van Kampen diagrams, consider the following elementary problem.

\begin{example}\label{ex1}
Let $P$ be the standard $8\times 8$ chess board with two opposite squares removed (Figure \ref{f4}). Prove that $P$ cannot be tiled by the standard $2\times 1$ dominos. \end{example}

\begin{center}

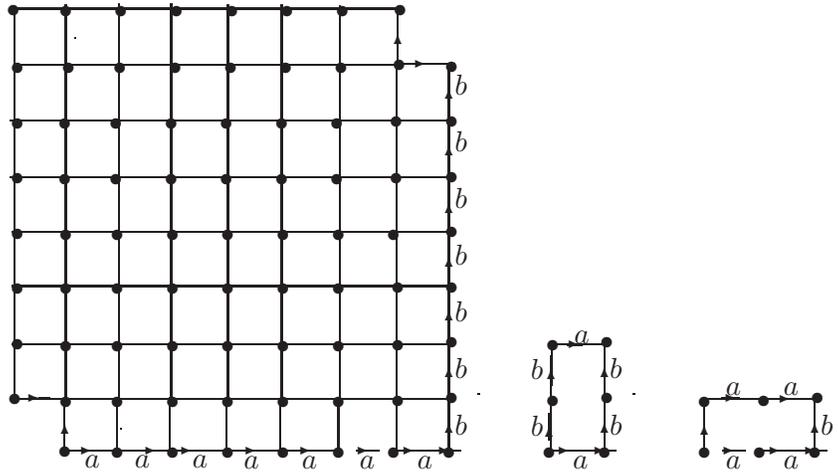
\begin{figure}[H]

\unitlength .7mm 
\linethickness{0.4pt}
\ifx\plotpoint\undefined\newsavebox{\plotpoint}\fi 
\begin{picture}(162.5,104.75)(0,0)
\put(25,14){\line(0,1){.25}}
\put(25,23.75){\line(0,1){.25}}
\put(16.25,98){\line(0,1){.25}}
\put(93,30.5){\line(-1,0){.25}}
\put(122.5,30.5){\line(-1,0){.25}}
\put(162.5,30.5){\line(-1,0){.25}}
\put(14,19.75){\line(1,0){10.75}}
\put(14,29.5){\line(1,0){10.75}}
\put(14,40){\line(1,0){10.75}}
\put(14,51){\line(1,0){10.75}}
\put(14,61.25){\line(1,0){10.75}}
\put(13.75,71.75){\line(1,0){10.75}}
\put(13.75,82.5){\line(1,0){10.75}}
\put(14.5,93){\line(1,0){10.75}}
\put(4.25,29.5){\line(1,0){10.75}}
\put(4.25,40){\line(1,0){10.75}}
\put(4.25,51){\line(1,0){10.75}}
\put(4.25,61.25){\line(1,0){10.75}}
\put(4,71.75){\line(1,0){10.75}}
\put(4,82.5){\line(1,0){10.75}}
\put(4.75,93){\line(1,0){10.75}}
\put(4.25,103.75){\line(1,0){10.75}}
\put(77.5,93.25){\line(1,0){10.75}}
\put(87.25,19.5){\line(0,1){10.75}}
\put(116.75,19.5){\line(0,1){10.75}}
\put(156.75,19.5){\line(0,1){10.75}}
\put(4.75,29.75){\line(0,1){10.75}}
\put(77.5,93.25){\line(0,1){10.75}}
\put(14.25,18.75){\line(0,1){10.75}}
\put(24.25,19){\line(0,1){10.75}}
\put(34.25,19){\line(0,1){10.75}}
\put(45,18.75){\line(0,1){10.75}}
\put(55.25,18.75){\line(0,1){10.75}}
\put(66.25,18.5){\line(0,1){10.75}}
\put(135.75,18.5){\line(0,1){10.75}}
\put(24.5,19.75){\line(1,0){10.75}}
\put(24.5,29.5){\line(1,0){10.75}}
\put(24.5,40){\line(1,0){10.75}}
\put(24.5,51){\line(1,0){10.75}}
\put(24.5,61.25){\line(1,0){10.75}}
\put(24.25,71.75){\line(1,0){10.75}}
\put(24.25,82.5){\line(1,0){10.75}}
\put(25,93){\line(1,0){10.75}}
\put(14.75,103.75){\line(1,0){10.75}}
\put(87.25,30){\line(0,1){10.75}}
\put(116.75,30){\line(0,1){10.75}}
\put(4.75,40.25){\line(0,1){10.75}}
\put(35,19.75){\line(1,0){10.75}}
\put(35,29.5){\line(1,0){10.75}}
\put(35,40){\line(1,0){10.75}}
\put(35,51){\line(1,0){10.75}}
\put(35,61.25){\line(1,0){10.75}}
\put(34.75,71.75){\line(1,0){10.75}}
\put(34.75,82.5){\line(1,0){10.75}}
\put(35.5,93){\line(1,0){10.75}}
\put(25.25,103.75){\line(1,0){10.75}}
\put(87.25,40.5){\line(0,1){10.75}}
\put(4.75,50.75){\line(0,1){10.75}}
\put(45.5,19.75){\line(1,0){10.75}}
\put(45.5,29.5){\line(1,0){10.75}}
\put(45.5,40){\line(1,0){10.75}}
\put(45.5,51){\line(1,0){10.75}}
\put(45.5,61.25){\line(1,0){10.75}}
\put(45.25,71.75){\line(1,0){10.75}}
\put(45.25,82.5){\line(1,0){10.75}}
\put(46,93){\line(1,0){10.75}}
\put(35.75,103.75){\line(1,0){10.75}}
\put(87.25,51){\line(0,1){10.75}}
\put(4.75,61.25){\line(0,1){10.75}}
\put(56,19.75){\line(1,0){10.75}}
\put(56,29.5){\line(1,0){10.75}}
\put(56,40){\line(1,0){10.75}}
\put(56,51){\line(1,0){10.75}}
\put(56,61.25){\line(1,0){10.75}}
\put(55.75,71.75){\line(1,0){10.75}}
\put(55.75,82.5){\line(1,0){10.75}}
\put(56.5,93){\line(1,0){10.75}}
\put(46.25,103.75){\line(1,0){10.75}}
\put(87.25,61.5){\line(0,1){10.75}}
\put(4.75,71.75){\line(0,1){10.75}}
\put(66.5,29.5){\line(1,0){10.75}}
\put(136,29.5){\line(1,0){10.75}}
\put(66.5,40){\line(1,0){10.75}}
\put(66.5,51){\line(1,0){10.75}}
\put(66.5,61.25){\line(1,0){10.75}}
\put(66.25,71.75){\line(1,0){10.75}}
\put(66.25,82.5){\line(1,0){10.75}}
\put(67,93){\line(1,0){10.75}}
\put(56.75,103.75){\line(1,0){10.75}}
\put(87.25,72){\line(0,1){10.75}}
\put(4.75,82.25){\line(0,1){10.75}}
\put(77,19.75){\line(1,0){10.75}}
\put(106.5,19.75){\line(1,0){10.75}}
\put(146.5,19.75){\line(1,0){10.75}}
\put(77,29.5){\line(1,0){10.75}}
\put(146.5,29.5){\line(1,0){10.75}}
\put(77,40){\line(1,0){10.75}}
\put(106.5,40){\line(1,0){10.75}}
\put(77,51){\line(1,0){10.75}}
\put(77,61.25){\line(1,0){10.75}}
\put(76.75,71.75){\line(1,0){10.75}}
\put(76.75,82.5){\line(1,0){10.75}}
\put(67.25,103.75){\line(1,0){10.75}}
\put(87.25,82.5){\line(0,1){10.75}}
\put(4.75,92.75){\line(0,1){10.75}}
\put(14.25,19.5){\circle*{2.062}}
\put(4.5,103.5){\circle*{2.062}}
\put(77.75,93){\circle*{2.062}}
\put(24.25,19.25){\circle*{2.121}}
\put(24.25,29){\circle*{2.121}}
\put(24.25,39.5){\circle*{2.121}}
\put(24.25,50.5){\circle*{2.121}}
\put(24.25,60.75){\circle*{2.121}}
\put(24,71.25){\circle*{2.121}}
\put(24,82){\circle*{2.121}}
\put(24.75,92.5){\circle*{2.121}}
\put(14.5,103.25){\circle*{2.121}}
\put(87.75,29.75){\circle*{2.121}}
\put(117.25,29.75){\circle*{2.121}}
\put(157.25,29.75){\circle*{2.121}}
\put(5.25,40){\circle*{2.121}}
\put(78,103.5){\circle*{2.121}}
\put(34.75,19.25){\circle*{2.121}}
\put(34.75,29){\circle*{2.121}}
\put(34.75,39.5){\circle*{2.121}}
\put(34.75,50.5){\circle*{2.121}}
\put(34.75,60.75){\circle*{2.121}}
\put(34.5,71.25){\circle*{2.121}}
\put(34.5,82){\circle*{2.121}}
\put(35.25,92.5){\circle*{2.121}}
\put(25,103.25){\circle*{2.121}}
\put(87.75,40.25){\circle*{2.121}}
\put(117.25,40.25){\circle*{2.121}}
\put(5.25,50.5){\circle*{2.121}}
\put(45.25,19.25){\circle*{2.121}}
\put(45.25,29){\circle*{2.121}}
\put(45.25,39.5){\circle*{2.121}}
\put(45.25,50.5){\circle*{2.121}}
\put(45.25,60.75){\circle*{2.121}}
\put(45,71.25){\circle*{2.121}}
\put(45,82){\circle*{2.121}}
\put(45.75,92.5){\circle*{2.121}}
\put(35.5,103.25){\circle*{2.121}}
\put(87.75,50.75){\circle*{2.121}}
\put(5.25,61){\circle*{2.121}}
\put(55.75,19.25){\circle*{2.121}}
\put(55.75,29){\circle*{2.121}}
\put(55.75,39.5){\circle*{2.121}}
\put(55.75,50.5){\circle*{2.121}}
\put(55.75,60.75){\circle*{2.121}}
\put(55.5,71.25){\circle*{2.121}}
\put(55.5,82){\circle*{2.121}}
\put(56.25,92.5){\circle*{2.121}}
\put(46,103.25){\circle*{2.121}}
\put(87.75,61.25){\circle*{2.121}}
\put(5.25,71.5){\circle*{2.121}}
\put(66.25,19.25){\circle*{2.121}}
\put(135.75,19.25){\circle*{2.121}}
\put(66.25,29){\circle*{2.121}}
\put(135.75,29){\circle*{2.121}}
\put(66.25,39.5){\circle*{2.121}}
\put(66.25,50.5){\circle*{2.121}}
\put(66.25,60.75){\circle*{2.121}}
\put(66,71.25){\circle*{2.121}}
\put(66,82){\circle*{2.121}}
\put(66.75,92.5){\circle*{2.121}}
\put(56.5,103.25){\circle*{2.121}}
\put(87.75,71.75){\circle*{2.121}}
\put(5.25,82){\circle*{2.121}}
\put(76.75,19.25){\circle*{2.121}}
\put(106.25,19.25){\circle*{2.121}}
\put(146.25,19.25){\circle*{2.121}}
\put(77.5,29.25){\circle*{2.121}}
\put(107,29.25){\circle*{2.121}}
\put(147,29.25){\circle*{2.121}}
\put(77.5,39.75){\circle*{2.121}}
\put(107,39.75){\circle*{2.121}}
\put(77.5,50.75){\circle*{2.121}}
\put(76.75,60.75){\circle*{2.121}}
\put(77.25,71.5){\circle*{2.121}}
\put(77.25,82.25){\circle*{2.121}}
\put(67,103.25){\circle*{2.121}}
\put(87.75,82.25){\circle*{2.121}}
\put(5.25,92.5){\circle*{2.121}}
\put(87.25,19.25){\circle*{2.121}}
\put(116.75,19.25){\circle*{2.121}}
\put(156.75,19.25){\circle*{2.121}}
\put(4.75,29.5){\circle*{2.121}}
\put(87.75,92.75){\circle*{2.121}}
\put(14.5,29){\circle*{2.121}}
\put(14.5,39.5){\circle*{2.121}}
\put(14.5,50.5){\circle*{2.121}}
\put(14.5,60.75){\circle*{2.121}}
\put(14.25,71.25){\circle*{2.121}}
\put(14.25,82){\circle*{2.121}}
\put(15,92.5){\circle*{2.121}}
\put(14.5,29.25){\line(0,1){75.25}}
\put(24.5,29.5){\line(0,1){75.25}}
\put(34.5,29.5){\line(0,1){75.25}}
\put(45.25,29.25){\line(0,1){75.25}}
\put(55.5,29.25){\line(0,1){75.25}}
\put(66.5,29){\line(0,1){75.25}}
\put(77.25,93.25){\line(0,-1){74.25}}
\put(19.5,17.75){\makebox(0,0)[cc]{$a$}}
\put(29,17.75){\makebox(0,0)[cc]{$a$}}
\put(40,17.75){\makebox(0,0)[cc]{$a$}}
\put(49.75,17.5){\makebox(0,0)[cc]{$a$}}
\put(60.75,17.5){\makebox(0,0)[cc]{$a$}}
\put(71.75,17.5){\makebox(0,0)[cc]{$a$}}
\put(141.25,17.5){\makebox(0,0)[cc]{$a$}}
\put(141.25,31.5){\makebox(0,0)[cc]{$a$}}
\put(82.75,17.5){\makebox(0,0)[cc]{$a$}}
\put(112.25,17.5){\makebox(0,0)[cc]{$a$}}
\put(112.5,41.25){\makebox(0,0)[cc]{$a$}}
\put(152.25,17.5){\makebox(0,0)[cc]{$a$}}
\put(152.25,31.5){\makebox(0,0)[cc]{$a$}}
\put(89.5,24.75){\makebox(0,0)[cc]{$b$}}
\put(119,24.75){\makebox(0,0)[cc]{$b$}}
\put(104,24.5){\makebox(0,0)[cc]{$b$}}
\put(159,24.75){\makebox(0,0)[cc]{$b$}}
\put(89.5,35.5){\makebox(0,0)[cc]{$b$}}
\put(119,35.5){\makebox(0,0)[cc]{$b$}}
\put(104,35.25){\makebox(0,0)[cc]{$b$}}
\put(89.5,46.25){\makebox(0,0)[cc]{$b$}}
\put(89.5,57){\makebox(0,0)[cc]{$b$}}
\put(89.5,67.75){\makebox(0,0)[cc]{$b$}}
\put(89.5,78.5){\makebox(0,0)[cc]{$b$}}
\put(89.5,89.25){\makebox(0,0)[cc]{$b$}}
\put(19.375,19.75){\vector(1,0){.07}}\put(17.25,19.75){\line(1,0){4.25}}
\put(31.625,19.75){\vector(1,0){.07}}\put(29.5,19.75){\line(1,0){4.25}}
\put(38.375,19.75){\vector(1,0){.07}}\put(36.25,19.75){\line(1,0){4.25}}
\put(50.625,19.75){\vector(1,0){.07}}\put(48.5,19.75){\line(1,0){4.25}}
\put(80.625,19.75){\vector(1,0){.07}}\put(78.5,19.75){\line(1,0){4.25}}
\put(110.125,19.75){\vector(1,0){.07}}\put(108,19.75){\line(1,0){4.25}}
\put(150.125,19.75){\vector(1,0){.07}}\put(148,19.75){\line(1,0){4.25}}
\put(87.375,19.75){\vector(1,0){.07}}\put(85.25,19.75){\line(1,0){4.25}}
\put(116.875,19.75){\vector(1,0){.07}}\put(114.75,19.75){\line(1,0){4.25}}
\put(156.875,19.75){\vector(1,0){.07}}\put(154.75,19.75){\line(1,0){4.25}}
\put(61,19.75){\vector(1,0){.07}}\put(58.5,19.75){\line(1,0){5}}
\put(71.875,19.75){\vector(1,0){.07}}\put(69.75,19.75){\line(1,0){4.25}}
\put(141.375,19.75){\vector(1,0){.07}}\put(139.25,19.75){\line(1,0){4.25}}
\put(87.25,25.125){\vector(0,1){.07}}\put(87.25,23){\line(0,1){4.25}}
\put(116.75,25.125){\vector(0,1){.07}}\put(116.75,23){\line(0,1){4.25}}
\put(156.75,25.125){\vector(0,1){.07}}\put(156.75,23){\line(0,1){4.25}}
\put(87.25,35.625){\vector(0,1){.07}}\put(87.25,32.75){\line(0,1){5.75}}
\put(116.75,35.625){\vector(0,1){.07}}\put(116.75,32.75){\line(0,1){5.75}}
\put(87.25,46.375){\vector(0,1){.07}}\put(87.25,44){\line(0,1){4.75}}
\put(87.25,56.625){\vector(0,1){.07}}\put(87.25,53.75){\line(0,1){5.75}}
\put(87.25,66.875){\vector(0,1){.07}}\put(87.25,64){\line(0,1){5.75}}
\put(87.25,77.75){\vector(0,1){.07}}\put(87.25,75){\line(0,1){5.5}}
\put(87.25,88.625){\vector(0,1){.07}}\put(87.25,86){\line(0,1){5.25}}
\put(77.5,99){\vector(0,1){.07}}\put(77.5,96.5){\line(0,1){5}}
\put(82.75,93.25){\vector(1,0){.07}}\put(80.5,93.25){\line(1,0){4.5}}
\put(9.375,29.625){\vector(1,0){.07}}\multiput(7.25,29.5)(.53125,.03125){8}{\line(1,0){.53125}}
\put(14.25,24.875){\vector(0,1){.07}}\put(14.25,22.5){\line(0,1){4.75}}
\multiput(106.75,40)(-.03125,-2.625){8}{\line(0,-1){2.625}}
\put(106.625,35.25){\vector(0,1){.07}}\multiput(106.75,32.75)(-.03125,.625){8}{\line(0,1){.625}}
\put(106.25,24.25){\vector(0,1){.07}}\put(106.25,21.75){\line(0,1){5}}
\put(112,39.875){\vector(1,0){.07}}\multiput(109.75,39.75)(.5625,.03125){8}{\line(1,0){.5625}}
\put(141.75,29.625){\vector(1,0){.07}}\multiput(139,29.75)(.6875,-.03125){8}{\line(1,0){.6875}}
\put(152,29.5){\vector(1,0){.07}}\put(149.5,29.5){\line(1,0){5}}
\put(135.75,24.625){\vector(0,1){.07}}\put(135.75,22.5){\line(0,1){4.25}}
\end{picture}

\caption{The chess board with two squares removed and two dominos.} \label{f4}
\end{figure}
\end{center}

\proof The elementary solution of this problem is well known. Color squares of $P$ in black and white in the usual (chessboard) way. Then the number of squares of one color (black or white) in $P$ differs from the number of squares of the other color by 2. Since each domino covers exactly one white and exactly one black square,  $P$ cannot be tiled by dominos.

Here is a solution which, although less elementary, can be applied to many regions of the plane square grid for which the elementary proof above does not work; it also applies to regions of non-square (say, hexagonal) lattices on the plane. This solution first appeared in a paper by W. Thurston \cite{Th}. The ideas of that paper have many applications in several areas of mathematics from combinatorics to probability to mathematical physics.

Let us label horizontal edges pointed rightward in $P$ by the letter $a$ and vertical edges pointed upward by the latter $b$. Then the (counterclockwise) boundary of $P$ has label $W=a^7b^7a\iv ba^{-7}b^{-7}ab$. Every domino can be placed either vertically or horizontally. In the first case its boundary label is $ab^2a\iv b^{-2}$, and in the second case its boundary label is $a^2ba^{-2}b\iv$. Now consider the group $G$ with the presentation $\la a,b \mid ab^2a\iv b^{-2}=1, a^2ba^{-2}b\iv=1\ra$. Suppose that $P$ can be tiled by the dominos. Then every such tiling turns $P$ into a van Kampen diagram over the presentation of $G$. Hence, by the van Kampen lemma (more precisely by Proposition \ref{prop1}), the word $W$ would be equal to 1 in $G$. But consider the 6-element symmetric group $S_3$ and two permutations $\alpha=(1,2), \beta=(2,3)$ in it. Clearly both relations of $G$ hold if we replace $a$ by $\alpha$ and $b$ by $\beta$. Hence the map $a\mapsto \alpha$, $b\mapsto\beta$ extends to a homomorphism $G\to S_3$. Note that $W(\alpha, \beta)=(\alpha\beta)^4=\alpha\beta$ is the permutation $(1,3,2)$ which is not trivial. Hence $W$ is not equal to 1 in $G$, a contradiction.\endproof

This example shows a similarity between the word problem in groups and the tiling problem. Indeed, by the van Kampen lemma, a word $W$ is equal to 1 in $G=\la X\mid R\ra$ if and only if we can tile a disc with boundary labeled by $W$ by pieces whose boundaries are labeled by words from $R$. Nevertheless the word problem differs from the tiling problem as in Example \ref{ex1} because, as we have seen, when we draw a van Kampen diagram, we do not fix the shapes or sizes of the tiles, only the labels of the boundaries of them, so one can view a general van Kampen diagram as a tiling of a disc by tiles made of soft rubber, while the traditional tiling problems such as Example \ref{ex1} are about tiles made of hard plastic.

\subsection{The computational complexity and algebraic systems}

\label{s:CAP}

\subsubsection{Solvability of algorithmic problems: theory and practice}\label{ss:soap} The solvability of an algorithmic problem does not mean that the problem can
be solved in practice.
First of all the existence of an algorithm does not mean that
it is readily available. For example, the word problem in every finite group is certainly decidable. But there is no algorithm that, given a finite set of relations $r_i=1$, $i=1,\ldots,n$, and a relation $r=1$ over a finite alphabet $X$, would decide whether every finite group generated by $X$ where the relations $r_i=1$ hold ($i=1,\ldots,n$) also satisfies $r=1$. That is a result of Slobodskoj \cite{Slo}. Thus there is no {\em \index{Uniform algorithm}uniform} algorithm solving the word problem in all finite groups at once. In general we usually ask for {\em existence} of an algorithm solving certain algorithmic problem but not for a procedure to find this algorithm. (In most concrete cases, if an algorithmic problem in a given group is decidable, we are able to write down an algorithm that solves the problem.)

Another, more important, obstacle is that the algorithm can be too slow.
For example, if a finitely presented group $G$ is \index{Group!residually finite}{\em residually finite} \footnote{A group is residually finite if the intersection of all its subgroups of finite index is trivial.}, then the word problem can be solved by the so called McKinsey algorithm \cite{McK}: list all finite quotients of $G$ (that is consider all finite groups $H$ one by one, consider every map $\phi$ from the set of generators of $G$ into $H$, and check if this  map takes all $r_i$ to $1$, and whether the images of generators of $G$ generate $H$; if - yes, then include the pair $(H, \phi)$ in the list of quotients), and at the same time list all the corollaries of the defining relations of $G$ (say, list all products of conjugates of $r_i$'s and their inverses). Every word $r$ is either in the list of corollaries or is not trivial in one of the finite quotients. In the first case $r=1$ in $G$, in the second case, $r\ne 1$ in $G$, and after some time we will surely know which of this options holds. It is clear that in general this algorithm is very slow. Even for a short word $r$ it would take a lot of time
to decide, using this algorithm, if $r=1$ in a given residually finite group.

\subsubsection{Computational complexity}\label{ss:cc} Thus the next thing to do, after we find out that an algorithmic problem is decidable,
is to find the computational complexity of this problem.

To make this more precise let us present here some concepts from the Computational
Complexity Theory.

Any decision problem $D$ may be considered as a membership problem for elements
of some set $B_D$ in a subset $S_D$. For example if $D$ is the word
problem for a
group $G=\langle X\mid R\rangle$, then $B_D$ is the set of all reduced group words in the alphabet $X$,
and $S_D$
is the set of products of conjugates of elements of $R$ and their inverses in the free group over $X$ (i.e. the normal subgroup of the free group over $X$ generated by $R$).

With any element $x$ in $B_D$ one associates a number which is called the
{\em size} of this element. Usually the size is roughly the minimal space
which is needed to write
$x$ down. For example the size of a word is typically its length.

The size depends on the way we choose to represent the elements.
For example if $x$ is a natural number then we can represent $x$
as $x$ units. Then the
size of $x$ will be equal to $x$. If we represent $x$ as a sequence of binary
digits then the size of $x$ will be approximately $\log_2(x)$.

Algorithms may be realized by Turing
machines (for  a formal definition see below).
We can assume that this machine is equipped with a voice
synthesizer and can say two words ``Yes" and ``No".
An algorithm solving
the decision problem $D$ starts working with an element $x$ of $B_D$
written on the tape of the machine. When it ends, it says ``Yes" if
$x\in S_D$ or ``No" if $x\not\in S_D$. There are also different kinds of algorithms that say ``Yes" when $x\in S_D$ and work indefinitely long without saying anything if $x\not\in S_D$.
Note that if there is an algorithm that recognizes the ``yes" part and there is an algorithm recognizing the ``no" part of $D$, then there exists an algorithm solving $D$.

\subsubsection{Turing machines (a formal definition)} \label{tm}

Recall one of the many equivalent definitions of a \index{Turing machine}{\em Turing machine}.

A
 {\em \index{Turing machine!multi-tape}multi-tape} Turing machine has $k$ tapes and $k$ heads observing
the tapes. One can view it as a quadruple $$M= \langle I, Y, Q,
\Theta\rangle$$ where $I$ is the input alphabet, $Y$ is the tape
alphabet ($I\subseteq Y$), $Q=\sqcup Q_i,
i=1,\ldots, k$ is the set of
states\footnote{$\sqcup$ denotes disjoint union.} of the heads
of the machine,
$\Theta$ is a set of
commands. Some ({\em \index{Turing machine!recognizing}recognizing}) Turing machines also have two distinguished vectors of state letters: $\vec s_1$ is the $k$-vector of start states, $\vec s_0$ is the $k$-vector of accept states.

The leftmost square on every tape is always marked by $\cee$, the
rightmost square is always marked by $\dol$.
The head is placed between two consecutive squares on the tape. A
{\em \index{Configuration of a tape}configuration} of a tape of a Turing machine is a word $\cee u
q v \dol$ where $q$ is the current state of the head of that tape,
$u$ is the word in $Y$ to the left of the head and $v$ is the word
in $Y$  to the right of the head, and so the word written on the
entire tape is $uv.$

At every moment the head of each tape observes two letters on that
tape: the last letter of $u$ (or $\alpha$) and the first letter of
$v$ (or $\omega$).

A {\em \index{Configuration of a Turing machine}configuration} $U$ of a
Turing machine is the word
$$U_1U_2\dots U_k$$
where $U_i$ is a configuration of tape $i$.

Assuming that the Turing machine is recognizing, we can define input configurations and accepted (stop) configurations.
An {\it input configuration} is a configuration where the word written
on the first tape is in $I^+$, all other tapes are empty, the head on the first tape observes
the right marker $\omega$, and the states of all tapes form the start vector $\vec s_1$.
An {\em \index{Accept configuration}accept (or stop) configuration} is any configuration
where the state vector is $\vec s_0$, the accept vector of the machine. We shall always assume (as can be easily achieved) that in the accept configuration every tape is empty.

A {\em command} of a Turing machine is determined
by the states of the heads and some of the $2k$ letters
observed by the heads. As a result of a
transition we replace some of these $2k$ letters by other letters,
insert new squares in some of the tapes and may move the heads
one square to the
left (right) with respect to the corresponding tapes.

For example in a one-tape machine every transition is of the
following form: $$uqv\to u'q'v'$$ where $u, v, u', v'$ are letters
(could be end markers) or empty words. The only constraint  is that
the result of applying the substitution $uqv\to u'q'v'$ to a
configuration word must be a configuration word again, in particular
the end markers cannot be deleted or inserted.

This command means that if the state of the
head is $q$, $u$ is written to the left of $q$ and $v$ is written to
the right of $q$ then the machine must replace $u$ by $u'$, $q$ by
$q'$ and $v$ by $v'$.

For a general  $k$-tape machine a {\em \index{Command of a Turing machine}command} is a vector
$$[U_1\to V_1,\ldots, U_k\to V_k]$$
where $U_i\to V_i$ is a command of a 1-tape machine, the elementary commands (also called {\em \index{Part of a command}parts of the command})
$U_i\to V_i$ are listed in the order of tape numbers. In order to execute this
command, the machine first checks if every $U_i$ is a subword of the configuration
of tape $i$ ($i=1,\ldots, k$), and then replaces all $U_i$ by $V_i$.

A {\em \index{Computation}computation} is a sequence of configurations $C_1\to \dots\to C_n$ such that for
every $i=1,\ldots,  n-1$ the machine passes from $C_i$ to $C_{i+1}$ by applying one
of the commands from $\Theta$.  A configuration $C$ is said to be {\em
\index{Computation!accepted}accepted} by a machine $M$ if there exists at least one computation which starts
with $C$
and ends with an accept configuration.

A word $u\in X^*$ is said to be {\em \index{Word!accepted by a Turing machine}accepted} by the machine if the
corresponding input configuration is accepted.  The set of all accepted words
over the alphabet $X$ is called the {\em \index{Language!accepted (recognized) by a Turing machine}language accepted (recognized) by the machine}

Let  $C = C_1\to\dots\to C_n$
be a computation of a machine $M$ such that for
every $j=1,\ldots, g-1$ the configuration $C_{j+1}$ is obtained from $C_j$ by a
command $\theta_j$ from $\Theta$.  Then we call the word $\theta_1 \ldots \theta_{n-1}$
the {\em
\index{History of computation}history} of this computation.  The number $n$ will be called the {\em \index{Time of computation}time} (or {\em length})
of the computation.  Let $p_i$ ($i=1,\ldots, g$) be the sum of the
lengths of the configurations of the tapes in the configuration $C_i$.
Then the maximum of all $p_i$
will be called the {\em \index{Space of computation}space of the computation}.

We do not only consider {\em \index{Turing machine!deterministic}deterministic} Turing machines, for example,
we allow several transitions with the same left side.

Note that Turing machines can also be viewed as {\em rewriting systems} where objects are the configurations and substitution rules are the commands of the machine. This point of view allows one to use the tools of the string rewriting theory, normal forms, confluence, etc.
One can also view Turing machines as semigroups of partial transformations of the set of configurations generated by the transformations induced by the commands. This point of view allows to explore dynamics properties of Turing machines (say, periodic computations). That is often needed when one tries to simulate a Turing machine by a group \cite{OSamenab}.

\subsubsection{The time and space functions}\label{ss:ttasf} With every Turing machine $A$ one can associate two
important functions: the time function and
the space function.

The {\em \index{Time function of a Turing machine} time } ({\em \index{Space function of a Turing machine}space }) function $t_A(n)$ (resp. $s_A(n)$) of a Turing machine $A$ is the minimal number of steps of the machine (resp. number of cells on the tape visited by the machine) to decide that any element $x\in S_D$ of size
$\leq n$ is indeed in $S_D$.

We can compare the time (space) functions of algorithms as follows. If $f,g\colon {\mathbb N}\to \N$, we write \index{$\prec$}$f\prec g$ if for some constant $C$ we have $f(n)\le Cg(Cn)+Cn+C$ for every $n\in \N$. We say that $f$ and $g$ are {\em \index{Equivalent functions}equivalent}, if $f\prec g$ and $g\prec f$. Thus every sublinear function is equivalent to $n$, functions $n^5$ and $2n^5$, $2^x$ and $3^x$ are equivalent, but functions $n^{3.1}$ and $n^{3.2}$ are not.

The following relation is clear: $$s_A(n)\prec t_A(n).$$

Indeed, even if at every step the algorithm used a new tape cell,
$s_A(n)$ would be only equivalent to $t_A(n)$ (since the number of cells used by the machine at every step is bounded). In reality $s_A(n)$ is almost always
less than $t_A(n)$. On the other hand the following relation also holds \cite{GJ}:
$$t_A(n) \prec 2^{s_A(n)}$$

\subsubsection{A very rarely accepting Turing machine}\label{ss:avratm}

Turing machines, like groups, can have some extreme properties. It is well known, for example, that a Turing machine can have undecidable {\em \index{Halting problem}halting problem}, that is there is no algorithm that checks whether the machine will ever stop starting with a given input configuration. This is obviously equivalent to the property that the time and space functions of this Turing machine are bigger than any given recursive function (i.e. a function $\N\to \N$ whose graph and its complement can be recognized by a Turing machine). In Section \ref{pnrdf} we shall need a Turing machine with even stronger property.

Let $X$ be a recursively enumerable language in a 2-letter alphabet, i.e. a set of  binary
words recognized by a (non-deterministic) Turing machine $M$. If
$x\in X$ then the {\em time} of $x$ (denoted $\cost(x)$ or $\cost_M(x)$) is, by
definition, the minimal time of an accepting computation of $M$ with
input $x$. For any increasing function $h\colon \N\to \N$, a real number $m$ is called {\em \index{$h$-good number}$h$-good} for $M$ if for
every $w\in X$, $|w|<m$ implies $h(\cost(w))<m$.

Let $h(n)=\exp\exp(n)$.

\begin{theorem}[Olshanskii, Sapir \cite{OSnq}] \label{re}
There exists a Turing machine $M$ recognizing a recursively enumerable non-recursive
set $X$ such that the set of all $h$-good numbers for $M$ is infinite.
\end{theorem}

 It is easy to observe that the set $G$ of $h$-good numbers of $M$ cannot be recursively enumerable. Moreover it cannot contain any infinite recursively enumerable subset. Indeed, if $G$ contained an infinite recursively enumerable subset $R$, then for every input configuration $C$ of $M$ we would start the Turing machine enumerating $R$, and wait till we get a number $r$ from $R$ that is bigger than the size of $C$. Then to check if $M$ halts on $C$, we would check all (finitely many) computations of $M$ starting with $C$ and having length less than $\log\log r$. The Turing machine $M$ halts starting with $C$ if one of these computations ends with the accept configuration. Hence the set $X$  would be recursive. On the other hand, the complement of $G$ is recursively enumerable because in order to verify that $r$ is {\em not} an $h$-good number, one needs only to start $M$ and wait till it accepts an input configuration $C$ of size $\le r$ but $\cost(C)>\log\log r$. Thus $G$ is an {\em \index{Immune set}immune set} and its complement is {\em \index{Simple set}simple} in terminology of \cite{Mal}. Note that the very existence of simple and immune sets was a non-trivial problem solved by Muchnik and Friedberg using the famous {\em priority argument}, answering a question of Post (see \cite{Mal} and \cite{So} for details). Our proof of Theorem \ref{re} also uses the priority argument and Matiyasevich's solution of the Hilbert 10th problem \cite{Mat}.

\subsubsection{The complexity classes}\label{ss:tcc} If there exists an algorithm $A$ which solves the ``yes" part of problem $D$ and $t_A(n)$ is bounded
from above by a polynomial in $n$ then we say that $D$
{\em \index{Polynomial time}can be solved in
polynomial time}. The solvability in exponential time, polynomial space,
exponential  space, etc. is defined similarly.

It is worth mentioning that if we modernize the Turing machine by, say,
adding more
tapes or heads, we won't
change the complexity of the problem much. For example a non-polynomial
time (space) problem cannot become polynomial as a result of that.
The class of all problems which can be solved in polynomial time is
denoted by P.

If in the definition of the time complexity, we replace the deterministic
Turing machines by non-deterministic Turing machines, then we obtain
definitions of the solvability in non-deterministic polynomial time,
non-deterministic exponential time, etc.
Recall that a {\em \index{Turing machine!non-deterministic}non-deterministic} Turing machine
is more intelligent than a deterministic one: it does not blindly
obey the commands of the program, but, at every step, guesses itself
what the next step should be.
Roughly speaking a problem $D$ can be solved in non-deterministic polynomial
time if
for every
element $x\in B_D$ there exists a proof (usually called {\em \index{Witness}witness}) that $x$ belongs
to $S_D$ and the length (size)
of this proof is bounded by a polynomial of the size of $x$.
The class of all problems which can be solved in polynomial time by
a non-deterministic Turing machine is denoted by NP.
It is not known if P=NP. This is one of the central problems in Theoretical
Computer Science.

In order to prove that a problem is solvable in polynomial time it is enough
to find a polynomial time algorithm solving this problem.

We have mentioned that the complexity of the McKinsey algorithm presented in \ref{s:CAP} is very high, but we do not know any finitely presented residually finite group with really hard word problem. Thus we formulate

\begin{prob} Is there (time or space) complexity bound for the word problem in a finitely presented residually finite group. More precisely, what is the highest class in the time complexity hierarchy \cite{GJ} where the word problem in every finitely presented residually finite group belongs? An even more bold question: is the word problem of every finitely presented residually finite group in NP?
\end{prob}

\subsubsection{Reducing one problem to another}\label{ropta}
In order to prove that a problem $D$ cannot be solved in polynomial (exponential, etc.) time
one has to take another problem $Q$ which is known to be
``hard" and reduce it to $D$.

There are several kinds of reductions used in the Computer Science literature.
One of them, the polynomial reduction in the sense of Karp \cite{GJ},
is the following. A reduction of a problem $Q$ to a problem $D$ is a function
$\phi$ from $B_Q$ to $B_D$ such that

\begin{itemize}
\item An element $x$ from $B_Q$ belongs to $S_Q$ if and only if $\phi(x)$
belongs to $S_D$.
\item The element $\phi(x)$ can be computed in polynomial time,
in particular the size of $\phi(x)$ is bounded by a polynomial
of the size of $x$.
\end{itemize}

It is clear that if $Q$ is ``hard" and $Q$ can be reduced to $D$ then $D$
is ``hard" as well.

Notice that, when we prove the undecidability of a problem $D$, we usually
also
reduce a problem $Q$ known to be ``hard" (which in this case means undecidable),
to $D$.
In order to reduce $Q$ to $D$ we find a similar mapping $\phi$, but
we do not care about the size of $\phi(x)$. The ``basic hard" problem is usually the halting problem of a Turing machine (or some other kind of machine such as Minsky machines). But it could be a different kind of undecidable problem such as the Hilbert's 10th problem (see \cite[Section 6]{KS} for details) or a membership in the range of a general recursive function as in McKenzie-Thompson \cite{McT}.

\subsubsection{An NP-complete problem}\label{anp}
Many known problems $A$ are such that $A$ is in NP and every other NP-problem polynomially reduces to $A$. Such problems are called {\em \index{NP-complete problem} NP-complete}. One of these problems is the following {\em \index{Exact bin packing problem}exact bin packing problem} \cite[Page 226]{GJ}, \cite[Page 255]{KLMT}:

\begin{prob}(Exact bin packing problem)\label{bin}
\begin{itemize}
\item {\bf Input:} A $k$-tuple of positive integers $(r_1, \ldots, r_k)$ and positive integers $B,N$.

\item {\bf Output:} ``Yes" if there exists a partition of $\{1,\ldots, k\}$ into $N$ subsets
$\{1,\ldots, k\} = S_1 \sqcup \ldots \sqcup S_N$
such that for each $i = 1,\ldots, N$ we have
$$
\sum_{j\in S_i} r_j=B.
$$
\end{itemize}
\end{prob}

\subsubsection{Complexity classes and algebraic systems} \label{ss:ccaas} A deep connection between Computer Science and  algebra
was found by R.Fagin \cite{Fa74}. He proved, in particular,
the following amazing model-theoretic characterization of classes of finite
algebras whose membership problem
can be solved in non-deterministic polynomial time. Recall that an algebraic system is by definition a set with certain operations and predicates of various arities \cite{Mal}. For example, a linearly ordered group is an algebraic system with one binary operation (product), one unary operation (taking inverses), one nullary operation (the identity element) and one binary predicate giving the order relation.

\begin{theorem}[Fagin \cite{Fa74}] \label{ct:fagin}
The membership problem for an abstract (i.e. closed under isomorphisms)
class of finite algebraic
systems is in NP if and only if
it is the class of all finite models of a second-order
formula of the following
type:
$$\exists Q_1\exists Q_2\ldots \exists Q_n (\Theta)$$
where $Q_i$ is a predicate, and $\Theta$ is a first-order formula.
\end{theorem}

Basically this theorem means that the membership problem of a class of
finite algebraic systems is in NP if and only if we
can describe the structure of algebraic systems of this class in terms
of functions and relations.
Since all known methods of studying the structure of algebraic systems are
based on studying functions (homomorphisms,
polynomial functions, etc.) and relations (congruences, orders, etc.)
we can conclude that for the membership problem of
a class of finite algebras
being in NP is equivalent to admitting
a reasonable structure description.

Note that a word in an alphabet with $k$ letters can be considered as an algebraic system as well: a word is a linearly ordered finite set (say, $\{1,2,\ldots,n\}$) with $k$ unary predicates $L_x$ indexed by the letters of the alphabet such that $L_x(p)$ holds if the $p$-th letter in the word is $x$. The set of all words in the $k$-letter alphabet can be defined by obvious axioms involving these $k+1$ predicates (the order predicate should satisfy the axioms of linear order, plus for every $p=1,2,\ldots, n$ one and only one of the statements $L_x(p)$ holds). Reduced group words can also be defined using similar axioms, etc. Hence being able to solve the word problem in a group in non-deterministic polynomial time means, by Fagin's theorem, that we can find algebraic description of words that are equal to 1 in the group.

Note that classes where the membership problem has other types of computational
complexity have similar model theoretic characterizations. Classes with
membership problem in P have been characterized by Immerman \cite{Imm86},
Sazonov \cite{Saz80a}, \cite{Saz80b} and Vardi \cite{Var82}.
Classes with exponential time membership problem
were characterized by Fagin \cite{Fa74}. Classes with non-deterministic
exponential time membership problem have been characterized by Fagin \cite{Fa74}
and Jones and Selman \cite{JS74}. For a detailed survey of these results
see \cite{Fa90}.

\subsubsection{Van Kampen diagrams as witnesses}\label{sub:witness}

Now let again $G=\la X\mid R\ra$ be a finitely presented group. Suppose that a word $w$ in $X$ is equal to 1 in $G$. Then by Lemma \ref{lvk}of van Kampen, $w$ labels the boundary of a van Kampen diagram $\Delta$ over the presentation of $G$. Let us consider that van Kampen diagram as a witness for the equality $w=1$. Since a van Kampen diagram cannot be drawn on a tape of a Turing machine, we need to replace it by an equivalent object which is a word. Such an object is given by Proposition \ref{prop1}. Instead of $\Delta$, we should consider the (non-reduced) product

\begin{equation}\label{word}
u_1\cdot r_1\cdot u_2\cdot r_2\cdot \ldots \cdot u_m\cdot r_d\cdot u_{m+1}
\end{equation} from that proposition. By Proposition \ref{prop1}, the length of that word is --- up to a constant multiple --- the number of cells in $\Delta$ plus $|w|$. Since we identify functions that differ by a constant factor or linear summand, we can say that the size of $\Delta$ is the number of cells, i.e. the {\em area} of $\Delta$. Clearly given a word of the form (\ref{word}), i.e. a word subdivided into subwords $u_i, r_j$, it takes linear time (in the length of the word) to check if it indeed satisfies conditions of Proposition \ref{prop1}. Thus we obtain the following

\begin{prop}[See the proof of Theorem 1.1 \cite{SBR}]\label{Dehn1} Suppose that for every word $w$ that is equal to 1 in $G$ there exists a van Kampen diagram over the presentation of $G$ with at most $f(|w|)$ cells and boundary label $w$. Then the ``yes part" of the word problem of $G$ can be solved in non-deterministic time at most $f(n)$. In particular if $f$ is bounded by a polynomial, then the ``yes" part of the word problem of $G$ is in NP.
\end{prop}

\subsection{Diagrams on surfaces, quadratic equations and Poincar\'e conjecture} \label{dos}

\subsubsection{The definition}\label{ss:td} Let $S=S_{g,n}$ be a (possibly non-oriented) surface of genus $g$ with $n$ boundary components. Let again ${\mathcal P}=\la X\mid R\ra$ be a finite group presentation. A {\em \index{Van Kampen diagram!on a surface}diagram} on $S$ is a tessellation of $S$ by a labeled graph with all edges labeled by letters of $X$ and all cells (regions) being discs with boundaries labeled by words from $R^{\pm 1}$. In fact, a diagram on a surface can be obtained from a usual (disc) van Kampen diagram as follows. Let $\Delta$ be a disc diagram whose boundary is a polygon where some pairs of sides are labeled by the same or mutually inverse words. If we identify these pairs of sides, we get a diagram on a surface (with possibly non-empty boundary). Conversely,
given a diagram on a surface $S$, we can cut the surface by simple closed curves and arcs following edges of the diagram so that we obtain a polygon with some pairs of sides (corresponding to the cutting curves) labeled by the same or mutually inverse words). The result is an ordinary van Kampen diagram. For example, if we want to find out whether an element represented by a word $u$ is a product of $g$ commutators\footnote{$[x,y]$ denotes $xyx\iv y\iv$.} $[x_1,y_1]\ldots [x_g,y_g]$ in a group $G=\la X\mid R\ra$, we need to check whether there exists a van Kampen diagram on an oriented surface of genus $g$ with one boundary component labeled by $u$. In general, one can consider van Kampen diagrams on surfaces with boundaries when one tries to solve any {\em \index{Equation}quadratic equation} $w(x_1,\ldots, x_k)=1$ in $G$. Here $w$ is a word in the alphabet $X\cup \{x_1,\ldots, x_k\}$ (letters from $X$ are called {\em \index{Equation! coefficients of}coefficients}, $x_i$ are called {\em \index{Equation!unknowns of} unknowns}), and for every $i$, the word $w$ has exactly two occurrences of the form $x_i^{\pm 1}$ (hence the term {\em \index{Equation!quadratic}quadratic}). If the equation $w=1$ does not have any coefficients, say, $w=[x_1,y_1][x_2,y_2]\ldots [x_n,y_n]$, then solutions of $w=1$ in a group $G$ are in one-to-one correspondence with homomorphisms from the surface group $\pi_1(S)$ to $G$.

\subsubsection{The complexity of solving quadratic equations}\label{ss:tcosqe}

Using van Kampen diagrams on surfaces Olshanskii  proved:

\begin{theorem}[Olshanskii \cite{Olsurf}, see also Grigorchuk and Kurchanov \cite{GrK}, Griorchuk and Lys\"enok \cite{GL}, and  Lys\"enok \cite{Lys}]\label{Olsurf} For every $n$, the solvability problem for quadratic equations in $n$ variables in any free group is in $P$ (the size of the input of the problem is the sum of lengths of the coefficients).
\end{theorem}

Note that the solvability of that problem was first proved by Cromerford and Edmunds in \cite{CE}.

Recently Kharlampovich, Lys\"enok, Myasnikov and Touikan \cite{KLMT} proved that the uniform solvability problem (when the number of variables is not fixed) is harder. They also use diagrams on surfaces and results from \cite{Olsurf}, plus some clever arguments reducing NP-complete {\em exact bin packing} problem (Problem \ref{bin}) to the problem of filling (tesselating) a surface with a van Kampen diagram over the presentation consisting of $0$-cells.

\begin{theorem}[Kharlampovich, Lys\"enok, Myasnikov,Touikan \cite{KLMT}] \label{KLMT} The (uniform problem of) solvability of quadratic equations in the free group in NP-complete (the {\bf input} of that problem is an equation, the {\bf output} is ``yes", if the equation has a solution in a free group).
\end{theorem}

Note that the fact that solvability of quadratic equations over free monoids is NP-hard was proved before \cite{KLMT} by Diekert and Robson in \cite{DR}. It is still not known whether that (seemingly easier) problem is NP-complete. As Volker Diekert told us, the NP-completeness in the case of monoids would imply Theorem \ref{KLMT}, and so the case of monoids is very interesting.

\subsubsection{Quadratic equations and the Poincar\'e conjecture}\label{ss:qeatpc}

Note that quadratic equations even without coefficients (i.e. homomorphisms from the surface group) are closely related to the Poincar\'e conjecture. There are several reformulations of the Poincar\' e conjecture in terms of homomorphisms of surface groups onto the direct product of two free groups (see Jaco \cite{J}, Stallings \cite{Stal1}, Olshanskii \cite{Olsurf}). Here is the reformulation from Hempel \cite{H}. Since the Poincar\'e conjecture has now been proved by Perelman \cite{P1,P2}, we formulate that reformulation as a theorem.

\begin{theorem}[Hampel \cite{H}, Perelman \cite{P1,P2}]\label{t:hp} Let $S$ be a closed oriented surface, $\pi_1(S)$ be its fundamental group. Suppose that $\phi_1$ and $\phi_2$ are two homomorphisms from $\pi_1(S)$ onto the direct product of two free groups $F_n\times F_n$. Then there is an automorphism $\alpha$ of $\pi_1(S)$ and an automorphism $\beta$ of $F_n\times F_n$ such that $\alpha\phi_1\beta=\phi_2$.
\end{theorem}

This theorem means that there exists essentially at most one ``maximal" solution of the equation $[x_1,y_1]\ldots[x_g,y_g]=1$ in $F_n\times F_n$ for every $n$.

\subsubsection{A link between Poincar\'e conjecture and NP}\label{ss:albpcanp} The proof of Poincar\'e conjecture implies that certain algorithmic problems are in NP. For example, consider the following problem

\begin{prob}\rm{(}Triviality of the fundamental group of a 3-manifold\rm{)}.\label{Ivanov}

{\bf Input.} A compact 3-manifold $M$ without boundary represented as a union of tetrahedra.

{\bf Output.} ``Yes" if the fundamental group of $M$ is trivial and ``no" otherwise.

\end{prob}

Note that the triviality problem is a part of the uniform word problem (see \ref{ss:soap}): given a presentation of a group we need to decide if each generator is equal to 1 in the group. Note also that giving a triangulation of $M$, we can easily find the presentation of $\pi_1(M)$ which is of a comparable size as $M$.

That Problem \ref{Ivanov} is in NP follows from the fact that simply connected compact 3-manifolds are homeomorphic to the 3-sphere (the Poincar\'e conjecture) and the result of S.V. Ivanov \cite{Ivanov} that the problem of recognizing a 3-sphere is in NP.

\subsubsection{Other equations in free groups}\label{ss:oeifg}

Arbitrary equations in free groups correspond to diagrams on more complicated 2-dimensional complexes than surfaces and the solvability problem for them is much harder. A very non-trivial result of Makanin shows that this problem is decidable.

\begin{theorem}[Makanin \cite{Mak}]\label{Mak} There exists an algorithm to check whether a system of equations over a free group has a solution.
\end{theorem}

Note that for some relatively easy groups the problem of solvability of systems of equations is undecidable (this is so for, say, a finitely generated nilpotent group, see \ref{ss:tacatcp}).

Makanin's algorithm is very complicated and its complexity is known to be a tower of many exponents. Later, Makanin's algorithm was simplified and clarified by  Razborov \cite{Raz1, Raz2}.  Razborov also gave a description of all solutions. Still the exact complexity of solving equations in free groups was an open problem till the following relatively recent result.

\begin{theorem}[Diekert, Gutierrez, Hagenah \cite{DGH}] The problem of solvability of an equation in a free group is in PSPACE, that is it can be solved by a deterministic Turing machine using polynomial amount of space (in terms of the size of the equation).
\end{theorem}

Other results from \cite{DGH} show that PSPACE most probably cannot be improved because a very similar problem of deciding the existential theory of a free group (i.e. solving systems of equations and {\em \index{Inequation} inequations} of the form $w\ne1$) turned out to be PSPACE-complete (the solvability of systems of equations and inequations in a free group was proved to be decidable by Makanin in \cite{Mak1}). See also \ref{ss:tipfhg} below.

\subsubsection{The conjugacy problem and annular diagrams} \label{tcpaad} Note that the conjugacy problem is about solvability of the (quadratic) equations of the form $xux\iv=v$ where $u, v$ are coefficients and $x$ is an unknown. Thus by \ref{dos}, we need to consider diagrams on an annulus.
So let the surface $S$ be an annulus (i.e. the surface $S_{0,2}$), $\pp$ be a presentation of some group $G$, and $u,v$ be two group words over the generating set of this presentation. Then $u$ and $v$ are conjugate in $G$ if and only if there exists a diagram $\Delta$ on $S$ over a presentation $\pp$ with boundary labels $u,v$ (that was first noticed by Schupp and appeared in \cite{MS}, hence annular diagrams are sometimes called {\em \index{Schupp diagram}Schupp diagrams}).

\subsection{Isoperimetric and isodiametric functions of groups}\label{s:iaifog}

\subsubsection{The isoperimetric functions}\label{ss:tiff} Thus with every finite presentation ${\mathcal P}$ of group $G$ we associate the following important function: $f(n)$ is the smallest number such that every word $w$ that is equal to 1 in $G$, $|w|\le n$, labels the boundary of a van Kampen diagram over the presentation of $G$ with at most $f(n)$ cells. This function is called the {\em \index{Dehn function}Dehn function} of ${\mathcal P}$. It is easy to see that the Dehn functions of different finite presentations of the same group are equivalent, hence we can talk about the Dehn  function of a finitely presented group. Every function $g\colon \N\to\N$ with $f\prec g$ is called an {\em \index{Isoperimetric function}isoperimetric function} of $G$.
The Dehn functions of groups were introduced (under different names) by Computer Scientists Madlener and Otto \cite{MO} studying complexity of the word problem in groups, and by Gromov \cite{GrHyp} as a geometric invariant of groups (see also Gersten \cite{G2} where the name ``Dehn function" was introduced).

As a direct consequence of Proposition \ref{Dehn1}, we get

\begin{theorem}\label{2} \rm{(}See \cite[Theorem 1.1]{SBR}\rm{)}. Every Dehn function of
a finitely presented group $G$ is (equivalent to) the time function
of a Turing machine solving (non-deterministically) the ``yes" part of the word problem
in $G$.
\end{theorem}

Hence we have

\begin{theorem}[Madlener, Otto \cite{MO}, Gersten \cite{G2}]\label{MO} The word problem in a finitely presented group is solvable if and only if the Dehn function (or one of the isoperimetric functions) is recursive.
\end{theorem}

Being a geometric invariant means, in particular, that if two finitely presented groups are {\em quasi-isometric}, then their Dehn functions are equivalent. Recall that if  $M_1$, $M_2$ are metric spaces with distance functions $d_1, d_2$, then a partial function $\phi\colon M_1\to M_2$ is called {\em \index{Quasi-isometry}quasi-isometry} between $M_1$, $M_2$ if there exist three constants $A\ge 1,B\ge 0,C\ge 0$ such that

\begin{equation}\label{e:L} \frac{1}{A}\; d_2(\phi(x),\phi(y))-B\leq d_1(x,y)\leq A\; d_2(\phi(x),\phi(y))+B \end{equation} for all $x,y\in M_1$
and every point in $M_1$ is at distance at most $C$ from the domain of $\phi$, while every point in $M_2$ is at distance at most $C$ from the range of $\phi$. In this case we call $\phi$ $(A,B,C)$-quasi-isometry. If $C=0$, we call $\phi$ an $(A,B)$-{\em quasi-isometry}, and if $B=C=0$, then $\phi$ is an $A$-{\em \index{Bi-Lipschitz map}bi-Lipschitz map}. Finally if $B=C=0$ and we only have the second of the two inequalities in (\ref{e:L}), then $\phi$ is an $A$-{\em \index{Lipschitz map}Lipschitz} map.

We say that $\phi\colon M_1\to M_2$ is a quasi-isometry (bi-Lipschitz map or Lipschitz map) from $M_1$ to $M_2$ if this map is a quasi-isometry (bi-Lipschitz map or Lipschitz map) between $M_1$ and $\phi(M_1)$.

The spaces $M_1$ and $M_2$ are called {\em \index{Quasi-isometric spaces}quasi-isometric} if there exists a quasi-isometry between $M_1$ and $M_2$.
For example, any two Cayley graphs (see \ref{ccms} for a definition) of a group with respect to two finite generating sets are quasi-isometric as metric spaces (the distance between two vertices is the length of a shortest path connecting them).  Moreover any isomorphism between any subgroup of finite index in one group and a subgroup of finite index of another group induces a quasi-isometry of the Cayley graphs of these groups (that was a starting point of a proof of the celebrated Mostow rigidity theorem). Two finitely generated groups are called {\em \index{Quasi-isometric groups}quasi-isometric} if their Cayley graphs (with respect to finite generating sets) are quasi-isometric. It is not difficult to observe that two quasi-isometric groups have equivalent Dehn functions. In particular, if one of the two quasi-isometric groups has solvable word problem, the other group must also have solvable word problem (by Theorem \ref{MO}).

\subsubsection{The conjugacy problem and quasi-isometry} \label{ss:tcpaq} Note that the analogous statement is not true for the conjugacy problem: there exist two pairs of groups $A_1>B_1$ and $B_2>A_2$, each extension is of index 2, such that $A_1,A_2$ have solvable conjugacy problem which $B_1, B_2$ have unsolvable conjugacy problem (see Gorjaga-Kirkinskii \cite{GK} and Collins-Miller \cite{CM}).

\subsubsection{The isodiametric function}\label{ss:tif}

Another important function associated with van Kampen diagrams is the {\em \index{Isodiametric function}isodiametric function}: for every word $w$ that is equal to 1 in the group $G=\la X\mid R\ra$, let $\mathrm{di}(w)$ be the minimal diameter of a van Kampen diagram with boundary label $w$. Then $\mathrm{di}(n)$ is the maximal value of $\mathrm{di}(w)$ for all $w$ with $|w|\le n$, $w=_G1$. The area of a van Kampen diagram with boundary label $w$ gives the number of factors in a representation of $w$ as a product of conjugates of elements of $R$ and their inverses (by Proposition \ref{Dehn1}). The diameter is roughly the
maximal size of a conjugator used in this representation. It is easy to see that the isodiametric function does not depend (up to equivalence) on the finite presentation of $G$. It is known (D. Cohen \cite{Coh}, Gersten \cite{Ger3}, Birget \cite{Bir3}, Gersten and Riley \cite{GeR}) that for every finitely presented group $G$ the Dehn function $f(n)$ and isodiametric function $g(n)$ satisfy the following double exponential inequality:

$$f(n)\le a^{b^{g(n)}}$$
for some constants $a,b$. This immediately implies that the word problem in $G$ is solvable if and only if the isodiametric function is recursive. It is still an open problem whether one can reduce the number of exponents to one.

\begin{prob} [D. Cohen \cite{Coh}, also attributed to Stallings] Is it true that for every finitely presented group $G$ there exists a constant $a>1$ such that $$f(n)\le a^{g(n)}$$ where $f, g$ are, respectively, the Dehn function and the isodiametric function of $G$? \end{prob}

\subsection{Filling Length functions and the space complexity}\label{s:flfatsc}

If $G=\la X\mid R\ra$, $w=1$ in $G$, then there is a sequence of transformations

\begin{equation} \label{eq1} w\to w_1\to\ldots\to 1\end{equation} where at every step we either insert a relation from $R^{\pm 1}$, insert or delete a subword of the form $xx\iv$, $x\in X^{\pm 1}$ \cite{MO}. The area of $w$ is equivalent to the minimal length of any such sequence. We can also estimate the space needed by a Turing machine in order to recognize words which are equal to 1 in $G$ by looking at the sizes of the intermediate words $w_i$. The minimum for all sequences (\ref{eq1}) of the maximal length of $w_i$ is called the {\em \index{Filling length of a word}filling length} of $w$. One can follow Gromov \cite{GrHyp} and define the {\em \index{Filling length (FL) function}filling length function} $\mathrm{FL}(l)$ of a group using the filling length of words in the usual way. Gromov \cite[Page 100]{Gr1} noticed that in many finitely presented groups $\mathrm{FL}(l)$ is at most a constant factor of the isodiametric function of $G$, and asked if it is so for all finitely presented groups.

FL is probably the most obvious way to define the space function for a group, but it turned out not to be the most natural function, or, more precisely, unlike in the case of Dehn functions, different natural definition of the space function give drastically different results. In particular, as suggested in Bridson-Riley \cite{BR},  one can allow taking cyclic shifts of words and splitting a word $w_i$ into a pair of words $(u_i,v_i)$ each equal to 1 in the group. The resulting function is called the {\em \index{Fragmenting free filling length (FFFL) function}fragmenting free filling length function} of a group denoted $\mathrm{FFFL}(l)$ \cite{BR} or the {\em \index{Space function of a group} space function} of a group \cite{OlFFFL}. Both Gromov's filling length function FL and Bridson-Riley fragmented filling length function FFFL have nice geometric interpretation in terms of the transformation of van Kampen diagrams. Computing the  filling length amounts to removing cells of the diagram one by one without changing the base point on the boundary and looking at the lengths of the boundaries of the resulting diagrams. When computing fragmenting free filling length function of a group, we are also allowed to change the base point, and divide a diagram into two subdiagrams (and then taking the sums of lengths of the boundaries of the pieces). It is proved in \cite{BR} that these functions behave differently for some finitely presented group  $G$, for instance, the function $\mathrm{FFFL}$ can grow linearly while $\mathrm{FL}$ has exponential growth. Since cutting van Kampen diagrams is the main tool in studying them, the function $\mathrm{FFFL}$ seems to be the most natural group theory analog of the space function of a Turing machine (see Olshanskii \cite{OlFFFL}). As in the case of Dehn functions, the function $\mathrm{FFFL}$ does not depend on the presentation of a group (up to equivalence), and for every finitely presented group $G$ there exists a non-deterministic Turing machine recognizing the word problem in $G$ whose space function is equivalent to $\mathrm{FFFL}$ (this ``diagram eating" Turing machine is essentially described in \cite[Section 3, Proof of Theorem 1.1]{SBR}). Thus the function $\mathrm{FFFL}$ is in general closer to the space complexity function of the word problem of a finitely presented group than $\mathrm{FL}$.

Note that similarly to Theorem \ref{2}, but easier, one can establish the following

\begin{prop}[Olshanskii \cite{OlFFFL}] \label{propos} The space function of a finitely presented group $G$ is equivalent to the space function of a non-deterministic Turing machine. The language accepted by this machine coincides with the set of words equal to $1$ in the group. \end{prop}

\subsection{Two examples of Dehn functions}\label{te}

\subsubsection{Two concrete groups.}\label{ss:tcg} We shall give two standard examples of group presentations where the Dehn function is computed explicitly.

\begin{ex} \label{example1} The Dehn function of $\Z^2=\la a,b\mid ab =ba\ra$ is at most quadratic (in fact it is exactly quadratic by, say,  \ref{s:tdfogaoms} or \ref{ss:tacatcp}, but we shall prove only the upper bound).
\end{ex}

\begin{ex} \label{example2} The Dehn function of the {\em \index{Group!Baumslag-Solitar $\mathrm{BS}(1,2)$}Baumslag-Solitar group} $\mathrm{BS}(1,2)=\la a, b \mid bab\iv  =a^2\ra$ is at least exponential (in fact it is exactly exponential by, say, J. Groves and S. Hermiller \cite{GH}, but we shall prove only the lower bound).
\end{ex}

\subsubsection{HNN-extensions} \label{hnne} Both examples are HNN-extensions (of the cyclic group). Recall \cite{LS} that if $G$ is a group, $A,B$ are isomorphic subgroups of $G$, and $\phi\colon A\to B$ is an isomorphism, then the {\em HNN-\index{HNN-extension}extension} of $G$ with free letter $t$ and associated subgroups $A,B$ is the group $$\HNN(G;A,B,\phi)=\la G, t\mid tat\iv =\phi(a) \hbox{  for every } a\in A.\ra$$
That is the presentation of that group is obtained by adding letter $t$ to the generating set, keeping the defining relations of $G$, and adding the \index{HNN-relation}{\em HNN-relation} $tat\iv=\phi(a)$ for every $a\in A$. In fact it is enough to add HNN-relation for every generator of $A$. Hence if $G$ is finitely presented, $A, B$ are finitely generated, then $\HNN(G;A,b,\phi)$ is finitely presented. We shall also use {\em \index{HNN-extension!multiple}multiple HNN-extensions}. These correspond to a group $G$, and a collection of pairs of isomorphic subgroups $A_i,B_i$, $i=1,\ldots,n$ and isomorphisms $\phi_i\colon A_i\to B_i$. The corresponding multiple HNN-extension of $G$ has the following presentation
$$\begin{array}{ll}\HNN(G;(A_i),(B_i),(\phi_i))= & \la G, t_1\ldots,t_n\mid t_iat_i\iv = \phi_i(a) \\ & \hbox{ for every } a\in A_i, i=1,\ldots,n\ra.\end{array}$$

A cell corresponding to an HNN-relation looks like a quadrangle (see Figure \ref{quad}) with edges labeled by $t_i$ on the opposite sides pointing the same way; $a$ and $\phi_i(a)$ there are words representing the corresponding elements of $A_i$ and $B_i$ respectively.

\begin{figure}[H]

\unitlength .7mm 
\linethickness{0.4pt}
\ifx\plotpoint\undefined\newsavebox{\plotpoint}\fi 
\begin{picture}(75,34.75)(0,0)
\put(27.25,31.75){\line(1,0){47.75}}
\put(75,31.75){\line(-1,-3){7.5}}
\put(67.5,9.25){\line(-1,0){35}}
\multiput(33,9.25)(-.033707865,.127808989){178}{\line(0,1){.127808989}}
\put(26,19.75){\makebox(0,0)[cc]{$t_i$}}
\put(73,18.75){\makebox(0,0)[cc]{$t_i$}}
\put(49,6){\makebox(0,0)[cc]{$a$}}
\put(49.25,34.75){\makebox(0,0)[cc]{$\phi_i(a)$}}
\put(30.25,20.25){\vector(1,-4){.07}}\multiput(29.25,24.25)(.03333333,-.13333333){60}{\line(0,-1){.13333333}}
\put(71.375,20.375){\vector(-1,-3){.07}}\multiput(72.75,24.5)(-.03353659,-.10060976){82}{\line(0,-1){.10060976}}
\end{picture}

\caption{An HNN-cell} \label{quad}
\end{figure}
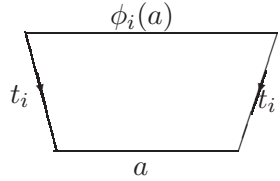

If $\Delta$ is a van Kampen diagram over this presentation, and $\pi$ is an HNN-cell inside it, then each edge of $\pi$ labeled by $t_i$ must be an edge of another HNN-cell, hence HNN-cells with $t_i$-edges form {\em \index{Band}$t_i$-bands} (also called {\em $t_i$-corridors}), that is sequences of cells $\pi_1,\ldots, \pi_k$ where every two consecutive cells share a $t_i$-edge. Bands are the main tools for studying van Kampen diagrams over multiple HNN-extensions.

The following basic lemma was essentially proved in \cite{MS}.

\begin{lemma} \label{ms} Let $\Delta$ be a diagram over an HNN-extension presentation. Suppose that $\Delta$ is {\em \index{Van Kampen diagram!minimal}minimal}, i.e. has the smallest number of HNN-cells among all diagrams with the same boundary label. Then no $t_i$-band is an annulus (i.e. one of the $t_i$-edges of the first cell in the band does not coincide with any other $t_i$-edge of the band).
\end{lemma}
\proof  Indeed, suppose a $t_i$-band $\ttt$ is an annulus. Suppose for convenience that the $t_i$-edges of that band point inside the annulus (the other case is similar). We can also assume that the subdiagram $\Delta'$ bound by the inner boundary component of $\ttt$ is a smallest (under inclusion) among all subdiagrams bounded by inner boundaries of $t_j$-annuli. Then $\Delta'$ cannot contain any $t_j$-annuli. Since the boundary of $\Delta'$ does not contain any $t_j$-edges, it cannot contain any HNN-cells except for the cells from $\ttt$. Hence, by the van Kampen lemma (more precisely by Proposition \ref{prop1}) the boundary label of $\Delta'$ is equal to 1 in $G$. Since $\phi_i$ is an isomorphism, then the label of the outer boundary component of $\ttt$ is equal to 1 in $G$. Hence we can replace the annulus $\ttt$ together with $\Delta'$ by a van Kampen diagram over $G$ (by Lemma \ref{lvk}). This reduces the number of HNN-cells in the diagram without changing the boundary label, a contradiction. \endproof

Lemma \ref{ms} immediately implies the main property of HNN-extensions

\begin{cy}\label{hnn} The group $G$ naturally embeds into $\HNN(G; (A_i), (B_i), (t_i))$.
\end{cy}

\proof Suppose that the natural homomorphism from $G$ to the HNN-extension $$\HNN(G;(A_i), (B_i), (t_i))$$ has a word $w$ in the kernel. Then by Lemma \ref{lvk} there exists a van Kampen diagram $\Delta$ over the presentation of $\HNN(G;(A_i), (B_i), (t_i))$ with boundary label $w$. We can assume that $\Delta$ is minimal. By Lemma \ref{ms} it does not contain $t_i$-annuli. Hence every maximal $t_i$-band in $\Delta$ must start and end on the boundary of $\Delta$. But the boundary of $\Delta$ does not have $t_i$-edges. Hence $\Delta$ does not have HNN-cells, and it is a diagram over the presentation of $G$. By Proposition \ref{prop1}, $w$ is 1 in $G$. Thus the kernel of the natural homomorphism is trivial.
\endproof

\subsubsection{Proofs of Examples \ref{example1}, \ref{example2}}\label{ss:poe}
Now Examples \ref{example1}, \ref{example2} are easy to prove.

{\bf Example \ref{example1}.} The group $\Z^2$ can be considered as an HNN-extension in two different ways: we can consider $a$ or $b$ as a free letter (indeed $ab=ba$ is equivalent to $aba\iv =b$ and to $bab\iv =a$). Hence we can consider $a$-bands and $b$-bands. Let $\Delta$ be a van Kampen diagram over the presentation of $\Z^2$ with minimal number of cells among all diagrams with the same boundary label.  It is easy to see (the proof is similar to the proof of Lemma \ref{ms}) that every $a$-band can intersect every $b$-band in $\Delta$ only once.
Since there are no $a$- and $b$-annuli in $\Delta$ (by Lemma \ref{ms}), every maximal $a$- and $b$-band starts and ends on the boundary of $\Delta$. Hence if $l$ is the length of the boundary of $\Delta$, then there are at most $l/2$ maximal $a$-bands and at most $l/2$ maximal $b$-bands in $\Delta$. Since every cell in $\Delta$ is the intersection of an $a$-band and a $b$-band, the area of $\Delta$ is at most $l^2/4$ and the Dehn function is at most quadratic.

{\bf Example \ref{example2}.} We are going to use the fact that the cyclic subgroup $\la a \ra$ is exponentially distorted in the Baumslag-Solitar group $\mathrm{BS}(1,2)$: $b^nab^{-n}=a^{2^n}$ (so the element $a^{2^n}$ has linear in $n$ length in $\mathrm{BS}(1,2)$ and exponential length in $\la a \ra$).

Consider a diagram $\Delta$ over the presentation of the Baumslag-Solitar group with boundary label $ab^{-n}a\iv b^na\iv b^{-n} a\iv b^n$ as on Figure \ref{bsex}. For simplicity suppose that $n$ is odd. Again we can assume that $\Delta$ has the smallest number of cells among all diagrams with the same boundary label.

\begin{figure}[H]
\unitlength .7mm 
\linethickness{0.4pt}
\ifx\plotpoint\undefined\newsavebox{\plotpoint}\fi 
\begin{picture}(80.25,106.25)(0,0)
\put(36.625,28.125){\vector(1,-3){.07}}\multiput(29,53.25)(.0337389381,-.1111725664){452}{\line(0,-1){.1111725664}}
\put(50.125,3){\vector(1,0){.07}}\put(44.25,3){\line(1,0){11.75}}
\put(75.25,53.375){\vector(1,0){.07}}\multiput(70.25,53.5)(1.25,-.03125){8}{\line(1,0){1.25}}
\put(49.625,53.375){\vector(1,0){.07}}\multiput(29.25,53.25)(5.09375,.03125){8}{\line(1,0){5.09375}}
\put(45.5,78.125){\vector(1,3){.07}}\multiput(35.75,53.5)(.0337370242,.0852076125){578}{\line(0,1){.0852076125}}
\put(60.875,102.75){\vector(1,0){.07}}\put(55.25,102.75){\line(1,0){11.25}}
\put(73.25,78.125){\vector(-1,4){.07}}\multiput(80.25,53.25)(-.0337349398,.1198795181){415}{\line(0,1){.1198795181}}
\put(48.75,6.25){\makebox(0,0)[cc]{$a$}}
\put(31.5,25.5){\makebox(0,0)[cc]{$b^n$}}
\put(70.75,26){\makebox(0,0)[cc]{$b^n$}}
\put(50.25,49){\makebox(0,0)[cc]{$a^{2^n}$}}
\put(40.5,74.75){\makebox(0,0)[cc]{$b^n$}}
\put(76.5,79){\makebox(0,0)[cc]{$b^n$}}
\put(61.5,106.25){\makebox(0,0)[cc]{$a$}}
\put(63.25,28.25){\vector(-1,-3){.07}}\multiput(70.75,53.25)(-.0337078652,-.1123595506){445}{\line(0,-1){.1123595506}}
\put(32.5,56.75){\makebox(0,0)[cc]{$a$}}
\put(76,49.75){\makebox(0,0)[cc]{$a$}}
\qbezier(42.75,71.25)(47.5,56.25)(33.25,39.25)
\qbezier(44.25,74.75)(50.875,52.25)(34,37.75)
\put(35,34){\line(1,0){29.75}}
\put(35.75,31){\line(1,0){28.25}}
\put(39.25,34){\line(0,-1){3}}
\put(48,34){\line(0,-1){2.75}}
\put(57,34.25){\line(0,-1){3.25}}
\put(61.5,33.75){\line(0,-1){2.75}}
\put(43.5,34){\line(0,-1){3}}
\put(52.5,34){\line(0,-1){3.25}}
\multiput(36.25,43)(.03333333,-.04444444){45}{\line(0,-1){.04444444}}
\multiput(38.75,46.75)(.03846154,-.03365385){52}{\line(1,0){.03846154}}
\multiput(40.25,49.75)(.04326923,-.03365385){52}{\line(1,0){.04326923}}
\put(43.25,57){\line(1,0){2.5}}
\put(44,62){\line(1,0){2.25}}
\multiput(43.75,66.25)(.28125,.03125){8}{\line(1,0){.28125}}
\multiput(43.25,69)(.03333333,.03333333){60}{\line(0,1){.03333333}}
\end{picture}
\caption{A diagram with boundary label $ab^{-n}a\iv b^na\iv b^{-n} a\iv b^n$ over the presentation of the Baumslag-Solitar group $\mathrm{BS}(1,2)$.}\label{bsex}
\end{figure}
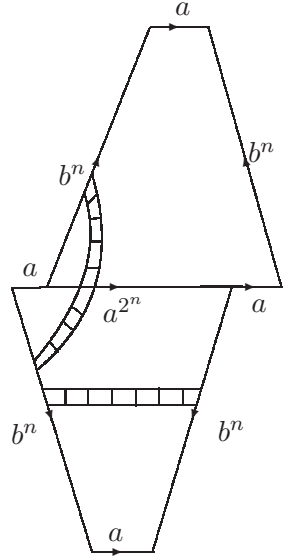

Since $b$ is a free letter, we can consider $b$-bands in $\Delta$. The length of the boundary of $\Delta$ is $4n+4$.
Consider the $b$-band that starts at a $b$-edge on one of the 4 sides labeled by $b^n$.  It must end on one of the other two sides labeled by $b^n$ where the direction of $b$-edges is opposite (we trace the boundary counterclockwise).  Since $b$-bands do not intersect, the $b$-band $\bbb$ staring at the middle $b$-edge (here we use the fact that $n$ is odd) must end at the middle $b$-edge of the other side of the diagram (see Figure \ref{bsex}).  It is not difficult to show that the $b$-band can only be horizontal (so only one of the two possibilities on Figure \ref{bsex} can occur). Let $u$ be the label of the shorter side $p$ of the $b$-band labeled by a power of $a$. Then the path $p$  cuts off a subdiagram $\Delta'$ with boundary label $$b^{(n-1)/2}ab^{-(n-1)/2}u\iv.$$ By Lemma \ref{lvk}, that word must be equal to 1 in $\mathrm{BS}(1,2)$. Note that $$b^{(n-1)/2}ab^{-(n-1)/2}=a^{2^{(n-1)/2}}$$ in that group. Hence $u=a^{2^{(n-1)/2}}$. By Corollary \ref{hnn}, then $u$ has length at least $2^{(n-1)/2}$. But that means that the length of $\bbb$ is at least $2^{(n-1)/2}$, hence the area of $\Delta$ is at least $2^{(n-1)/2}$ while the perimeter is linear in $n$. Since $\Delta$ is a minimal diagram, the Dehn function of the $\mathrm{BS}(1,2)$ is at least exponential.

\subsection{Isoperimetric functions of geodesic metric spaces. Different area functions.}\label{s:ifogmsdaf}

The notion of the area of a loop is well defined in the setting of Riemannian manifolds as well
as in finitely presented groups (see for instance \cite[Chapter 1, Section 8.1.4]{BH}). It also can be defined for geodesic metric spaces. In fact there are several kinds of definitions of area in that case. Below, we shall mention the definition of Bowditch \cite{Bow1}, \cite[Sections 2.3, 5]{Bow2}, Gromov's definition from of coarse area function from \cite[5.F]{Gr1} and a definition using metric currents \cite{AK,Wen1, Wen2}. In all these definitions $(X,\dist)$ is a geodesic metric space. By a {\em \index{Loop}loop} in $X$ we shall always assume a Lipschitz (hence rectifiable) map from the unit circle $S^1$ to $X$. Let $\Omega$ be the set of all loops in $X$. Each definition of an area defines a function $A\colon \Omega\to \R_+$. Then we define the corresponding {\em isoperimetric function} as the function $A(l)$ such that $A(l)=\max\{A(\gamma) \mid \gamma\in \Omega, |\gamma|\le l\}$. Sometimes we shall call any function $f\succ A$ also an {\em isoperimetric function} of $X$.

\subsubsection{The definition of Bowditch}\label{ss:tdob}

For every curve $\alpha\colon [0,1]\to X$, by $-\alpha$ we denote the curve $\alpha(1-t)$ (that is the same curve traced in the opposite direction). Let $\alpha_0, \alpha_1, \alpha_2$ be three Lipschitz curves $[0,1]\to X$ connecting points $A, B\in X$. Then we can form three loops $\alpha_i\cup(-\alpha_{i+1})$ where $i=0,1,2$, and $+1$ is understood modulo 3. Whenever three loops are obtained this way, one says that they form  a {\em \index{$\theta$-loop}$\theta$-loop}. Now let $\Omega$ be the set of all loops in $X$ \footnote{In fact it is enough to assume that $\Omega$ is a set of loops closed under the operation of cutting a loop into two loops by a geodesic segment connecting two points on the loop. Thus the area function can be defined even for some non-simply connected spaces.}. Let $A$ be a function from $\Omega$ to $\R_+$ satisfying the following two conditions

\begin{itemize}
\item[($B_1$)][The triangle inequality for $\theta$-loops.] For every three loops $\gamma_1$,$\gamma_2,$ $\gamma_3\in \Omega$ forming a $\theta$-loop, we have $A(\gamma_1)\le A(\gamma_2)+A(\gamma_3)$.

\item[($B_2$)][The quadrangle inequality.] Suppose that $\gamma\in\Omega$ is split into 4 subpaths $\alpha_1$, $\alpha_2$, $\alpha_3$, $\alpha_4$. Then $A(\gamma)\ge kd_1d_2$ where $d_1=\dist(\alpha_1, \alpha_3)$ and $d_2=\dist(\alpha_2, \alpha_4)$, $k$ is a constant.
\end{itemize}

\subsubsection{The coarse definition of area function by Gromov.}\label{ss:tcdoafbg}
A {\em \index{$\delta$-filling of a loop}$\delta$-filling} of a loop $\gamma$ is
a pair consisting of a triangulation of the planar unit disk $D^2$ and of an injective map $\psi$
from the set of vertices of the triangulation to $X$ such that the restriction of  $\psi$ to the points on $\partial D^2=S^1$ coincides with $\gamma$. The image of the map $\psi$ is called a {\em \index{Filling disk of a loop}filling disk of} $\gamma$. We can join the images of the vertices of each
triangle of the triangulation of $D^2$ by geodesics. If two of these vertices are on $\partial D^2$,  we replace the geodesic
by the arc of $\gamma(S^1)$ connecting these points. Thus we obtain a finite number of triangles (some sides are geodesics, some sides are arcs in $\gamma(S^1)$.  These triangles are called
{\em \index{Brick}bricks}. The {\em \index{Brick!length of}length of a brick} is the sum of the lengths of its sides.
The maximum of the lengths of the bricks in a partition
is called the {\em \index{Mesh of a partition}mesh} of the partition. The partition is called $\delta$-filling partition of $\gamma$ if its mesh is at
most $\delta$. The corresponding filling disk is called $\delta$-{\em filling disk} of $\gamma$. The  $\delta$-area of $\gamma$ is the
minimal number of triangles in a triangulation associated to a $\delta$-filling partition of $\gamma$. We denote
it by $A_\delta(\gamma)$. If no $\delta$-filling partition of the loop $\gamma$ exists, we put
$A_\delta(\gamma)=\infty$. Given the definition of a $\delta$-area of a loop, we can naturally define the $\delta$-isoperimetric function $A_\delta(l)$ as the supremum of $\delta$-areas of all loops of length at most $l$. Note that $A_\delta(l)$ may be equal $\infty$ even if the $A_\delta(\gamma)<\infty$ for every $\gamma$. The corresponding isoperimetric function will be denoted by $A_\delta(l)$. We shall always assume that our space $X$ is {\em \index{$\mu$-simply connected metric space}$\mu$-simply connected} for some $\mu>0$, that is $A_\mu(l)<\infty$ for every $l$. In this case all functions $A_\delta$, $\delta>\mu$ are equivalent (see \cite{DrutuIJAC}).
Note that if two metric spaces $X$ and $Y$ are quasi-isometric, and $\mu$-simply connected, then their $A_\delta$-functions are equivalent for $\delta>\mu$.

Note that the coarse area function $A_\delta$ satisfies the Bowditch conditions ($B_1$), ($B_2$).

\subsubsection{The metric currents definition}\label{ss:tdok}

This definition was used to obtain remarkable results in \cite{Wen1,Wen2}. Since it is more complicated than the other two definitions,  we present only some ideas it is based on, referring the reader to \cite{AK,BZ}.
First let $M$ be an $n$-dimensional Riemannian manifold. Let $D^2$ be a 2-dimensional disc in $\R^2$ equipped with the Euclidean coordinate system, and $f$ be a differentiable map from $D^2$ to $M$. Then for every $x\in D^2$ we have two derivatives $f_x=(a_1,\ldots,a_n)$ and $f_y=(b_1,\ldots,b_n)$, vectors in the tangent space $T_{f(x)}(M)$ corresponding to the two basic directions in $D^2$. Then the {\em \index{Jacobian of a map}Jacobian} of $f$ at $x$ is the number $$J(x)=\sum_{i<j} \det \left(\begin{array}{cc} a_i & a_j\\ b_i & b_j\end{array} \right)$$ (see \cite[Section 15]{BZ}) Integrating $J(x)$ over $D$ gives the {\em area} of the map $f$. If $f$ is not differentiable but only Lipschitz, then by the classical Rademacher's theorem (see \cite{BZ}),  $f$ is differentiable almost everywhere with respect to the Lebesques measure, so $J(x)$ is defined for almost all $x$, and we still can integrate $J(x)$ over $D$. So areas are defined for Lipschitz maps too. This allows one to define a {\em \index{Lipschitz area function}Lipschitz} area function $L(n)$ in every Riemannian manifold $M$.

Now if $M$ is not a Riemannian manifold but just a metric space, and $f$ is a Lipschitz map from $D^2$ to $M$ then in order to define an area of $f$, we need to use, for example, the Hausdorff measure (one can use other measures as well, such as the inner Hausdorff measure of Busemann \cite{BZ} or mass* measure of Gromov \cite[Section 4.1]{GrFil}). Recall, that if $E$ is a subset of $M$, then the 2-dimensional {\em \index{Hausdorff measure}Hausdorff} measure of $E$ is, up to a scalar multiple, the limit as $\varepsilon\to 0$ of the infimum of sums $H_2(E)=\sum \mathrm{diam}(E_i)^2$ for all covers of $E$ by countably many balls $E_i$ of diameters $\le \varepsilon$. Then the {\em area} of the map $f$ is the $H_2$-measure of $f(D^2)$. Note that for Riemannian manifolds $M$ this definition coincides (up to a constant factor) with the previous definition.

Unfortunately this definition behaves badly if we try to estimate area functions of limits of metric spaces, say, asymptotic cones: the area of a limit is not the limit of areas. In order to overcome that we need to further generalize the Lipschitz maps.

A {\em \index{Singular Lipschitz chain}singular Lipschitz chain} is a finite formal combination $C=\sum_i m_if_i$ where $f_i$ is a Lipschitz map from some region in $\R^2$ to  the metric space $X$. The {\em area} of a chain $C$ is $\sum m_i\area(f_i)$ (see \cite{GrFil}). The boundary of one Lipschitz map $f\colon D^2\to X$ is the restriction of $f$ to $\partial D^2$. The boundary of a linear combination is the linear combination of boundaries which can be viewed as an element of the first homology group of $X$. Thus we can define a {\em (homological) area} of a loop $\gamma$ as the smallest area of a singular Lipschitz chain with boundary $\gamma$.\footnote{Note that a homological version of the Dehn function of groups can be defined in a similar way \cite{G2}. The homological Dehn function is easier to compute, and it bounds the ordinary Dehn function from below. It has been studied, for example, in \cite{BMS}.}

One can introduce a metric on the space of all singular Lipschitz chains from $\R^2$ to $X$. It can be, say, the {\em flat metric} (the definition goes back to Witney \cite{Wit}), similar to the well-known Gromov-Hausdorff metric \cite{GrBook}. Then one can complete the space of singular Lipschitz chains with respect to this metric and obtain the space of {\em 2-dimensional \index{Integral metric current}integral metric currents}. The area of an integral metric current is the limit of areas of the corresponding singular Lipschitz chains. Thus one can view 2-dimensional integral metric currents as infinite linear combinations of Lipschitz maps from $D^2$ into $X$.

\subsection{Cayley complexes and metric spaces}\label{ccms} Let $\pp=\la X\mid R\ra$ be a presentation of a group $G$ with $X$ finite, $X=X\iv$. Consider the {\em \index{Cayley graph} Cayley graph} $\Gamma=\Gamma_X(G)$ (i.e. $G$ as the set of vertices and $\{(g,gx)\mid  x\in X, g\in G\}$ as the set of edges. If we label every edge of $\Gamma$ by the corresponding element $x\in X$, we get a {\em \index{Cayley graph!labeled}labeled Cayley graph}. The graph is directed but each edge $(g, gx)$ has the inverse $(gx,x)$. Thus when we travel from $g$ to $gx$, we read $x$, when we travel from $gx$ to $g$, we read $x\iv$.

Every group word $w$ in $X$ labels unique path in the labeled Cayley graph starting from any vertex $g\in G$. It ends at the vertex $gw$. In particular the equality $w=1$ is true in $G$ if and only if that path labels a loop starting from any vertex of $G$. For every $r\in R$, let us consider all loops $\gamma$ in $\Gamma$ labeled by $r$ and let us glue in a disc $D_\gamma$ to $\Gamma$ with $\partial D_\gamma=\gamma$. This way we obtain an (edge-labeled) CW-complex $C\Gamma(\pp)$, the {\em \index{Cayley complex}Cayley complex} of $G$ corresponding to the presentation $\pp$. Note that every simplicial loop in that complex (i.e. loop consisting of edges of $\Gamma$) is labeled by a group word $W$ in $X$ which is equal to 1 in $G$. By the van Kampen lemma, there exists a van Kampen diagram over $\pp$ with boundary label $W$. That van Kampen diagram is a tesselated disc $D^2$, and the labels of its edges give us a map from $D^2$ to the labeled graph $C\Gamma(\pp)$ such that the image of $\partial D^2$ is $\gamma$. This immediately implies that,

\begin{itemize}
\item if $\pp$ is finite then $A_\mu(l)$ for $C\Gamma(\pp)$ is always finite for, say, $\mu=\max\{|r|\mid r\in R\}$, i.e. $C\Gamma(\pp)$ $\mu$-simply connected;
\item the Dehn function of $G$ is equivalent to $A_\delta$ for every sufficiently large $\delta$ and satisfies Bowditch's conditions ($B_1$), ($B_2$) with $\frac1k$ equal to the square of the maximal length of a defining relator.
\end{itemize}

Notice the following important topological properties of the Cayley complex.

\begin{enumerate}
\item[($C_1$)] The Cayley complex is simply connected.

\item[($C_2$)] The group presentation is finite if and only if the corresponding Cayley complex is locally compact.
\end{enumerate}

\subsection{The Dehn functions of groups acting on metric spaces.}\label{s:tdfogaoms}

Geometric group theory studies groups acting by isometries on metric spaces. We say that the action of $G$ on $X$ is {\em \index{Action! geometric}geometric} if it is co-compact, and properly discontinuous. Here {\em \index{Action!co-compact} co-compact} means that there exists a compact subset $Y\subset X$ such that $X=G\cdot Y$; {\em \index{Action!properly discontinuous}properly discontinuous} means that for every point $x\in X$ there exists a neighborhood $U\ni x$ such that $y\cdot U\cap U=\emptyset$ for every $y\in G\setminus \{1\}$. For example the fundamental group of a compact Riemannian manifold $M$ acts geometrically on the universal cover of $M$.

It is easy to observe (see Gromov \cite{Gr1}) that if a finitely presented group acts geometrically on a simply connected metric space with the coarse space function $A_\delta$ (see \ref{ss:tcdoafbg}), then the Dehn function of $G$ is equivalent to $A_\delta$.

For Riemannian manifolds, the fact that the Riemannian area function (see \ref{ss:tdok}) is equivalent to $A_\delta$ is not obvious, although it is a part of ``folklore". It was proved in all details by Bridson in \cite{Br}.

\begin{theorem}[Bridson \cite{Br}] If a finitely presented group acts geometrically on a simply connected Riemannian manifold, then the Dehn function of the group and the corresponding filling function of the manifold are equivalent.
\end{theorem}

Since $\Z^2$ acts geometrically on the Euclidean plane $\R^2$ by translations, and, as we learned in Real Analysis (and ancient Greeks knew without any Real Analysis), the area function of $\R^2$ is $\frac{l^2}{4\pi}$, we finally deduce that the Dehn function of $\Z^2$ is quadratic.

\subsection{Asymptotic cones of metric spaces and groups}\label{s:acomsag}

\subsubsection{The definition}\label{ss:td2} Let $X$ be a metric space with distance function $\dist$, say, a Cayley graph of a finitely generated group. Consider a non-decreasing  sequence of numbers $d=(d_i)$ with $\lim d_i=\infty$, called {\em \index{Scaling constants}scaling constants} and the sequence of metric spaces $X/d_i$ where $W/n$ denotes the space $W$ with distance function $\frac{\dist}{n}$ ($\dist$ is the distance function in $W$). We want to define a limit of this sequence of spaces. One of the reasons  is that we want to ``forget" about local properties of $X$ in favor of global properties. In particular, we want to study properties of configurations of points in $X$ which are satisfied by configurations with arbitrary large pairwise distances between points, and do not want to be disturbed by properties which only hold for configurations of small diameter. The limit space is supposed to have the global properties of $X$ while the local properties may disappear. One way to define a limit was proposed by Gromov in \cite{Gromov}. The limit is called the Gromov-Hausdorff limit but that definition can only be applied when the spaces $X/d_i$ are uniformly locally compact which is true in the case of Cayley graphs of groups of polynomial growth but not true in general. A much more universal definition (which coincides with Gromov's when his definition applies) was given by van der Dries and Wilkie \cite{DW}. The limit is called an {\em \index{Asymptotic cone}asymptotic cone of} $X$ and is defined as follows. In fact we shall define a limit of any sequence of spaces $(X_i,\dist_i)$. As usual for definitions of boundaries and completions of different sorts, the limit will consist of sequences $(x_i)$, $x_i\in X_i$. Thus consider the direct product $Z=\Pi X_i$. The most obvious way to define a distance function on $Z$ is coordinate-wise:

$$\dist((x_i),(y_i))=\lim \dist_i(x_i,y_i).$$
The problem is that the limit may not exist (we ignore the issue that the limit can be $\infty$ for now). To circumvent this problem, let us pick a {\em \index{Non-principle ultrafilter}non-principle ultrafilter} on $\N$, that is a set $\omega$ of substsets of $\N$ which is closed under intersections, super-sets, does not contain finite subsets, and with every set $U\subset \N$ either $U$ or $\N\setminus U$ is in $\omega$. The sets from $\omega$ are called {\em \index{Big set}big}, the other sets are {\em \index{Small set}small}. We can also view $\omega$ as a finitely additive measure on the set of all subsets of $\N$ which has only two values 0, 1, and $\omega(\N)=1$ while $\omega(\emptyset)=0$. If $A_i$ is a sequence of statements and $A_i$ holds for all $i$ from a big subset, we say that $A_i$ holds {\em \index{$\omega$-almost surely}$\omega$-almost surely}.

To prove the existence of ultrafilters one needs the Axiom of Choice (or, more precisely, a slightly weaker hypothesis, the so called {\em Boolean prime ideal theorem} \cite{CN}), so although we believe that ultrafilters exist, one cannot actually define an explicit ultrafilter using only axioms of Zermelo and Frenkel. Anyway, as soon as we fix an ultrafilter $\omega$, we can modify the definition of the limit of a sequence of numbers, introducing the definition of the $\omega$-limit. It is almost exactly the same as the standard calculus definition: if $b_i$ is a sequence of real numbers, then

\medskip

$\lim^\omega(b_i)=b$ means that for every $\varepsilon>0$ the set of indices $i$ such that  $|b_i-b|<\varepsilon$ is big.

\medskip

So the only difference with the standard definition of a limit is that we replaced the words ``for almost all indices $i$" by ``the set of indices $i$ is big". The fact that $\omega$ is an ultrafilter immediately implies - exercise - that any sequence of numbers has unique $\omega$-limit (it may be equal to $\pm\infty$). For example the limit of the sequence $1,0,1,0,1,\ldots $ is 0 if the set of even numbers is big, and 1 otherwise.

Now we have a ``distance" function on the set $\Pi X_i$: $\dist((x_i), (y_i))=\lim^\omega\dist_i(x_i,y_i)$. Clearly the triangle inequality is satisfied. Since this function can take infinite values, we consider a ``connected component". Namely, pick a point $o=(o_i)$ in $\Pi X_i$, and let $\Pi_b X_i$ be the set of all points at finite distance from $o$. This restriction of the ``distance function" to $\Pi_b X_i$ is a quasi-distance, because different points can be at distance 0. For example, if two points $(x_i)$, $(y_i)$ are such that $x_i=y_i$ for all $i$ from a big set, then the distance between these points is 0. The relation $(x_i)\sim (y_i)$ iff $\dist((x_i), (y_i))=0$ is an equivalence relation and the quotient $\Pi_b X_i/\sim$ is a metric space with the induced metric. That metric space is called the {\em \index{$\omega$-limit}$\omega$-limit} of $X_i$ corresponding to the observation point $(o_i)$.

\begin{df} The {\em asymptotic cone} of a metric space $X$ corresponding to an ultrafilter $\omega$,  an observation point $o=(o_i)$, $o_i\in X$, and a sequence $d_i$ as above is the $\omega$-limit of spaces $X/d_i$ corresponding to the observation point $o$. It is denoted by $\Con^\omega(X, (d_n),o)$.
\end{df}

Asymptotic cones of a group endowed with some metric (say, a finitely generated group with the {\em \index{Word metric on a group}word metric}\footnote{I.e. the metric on the Cayley graph corresponding to a finite generating set.} or a Lie group with Riemanian metric) are asymptotic cones of the corresponding metric space.

\subsubsection{Some properties of asymptotic cones}\label{spoac} Here are some basic properties of asymptotic cones that can be found in \cite{DW}, \cite{Gr1}, \cite{DrutuIJAC} and other papers. Some of these properties are quite easy to prove, for other properties we provide references.

\begin{enumerate}
\item \label{1} If the metric space $X$ is homogeneous (say, $X$ is the Cayley graph of a group), then asymptotic cones of $X$ do not depend on the choice of the observation point (so we shall always assume, if $X$ is a Cayley graph, that the observation point is $(e)$ and we shall omit it from the notation for an asymptotic cone of a group).

\item\label{l2} If two metric spaces $X$ and $Y$ are quasi-isometric, then their asymptotic cones corresponding to the same ultrafilters, same scaling constants and appropriately chosen observation points are bi-Lipschitz equivalent, hence homeomorphic. In particular, the asymptotic cones (with the same parameters) corresponding to two Cayley graphs of the same finitely generated group but different finite generating sets are bi-Lipschitz equivalent.
\item\label{3} For every group $G$ the set $\Pi_b G$ is a group (with coordinate-wise multiplication) acting transitively on any asymptotic cone $\Con^\omega(G,(d_i))$ by left multiplication.
\item\label{4} Any asymptotic cone is a complete metric space.
\item\label{5} For every sequence of paths $p_i\colon [0,l_i]\to X$ parameterized by their lengths ($l_i>0$), the function $p\colon [0,\lim^\omega l_i]\to \Con^\omega(X,(d_i), (o_i))$ (where $\dist(o_i,p_i(0))=O(d_i)$) is a path parameterized by its length. Moreover if $p_i$ is an $(A,B_i)$-quasi-geodesic (i.e. it is a $(A,B_i)$-quasi-geodesic map from the interval onto its image), $B_i=o(d_i)$, or an $A$-Lipschitz path, then $p$ is an $A$-{\em \index{Bi-Lipschitz path}bi-Lipschitz path} or {\em \index{Lipschitz path}$A$-Lipschitz} path respectively (i.e. an $A$-bi-Lipschitz or an $A$-Lipschitz map from an interval to $X$).
\item\label{6} In particular, the $\omega$-limit of every sequence of geodesic paths in $X$ is a geodesic path in $\ccc=\Con^\omega(X; (d_i), (o_i))$ (it can be infinite or a point or empty). It is not true that every geodesic path in $\ccc$ is an $\omega$-limit of geodesic paths in $X$ but if $\g$ is a geodesic in $\ccc$, then for every $\varepsilon>0$ there exists a piecewise geodesic path $\g'$ with $k=k(\varepsilon)$ geodesic pieces that is an $\omega$-limit of $k$-piecewise geodesic paths in $X$, and $\g'$ (resp $\g$) is in $\varepsilon$-neighborhood of $\g$ (resp. $\g'$) \cite{DS}, i.e. the {\em Hausdorff distance} between $\g, \g'$ is at most $\varepsilon$.
\item\label{7} Suppose that $\ccc= \Con^\omega(X; (d_i), (o_i))$ contains a simple geodesic triangle $T$. Then for every $\varepsilon$, there exists a $k=k(\varepsilon)$ and a sequence of geodesic $k$-gons $\Pi_i$ in $X$ such that $\Pi=\lim^\omega \Pi_i$ is at Hausdorff distance at most $\varepsilon$ from $T$, contains the midpoints of the three sides of $T$, and each $\Pi_i$ is {\em thick}, that is the middle third of each side of $\Pi_i$ is far from the union of the other sides, and every vertex of $\Pi_i$ is far from the vertices not adjacent to it (a precise definition of thick geodesic polygons  can be found in \cite{DS}).

\item\label{8} An asymptotic cone of a finitely generated group may depend on the ultrafilter, and the scaling constants $(d_i)$ \cite{TV, KSTT, DS}. If the Continuum Hypothesis is false, then there exists a group (an uniform lattice in $\mathrm{SL}_n(\R)$) with the set of non-homeomorphic asymptotic cones $2^{2^{\aleph_0}}$. On the other hand if the Continuum Hypothesis is true then the set of all asymptotic cones of all countable metric spaces is of cardinality continuum \cite{KSTT}, and there exists one finitely generated (and recursively presented) group with continuum pairwise non-$\pi_1$-equivalent (hence non-homeomorphic) asymptotic cones \cite{DS}. There are also finitely presented groups with 2 non-$\pi_1$-equivalent asymptotic cones \cite{OS2}.
\end{enumerate}

\subsection{Asymptotic cones and Dehn functions}\label{s:acadf}

We have mentioned that asymptotic cones of fi\-nite\-ly generated groups capture asymptotic properties of groups, i.e. properties that manifest themselves on configurations of elements of arbitrary large diameters. For example, working with hyperbolic groups (that is groups all whose asymptotic cones are $\R$-trees (see \ref{ss:tdaac}) it is convenient to remember that configurations of elements with large diameter in a hyperbolic group behave like points on a tree with similar pairwise distances.

In particular, asymptotic cones of a group reflect the properties of the Dehn function.

\begin{theorem}[Gromov \cite{Gr1}] \label{Gro} Suppose that all asymptotic cones of a group $G$ are simply connected. Then $G$ is finitely presented, has  polynomial isoperimetric function and linear isodiametric function.
\end{theorem}

\begin{proof} The proof of this remarkable statement is very easy. Indeed, if, say, the Dehn function $f$ of  $G$ grows faster than any polynomial, then for every $k\ge 1$ $f$ satisfies inequality $f(n)\ge k f(n/2)$ for infinitely many $n$. For each $k$ let $\gamma_k$ be a loop in the Cayley complex of $G$ such that $\area(\gamma_k)\ge k f(|\gamma_k|/2)$. Let $d_k$ be the length of $\gamma_k$. For some ultrafilter $\omega$, consider the asymptotic cone $\ccc=\Con^\omega(G,(d_k))$, and the $\omega$-limit $\gamma$ of $\gamma_k$ in $\ccc$. Then $\gamma$ has length $1$. Since $\ccc$ is simply connected by assumption, there exists a continuous map $\phi\colon D^2\to\ccc$ with $\partial(g)=\gamma$. Since $\phi$ is continuous, and $D^2$ is compact, $\phi$ is uniformly continuous. Hence there exists an $\varepsilon>0$ and a decomposition of $D^2$ into, say, $n$ triangles of $\Delta_1,\ldots,\Delta_n$ such that the perimeter of each $\phi(\Delta_i)$ is at most $1/3$. By Property \ref{6} of asymptotic cones from \ref{spoac}, it follows, that we can assume that the sides of all curvy triangles $\phi(\Delta_i)$ except for the sides contained in $g(\partial D^2)$, are $\omega$-limits of geodesics from the Cayley graph of $G$. Therefore each $\phi(\Delta_i)$ is an $\omega$-limit of some sequence of loops $\delta_m^i$, $m=1,\ldots$, in the Cayley graph. But that means a disc bounded by the loop $\gamma_m$ can be decomposed into $n$ loops $\delta_m^i$ of length $\le d_m/3+o(d_m)$ $\omega$-almost surely for all $m\ge 1$. Hence $\area(\gamma_m)\le n f(d_m/2)$ $\omega$-almost surely for all $m$, a contradiction.
\end{proof}

Riley \cite{Rie} proved that under the assumptions of Theorem \ref{Gro}, the group also has linear FL function (and hence linear FFFL function as well).

The converse statement of Theorem \ref{Gro} is not true. In \cite{OS2c}, we showed that the group (a multiple HNN-extension of a free group) $$\la a,b,t_1,t_2\mid t_iat_i\iv =ab,  t_ib=bt_i, i=1,2\ra$$ has cubic Dehn function, linear isodiametric function and non-simply connected asymptotic cones. The ``cubic" in that statement can be improved  to as low as $n^2\log n$ by modifying the group using the method from \cite{OSnlogn}. It is impossible, though, to have an example with quadratic Dehn function.

\begin{theorem} [Papasoglu \cite{Pap1}]\label{pap1} All asymptotic cones of a finitely presented group having quadratic Dehn function are simply connected.
\end{theorem}

As Theorem \ref{Gro}, this is a very important fact, but the idea of the  proof is not that difficult. We take a van Kampen diagram $\Delta$ over the presentation of our group $G$, and start removing layers of cells one by one from the outside in, peeling the diagram as an apple, so that each layer is 1-cell wide. The fact that the area is quadratic lets us estimate the number of layers and their lengths. This allows us to decompose the diagram into a constant, say, $c$, number of subdiagrams whose perimeters are at most half of the perimeter of $\Delta$. That, in turn, implies that every loop $\gamma$ in any asymptotic cone of $G$ can be decomposed into $O(1)$ loops of length $|\gamma/2|$. This implies simple connectivity of the asymptotic cone.

An analysis of Gromov's proof of Theorem \ref{Gro} allowed Papasoglu to obtain the following useful result.

\begin{theorem}[Papasoglu \cite{Drutuquadr}] \label{pap2} Suppose that all asymptotic cones of $G$ are simply connected and have Gromov coarse area function $A_\delta(l)$ bounded by $l^{c}$. Then the Dehn function of $G$ is bounded by $n^{c+\varepsilon}$ for every $\varepsilon>0$.
\end{theorem}

That result was used by Dru\c tu in \cite{Drutuquadr} (see also \ref{ss:tacatcp} and  \ref{Dru} below).

Note that Theorem \ref{pap1} does not give metric properties of asymptotic cones, only the topological property of being simply connected. In a simply connected asymptotic cone even of a nice nilpotent group, a Lipschitz loop does not necessarily bound a Lipschitz 2-disc or even an integral 2-current (it can be deduced from \cite{AK1} for the Heisenberg group $H_3$, for other examples see Wenger \cite{Wen2}, and \ref{ss:ocp}). But for groups with quadratic Dehn functions, the asymptotic cones are much nicer.

\begin{theorem}[Wenger \cite{Wen2}]\label{Wen2} Let $G$ be a finitely presented group with quadratic Dehn function. Then
every asymptotic cone of $G$ admits a quadratic isoperimetric inequality in terms of 2-currents, that is for some constant $C$, every Lipschitz loop $\gamma$ bounds an integral 2-current of area $\le C|\gamma|^2$.
\end{theorem}

\section{Isoperimetric functions of important groups}

\subsection{Hyperbolic groups and asymptotic cones}\label{hgac}

\subsubsection{The definition and a characterization}\label{ss:tdaac}

There are several equivalent definitions of $\delta$-hyperbolic (in other terminology, word hyperbolic) metric spaces and groups \cite{GrHyp}. Rips' definition is the following.

\begin{df} A geodesic metric space $X$ is $\delta$-{\em \index{$\delta$-hyperbolic metric space}hyperbolic} for some $\delta\ge0$ if every \index{Geodesic triangle}geodesic triangle\footnote{That is a triangle where all three sides are geodesics.} in  $X$ is $\delta$-{\em \index{Geodesic triangle!thin}thin}, that is every side of it is in the closed $\delta$-neighborhood of the union of the other two sides.  A finitely generated group $G$ is $\delta$-{\em \index{Group!$\delta$-hyperbolic}hyperbolic} if its Cayley graph is $\delta$-hyperbolic. A geodesic metric space (a finitely generated group) is {\em hyperbolic} if it is $\delta$-hyperbolic for some $\delta\ge 0$.
\end{df}
 This property is obviously invariant under quasi-isometry, so, in particular, hyperbolicity of a finitely generated group does not depend on the choice of a (finite) generating set.

Note that if $X$ is $\delta$-hyperbolic, then for every real $d>0$ we have that $X/d$ is $\delta/d$-hyperbolic. This and Properties \ref{6}, \ref{7} from \ref{spoac} immediately imply that every asymptotic cone of a hyperbolic geodesic metric space is 0-hyperbolic. Hence in every geodesic triangle in the asymptotic cone every side is contained in the union of the other two sides. This means that every asymptotic cone of a hyperbolic geodesic metric space is an $\R$-{\em \index{$\R$-tree}tree}. Thus all asymptotic cones of a hyperbolic group are complete $\R$-trees.

The converse statement is true also. Thus we have

\begin{theorem}[Gromov \cite{GrHyp, Gr1}]\label{hypascon} A finitely generated group $G$ is hyperbolic if and only if every asymptotic cone of $G$ is an $\R$-tree.
\end{theorem}

It turns out that the $\R$-trees that can appear as asymptotic cones of groups are of three kinds only because of the following recent theorem of Sisto.

\begin{theorem}\rm{(}Sisto \cite[Corollary 5.9]{Sis}\rm{)}. \label{sis}
If an $\R$-tree $T$ is an asymptotic cone of a group, then it is
a point, a line or an  $\R$-tree with valency $2^{\aleph_0}$ at each point (i.e. removing every point of the tree separates the tree into $2^{\aleph_0}$ connected components).
\end{theorem}

\begin{rk}\label{rk3} In view of  \cite{N, MNO, EP}, by Theorem \ref{sis}
there are 3 possible isometry types of $\R$-trees appearing as asymptotic cones
of groups: the point, the line, the tree with valency $2^{\aleph_0}$ at every point.
\end{rk}

\begin{rk}\label{rk4} By \cite[Proposition 6.1]{DS}, an asymptotic cone of a group is a line if and only if  the group is \index{Group!virtually cyclic}virtually cyclic\footnote{Recall that a group possesses a property $P$ {\em virtually} if it has a finite index subgroup possessing $P$.}. It is obvious that an asymptotic cone is a point if and only if the group is finite.
\end{rk}

\subsubsection{First examples of hyperbolic groups}\label{ehg} Finite and free groups, and more generally free products of finite and cyclic groups are examples of hyperbolic groups.

The classical (motivating) examples of hyperbolic groups are fundamental groups of compact negatively curved Riemannian manifolds. In particular co-compact lattices in Lie groups $SO_{n,1}(\R)$ are hyperbolic because these are fundamental groups of compact Riemannian manifolds of constant negative curvature. Hence if the genus of a closed oriented surface $S_g$ is at least 2, then the fundamental group of $S_g$ is hyperbolic. It has the presentation $$\la x_1,y_1,\ldots,x_g, y_g\mid [x_1,y_1]\ldots[x_g,y_g]=1\ra$$ found by Dehn.  It is easy to see that no two cyclic shifts of the relator or its inverse have a common prefix of length more than 1. Thus this presentation satisfies the {\em \index{Small cancelation condition!$C'(\lambda)$}small cancelation} condition $C'(\frac1{4g-1})$.

\subsubsection{The small cancelation conditions}\label{ss:tscc}
\begin{df}\label{sc} Let $\lambda>0$, and let $\pp=\la X\mid R\ra$ be a group presentation. We say that $\pp$ satisfies the condition $C'(\lambda)$ if the length of any common prefix of any two different cyclic shifts of words $r_1,r_2\in R^{\pm1}$ (called a {\em \index{Piece}piece}) is strictly less than $\lambda\min(|r_1|, |r_2|)$ (recall that all words in $R$ are cyclically reduced by our assumption (see \ref{td}), so cyclic shifts of these words and their inverses are reduced words).
\end{df}

The main property of $C'(\lambda)$-presentations is the following \index{Greendlinger lemma}Greendlinger Lemma \cite{LS}, proved by Dehn in the case of $\pi_1(S_g)$, $g\ge 2$.

\begin{theorem}[See Lyndon, Schupp \cite{LS}] \label{greend} Let $\pp=\la X\mid R\ra$ be a group presentation satisfying $C'(\lambda)$, $\lambda\le 1/6$, then every minimal van Kampen diagram $\Delta$ over $\pp$ either has no cells or has a cell $\pi$ with $\partial(\pi)=uv$ with $|u|>|v|$, and $u$ is a subpath of $\partial\Delta$ (this cell is called a {\em \index{Greendlinger cell}Greendlinger cell}).
\end{theorem}

Theorem \ref{greend} implies that in every group where the (finite) presentation $\pp$ satisfies $C'(\frac16)$ the word problem can be solved by the following {\em \index{Dehn algorithm}Dehn algorithm}. Take a reduced word $w$. If a cyclic shift $w'$ of $w$ contains more than a half of  a cyclic shift of a defining relation $r$ or its inverse ($r'=uv$, $w'=uw''$, $|u|>|v|$), then replace $uw''$ by $v\iv w''$ and reduce the resulting word. The new words is equal to 1 modulo $\pp$ if and only if $w=1$ modulo $\pp$. Since $v\iv w''$ is shorter than $w$, we can continue until we either get the empty word (in which case we conclude that $w=1$ modulo $\pp$) or we get a non-empty word for which the shortening is no longer possible (in which case we conclude that $w\ne 1$ modulo $\pp$).

Note that every group possessing a finite presentation for which Dehn's algorithm works (we shall call this a {\em \index{Dehn presentation}Dehn presentation}) has linear Dehn function because removing Greend\-lin\-ger cells from a minimal van Kampen diagram reduces the length of the boundary (hence the area of the diagram cannot exceed the perimeter).

\subsubsection{The linear isoperimeric inequality}\label{tlii} Gromov proved that every hyperbolic group has linear Dehn function and in fact possesses a finite Dehn presentation. Here is a relatively short proof of this fact using asymptotic cones.

\begin{theorem}[Gromov \cite{GrHyp}]\label{dehnp} Every hyperbolic group is finitely presented and has a finite Dehn presentation, and hence linear Dehn function. (Thus a group is hyperbolic if and only if it has a finite Dehn presentation.)
\end{theorem}

\proof We present here a proof that illustrates the usefullness of asymptotic cones. Let $G$ be a hyperbolic group. We need the following statement first. A path $\gamma\colon [0,l]\to \Gamma$ in the Cayley graph $\Gamma$ of a group $G$ is called a {\em \index{Bump}bump} if the the length $l$ of the path is at least 3 times bigger than the distance between its end points.

\begin{lemma}\label{bump} There exists a constant $C$ such that every bump of length $l>C$ contains a sub-bump of length at most $\frac13l$.
\end{lemma}

\proof Suppose that such a constant does not exist. Then for every $n=1,2,\ldots $ there exists a bump $\gamma_n$ of length $d_n>n$ which does not contain sub-bumps of length $\le d_n/3$. Pick an ultrafilter $\omega$ and consider the asymptotic cone $\ccc=\Con^\omega(G,(d_i))$. The space $\ccc$ contains the $\omega$-limit $\gamma$ of paths $\gamma_i$. Then by Property (\ref{5}) of asymptotic cones from \ref{spoac} $\gamma$ is a path of length 1 and the distance between its endpoints is at most 1/3. Recall that $\ccc$ is an $\R$-tree (Theorem \ref{hypascon}). A path of length 1 on a tree with enpoints at distance $\le 1/3$ must have a subpath of arbitrary small (but non-zero!) length which is a loop. Let $\alpha$ be that loop of length at most $\frac14$, $a$ be the beginning (= the end) point of that loop.  Then $a$ is a double point on $\gamma$: $a=\gamma(p)=\gamma(p+\varepsilon)$ where $\varepsilon\le \frac14$. That means the distance between $\gamma_i(p)$ and $\gamma_i(p+\varepsilon)$ is $o(d_i)$ while the length of the subpath $\gamma_i$ between $p$ and $p+\varepsilon$ is $O(d_i)$ and at most $\frac14 d_i +o(d_i)< \frac13 d_i$ $\omega$-almost surely. Therefore for infinitely many values of $i$ the subpath $\gamma_i(p,p+\varepsilon)$ is a bump of length at most $\frac13 d_i$, a contradiction. \endproof

Now Theorem \ref{dehnp} can be deduced as follows. Let $G$ be a hyperbolic group. Let $C$ be the constant from Lemma \ref{bump}. Let $R$ be the (finite) set of all words in the generators $X$ of $G$ which are of length $\le 2C$ and equal to 1 in $G$. Lemma \ref{bump} immediately implies that every loop in the Cayley graph of $G$ has a subpath that is a bump of length at most $C$ (note that a loop is a bump too). The bump together with a geodesic connecting its endpoints is a loop of length at most $C+\frac13 C<2C$. Therefore every group word in the alphabet $X$ that is equal to 1 in $G$ contains more than a half (in fact, at least $\frac34$) of a relator from $R$. This implies that $\la X\mid R\ra$ is a finite presentation of $G$ satisfying the Dehn property. \endproof

The converse statement for Theorem \ref{dehnp} is also true and well-known:

\begin{theorem}[Gromov \cite{GrHyp}]\label{t:lin} A finitely generated group is hyperbolic if and only if it has linear isoperimetric inequality.
\end{theorem}

\begin{exercise}\label{exe1}
Prove the ``if" part of Theorem \ref{t:lin} using asymptotic cones. (Hint: use Property \ref{7} from \ref{spoac}.)
\end{exercise}

\begin{exercise}\label{exe2} Prove all other basic properties of hyperbolic groups from \cite{GrHyp} using asymptotic cones. (Asymptotic cones do not give precise values of constants so you should replace estimates like $400\sqrt{500}$ by ``some constant".)
\end{exercise}

\subsubsection{Equations in hyperbolic groups}\label{ss:oapihg} Several other algorithmic problems have nice solutions in hyperbolic groups too. For example, Theorem \ref{Olsurf} about solvability of quadratic equations is true if we replace there ``free" by ``hyperbolic" \cite{Olsurf}. All solutions of quadratic equations without coefficients in a hyperbolic group can be described by the following

\begin{theorem}[Lys\"enok \cite{L1}] Let $S$ be a closed surface. For every hyperbolic group $G$ there exist only finitely many homomorphisms of $\pi_1(S)$ to $G$ up to the action by the mapping class group of $\pi_1(S)$ and the action of ${\mathrm{Aut}}(G)$ on $G$.
\end{theorem}

The proof of this theorem is relatively elementary comparing with the similar looking Theorem \ref{t:hp}.

An analog of Makanin's Theorem \ref{Mak} for free groups is also true for torsion-free hyperbolic groups in general.

\begin{theorem}[Rips, Sela \cite{Sela1}] \label{Sela1} There exists an algorithm to check whether a system of equations over a given torsion-free hyperbolic group has a solution.
\end{theorem}

\subsubsection{The language theory point of view.}\label{ss:tlpov} A {\em \index{Language}language} is a set of words in a finite alphabet $X$.
An {\em \index{Automaton}automaton} is a finite labeled directed graph (labels are taken from a finite alphabet $X$) with two distinguished subsets of vertices $S$ (start) and $F$ (finish) \cite{Ch}. For example, any finite subgraph of a labeled Cayley graph becomes an automaton if we choose the start and finish vertices. A language $L$ is {\em recognized} by an automaton $M$ if it consists of all words that can be read on $M$ starting from a start vertex and ending at a finish vertex. Such languages are called {\em \index{Language!regular}regular}.

It is easy to see that the language of all group words in generators of a finitely generated group that are equal to 1 in that group (let us denote it by $W(G)$) is regular if and only if the group is finite. For a bigger class of {\em \index{Language!context-free}context-free} languages, i.e. languages recognized by more advanced {\em pushdown} automata (for the definition see \cite{HU}), the situation is already different as shown by the following non-trivial theorem.

\begin{theorem}[Muller, Schupp \cite{MuSc}] \label{musc} Let $G$ be a finitely generated group. $G$ is free if and only if $W(G)$ is
context-free and $G$ is torsion-free. $G$ is virtually free if and only if $W(G)$ is context-free and
accessible.
\end{theorem}

Recall that a group $G$ is called {\em \index{Group!inaccessible}inaccessible} if it contains an infinite sequence of subgroups $G> G_1> G_2>\ldots$ and a sequence of actions by isometries on simplicial trees\footnote{A \index{Simplicial tree}simplicial tree is just a loopless connected graph in the sense of Serre \cite{Serre}.} $T_i$ such that $G_i$ is the stabilizer of a vertex in $T_i$ and all stabilizers of edges are finite. An inaccessible finitely generated group was constructed by Dunwoody
\cite{Dun}. A group is {\em \index{Group!accessible}accessible} if it is not inaccessible. All finitely presented groups are accessible \cite{Dun1}.

By Holt \cite{Hol}, the word problem in every hyperbolic group is {\em \index{Language!real time recognizable}real time recognizable}, which, informally, means that it can be recognized by the Turing machine that reads its input at a constant rate, and stops at the end of the input (regular and context-free languages are real time recognizable). This implies, in particular,

\begin{theorem}[See, for example, Holt \cite{Hol}]\label{Hol} The word problem in every hyperbolic group can be solved in linear time by a (multi-tape) Turing machine.
\end{theorem}

Note that the fact that the Dehn function is linear directly implies only that the word problem can be recognized in linear time by a non-deterministic Turing machine. The proof of Theorem \ref{Hol} uses the fact that every hyperbolic group has a Dehn presentation (Theorem \ref{dehnp}). In \cite{HR}, Holt and Rees used the fact that nilpotent and some relatively hyperbolic groups have so-called {\em \index{Dehn presentation!generalized}generalized Dehn presentations} introduced earlier by Cannon, Goodman and Shapiro. They prove that groups from these classes also have real time recognizable word problem.

For more information about groups with word problem from some more complicated language classes see Gilman \cite{Gi}.

\subsubsection{The isomorphism problem for hyperbolic groups} \label{ss:tipfhg}
The \index{Isomorphism problem for hyperbolic groups}{\em isomorphism problem for hyperbolic groups}\footnote{This problem has as the input two presentations of groups that are known to be hyperbolic, and asks whether the groups are isomorphic} is decidable. This is a very strong theorem proved in the torsion-free case by Sela (see \cite{Sela} for a weaker result, and \cite{DaGr} by Dahmani and D. Groves for the full proof) and in the general case by Dahmani and Guirardel \cite{DaG}.

Here we present a short description of the solution of isomorphism problem which is a somewhat modified text sent to us by Fran\c cois Dahmani.

Let $G=\la X\mid r_1,\ldots,r_m\ra$ and $G'=\la X'\mid r_1',\ldots, r_n'\ra$ be two hyperbolic groups. Suppose first that every action of each of the groups $G,G'$ on a simplicial tree is trivial, i.e. fixes a point. Every homomorphism $\phi\colon G\to G'$ induces an action of $G$ on the Cayley graph of $G'$: $a\circ_\phi g=\phi(a)g$. Let $d_\phi$ be the maximal number such that every vertex of the Cayley graph of $G'$ is moved by at least $d_\phi$ by some generator of $G$ under this action. It is easy to check that if a sequence of homomorphisms $\phi_1,\phi_2,\ldots$ consists of pairwise non-conjugate in $G'$ homomorphisms, then the sequence of numbers $d_{\phi_i}$ is unbounded, hence $G$ acts non-trivially on the asymptotic cone of $G'$ corresponding to the scaling sequence $(d_{\phi_i})$ which is an $\R$-tree by Theorem \ref{hypascon}. In fact one can deduce (as in Sela \cite[Theorem 9.1]{Sela} using Rips' theory of groups acting on $\R$-trees and its generalizations) that in this case $G$ acts non-trivially on a simplicial tree as well. Hence we can assume that any set of pairwise non-conjugate homomorphisms $G\to G'$ and $G'\to G$ is finite.

Therefore in order to decide whether $G$ and $G'$ are isomorphic we need to compute a finite list of homomorphisms containing a representative of each of the homomorphisms $G\to G'$ (up to conjugacy), and a similar finite list of homomorphisms $G'\to G$. Assuming that this can be done, one checks whether a composition of a homomorphism from the first (second) list with a homomorphism from the second (first) list is conjugate to the identity map, that is whether the images of the generators of $G$ (of $G'$) under this composition are conjugate to the generators of $G$ (of $G'$) by the same conjugator. It is a system of quadratic equations (with one variable - the conjugator) and can be solved essentially as one quadratic equation (see \ref{ss:oapihg}), or using the general technique by Rips and Sela \cite{Sela1}.

There is an obvious infinite procedure to write down these lists: check all maps $X\to G'$ and $X'\to G$ one by one, and verify which maps induce homomorphisms. The problem is
to understand when to stop: that is when  a  list contains representatives  of all conjugacy classes of homomorphisms. For this, the key idea is to solve systems of equations. Indeed,  morphisms $G\to G'$ are in 1:1 correspondence with solutions in $G'$ of the system  of equations $$\left\{\begin{array}{l}r_1=1\\ \ldots \\ r_m=1\end{array}\right.$$ with generators of $G$ viewed as unknowns.  Adding certain inequations (of the form $w \neq 1$), one can encode the property of being different from homomorphisms that are already on the list. Thus we can easily write a system of equations and inequations that has no solution if and only if we have collected a finite list containing representatives of all homomorphisms. The difficulty here is that we want not only that the solution be different from the recorded morphisms, but we want that it is not conjugated to any of them. For that, we record only homomorphisms $G\to G'$ that are {\em short}  meaning that one cannot shorten the total length of the images of $x\in X$ by ``obvious" conjugations in $G'$ (such as cyclic shifts). The key points are that, first, shortness can be interpreted as the membership in a regular language (we view a homomorphism $\phi\colon G\to G'$ as one large word containing all words $\phi(x), x\in X$ separated by special symbols); that membership is called a {\em \index{Regular constraint}regular constraint}, and, second, that the number of short homomorphisms is finite. The fact that some kind of shortness of a homomorphism can be recognized by an automaton is similar to Lemma \ref{bump} above which characterizes words in a hyperbolic group without arbitrary bumps as words which do not contain  bumps of at most constant length. It is easy to see that for every finite set of words $W$, the set of all words in a finite alphabet that do not contain subwords from $W$ is regular. Thus  the certificate for completeness of our list of homomorphisms is unsolvability of the system of equations and inequations subject to the regular constraint described above. Thus one needs to generalize Makanin's Theorem \ref{Mak} and Rips and Sela's Theorem \ref{Sela1} to systems of equations and inequations subject to regular constraints in not necessarily torsion-free groups. This has been done in a long series of papers including \cite{Mak1,Schu,DGH, Dah,DaG1,LoSe}.

If the groups $G, G'$ may act on simplicial trees without fixed points, the strategy is to find a canonical action on a tree where the vertex stabilizers are in a sense the smallest possible for both $G$ and $G'$ (the so-called {\em JSJ-decomposition} introduced by Rips and Sela \cite{RiSeJSJ}), and then study vertex stabilizers as in the previous paragraph.
\subsubsection{More examples of hyperbolic groups.}\label{ss:meohg}

As we have seen, every group having a presentation with the small cancelation condition $C'(1/6)$ is hyperbolic (see \ref{tlii}). Moreover almost every finitely presented group is hyperbolic \cite{GrHyp}. That statement was proved in \cite{GrHyp, AO, OlHyp, GrRand}, but although the statements proved there have basically the same formulations,  the results are different because of different meanings of the words ``almost every", or, more precisely, different choices of the probabilistic model.

\paragraph{\index{Probability model!few relator}Few relator model.}\label{ss:cm} In the {\em few relator} model (also known as the {\em asymptotic density} model or {\em Arzhantseva-Olshanskii model} introduced by Gromov in \cite{GrHyp} and Olshanskii in late 80s, we choose an alphabet of generators $A$ and a number of relations $k$, and then for every $n>1$ choose $k$ group words $r_1,\ldots,r_k$ of length $\le n$ in the alphabet $A$, and consider the group given by the presentation $\la A\mid r_1=1,\ldots,r_k=1\ra$. Let $N(n)$ be the number of all presentations we get this way, and $N_{\rm h}(n)$ be the number of presentations of hyperbolic groups among them. Then
the quotient $\frac{N_{\rm h}(n)}{N(n)}$ tends to 1 very rapidly with $n$ (see Arzhantseva-Olshanskii \cite{AO}). In fact, as proved in \cite{AO}, in the few relator model, we can count even the number of presentations with small cancelation $C'(\lambda)$ for any fixed $\lambda>0$, and the quotient will tend to 1 almost as rapidly. Note that in \cite{AO}, Arzhantseva and Olshanskii proved that almost all groups having presentations with $m$ generators and $k$ relators have the property that all their $(m-1)$-generator subgroups are free (that strengthened a previous result of Guba about 2-generated subgroups of 1-related groups with at least 4 generators \cite{Guba86}). The same or similar models were used by Arzhantseva and Champetier  \cite{Ar1,Ar2,Ar3, Cha1, Cha} where other generic properties of groups in the few relator model were found.

\paragraph{\index{Probability model!different lengths}Different lengths model.}\label{ss:om}

If we allow lengths of different relators satisfy different length restrictions, we obtain a stronger theorem.

\begin{theorem}[Formulated by Gromov \cite{GrHyp}, proved by Olshanskii \cite{OlHyp}] Let $A$ be an alphabet, $|A|=k\ge 2$. Let $d\geq 0$ and let $(n_1, \cdots, n_d)$ be an increasing sequence of positive integers. Let $N = N(k,i,n_1, \ldots, n_d)$ be the number of group presentations $G = \langle A | r_1, \cdots, r_d\rangle$ such that $r_1,\cdots,r_d$ are reduced group group words in the alphabet $A$ such that the length of $r_j$ is $n_j$ for $j=1,2,\ldots,i$. If $N_{\rm h}$ is the number of infinite hyperbolic groups in this collection, then for every $\varepsilon>0$ there exists $n=n(\varepsilon,d)$ such that for every choice of $k, n_1\ge n,n_1\le \ldots\le n_d$  the quotient $\frac{N_{\rm h}(k,n_1,\ldots, n_d)}{N(k,n_1,\ldots, n_d)}$ exceeds  $1-\varepsilon$.
\end{theorem}

The probabilistic model used in this theorem is very different from the few relator model. In particular, the probability of having a small cancelation presentation tends to 0 in that model (say, if we assume that $n_d\approx \exp n_1$).

\paragraph{\index{Probability model!random groups}Gromov's random groups model.}\label{ss:grgm} If we also allow varying the number of relations $d$, we obtain one of the probabilistic models used by Gromov in constructing his {\em \index{Group!random}random groups} \cite{GrRand} (see also Gromov \cite{Gr1} and Arzhantseva-Delzant \cite{AD}).

Here is the definition of the Gromov's random groups model.

Let $\Gamma = (V, {E})$ be an oriented graph and $X$ be a finite set. A {\em labelling} of $\Gamma$ is a map $\alpha \colon {E}\rightarrow X \cup X^{- 1}$. Let $W$ be the normal subgroup of the free group $\pi_1(\Gamma,p)$ (where $p\in V$) generated by all words that label a loop in $\Gamma$ based at $p$. Let $G_\alpha$ be the factor-group $\pi_1(\Gamma,p)/W$.

This map obviously defines a homomorphism $\alpha^\ast$ from the free group $\pi_1(\Gamma,p)$ to $G_\alpha$, where $p$ is a vertex in $V$. The kernel of that map is generated (as a normal subgroup) by all words that label loops of $\Gamma$.
Note that by construction there exists a natural homomorphism $\alpha^+$ from the graph $\Gamma$ to the Cayley graph of $G$ induced by $\alpha$. The general questions to ask about $G$ is whether $G$ is infinite and hyperbolic, whether $\alpha^+$ is ``almost an isometric embedding", and whether $\alpha^\ast$ is injective on a large enough ball in the Cayley graph of the free group $\pi_1(\Gamma,p)$. The set of all labellings $(X \cup X^{- 1})^E$  has a natural uniform probability measure, and a random choice of a labelling $\alpha$ gives a random group $G_\alpha$.

In particular, if $\Gamma$ is a disjoint union of finite number of circles of lengths $n_1,\ldots,n_d$, we get a random model similar to the different lengths model only the number of relations may be big comparing to the lengths of them. In fact, if the total number of possible labelings is $\exp Cn$, where $n$ is the maximal length of a cycle, for some constant $C$, and $n_i$ do not differ very much from $n$, then $d$ can be chosen as large as $\exp \alpha Cn$ for some $0<\alpha<1$. The number $\alpha$ is called the {\em density}. If the density is $>1/2$, then the random group is trivial with probability close to 1 (assuming that $n\to \infty$), but if $\alpha<\frac12$, then the group is infinite and non-virtually cyclic hyperbolic with probability tending to 1 as $n\to \infty$.

It is important to note that, as with the case of small cancelation presentations, we can impose the random relations not only on the  free group, but on any non-virtually cyclic hyperbolic group. Thus if $\Gamma_i$ is a sequence of graphs, we obtain a sequence of random non-virtually cyclic hyperbolic groups and surjective homomorphisms
\begin{equation}\label{grse} G\to G_1\to \ldots \to G_n\to\ldots  \end{equation}
similarly to \ref{ss:qohg}.

Particularly strong results are obtained in \cite{GrRand} when the graph $\Gamma$ runs over an expander sequences of graphs, that is $k$-regular finite graphs having the following property: the second highest eigenvalues of their adjacency matrices is bounded away from the highest eigenvalue\footnote{Equivalently, an expander sequence of graphs can be defined as a sequence of $k$-regular graphs $\Gamma_i=(V_i,E_i)$ such that for every subset $M\subset V_i$ of vertices of $\Gamma_i$ with $|M|<|V_i|/2$, $\frac{|\partial M|}{|M|}>\lambda$ for some constant $\lambda>0$ \cite{Lub}. Here $\partial M$ is the set of edges with one vertex in $M$ and the other in $V\setminus M$.}. If the expander sequence is chosen well enough, the Cayley graph of the inductive limit $\bar G$ of the sequence \ref{grse} contains ``almost isometric" copy of the union of $\Gamma_i$, and so $\bar G$ cannot be coarsely embedded into a Hilbert space, every action of $\bar G$ on a CAT(0) manifold has a global fixed point, etc. For more details see the book by Ollivier \cite{Oll} and I. Kapovich and Schupp \cite{KaSc}.

\subsubsection{Almost all 1-related groups are residually finite}\label{ss:aa1garf} Note that for 1-related groups all three models give the same results. A random 1-relator group with at least 2 generators is infinite and hyperbolic. The following theorem says that these groups have some other nice properties as well at least when the number of generators is $\ge 3$.

\begin{theorem}[Sapir, \v Spakulova \cite{SS}]\label{ss} A random 1-related group with at least 3 generators is
\begin{enumerate}
\item residually finite;
\item virtually \index{Group!residually (finite $p$-)}residually (finite $p$-)group\footnote{A group is called {\em residually (finite $p$-) group}, $p$ a prime, if the intersection of all its subgroups with indexes powers of $p$ is trivial.} for all but finitely many primes $p$ (i.e. it has a finite index subgroup whose homomorphisms onto finite $p$-groups separate all elements);
\item \index{Group!coherent}coherent (i.e. all finitely generated subgroups are finitely presented).
\end{enumerate}
\end{theorem}

The proof employs facts from different areas of mathematics: properties of Brownian motions in $\R^k$ \cite{CHM}; some dynamical properties of polynomial maps over finite fields (existence of periodic quasi-fixed points)\cite{BoSa1} and $p$-adic completions of number fields \cite{BoSa2}; existence of dense free subgroups in $p$-adic Lie groups \cite{BrGe}; the {\em \index{Congruence extenson property}congruence extension property} of subgroups of the free groups generated by sets of words satisfying some small cancelation condition \cite{O1}; the fact that every asceding HNN-extension of a free group is coherent \cite{FH}, and a characterization of 1-related groups with 2 generators that are ascending HNN-extensions of free groups \cite{Brown1}. It turned out also that in Theorem \ref{ss} one can count not just presentations but groups up to isomorphism (which is obviously harder) due to a very strong result of I. Kappovich, Shpilrain, Schupp \cite{KSS}. See \cite{SaRes} for more details.

Note that Theorem \ref{ss} gives many examples of residually finite hyperbolic groups because random 1-related groups are hyperbolic (and even have small cancelation presentations). But we still do not know if all hyperbolic (and even small cancelation) groups are residually finite \cite{GrHyp, Gr1}. The general belief is that the answer is ``no", and there are many potential counterexamples (see, say, \cite{IO}), but it is very hard to show that a particular hyperbolic group is not residually finite because hyperbolic groups typically have many quotients.

\subsubsection{Quotients of hyperbolic groups}\label{ss:qohg}

One of the key properties of hyperbolic groups discovered by Gromov in \cite{GrHyp} is that they behave similar to the free groups as far as the multitude of quotients is concerned.

For example, Delzant and Olshanskii independently proved the following

\begin{theorem}[Delzant \cite{Del}, Olshanskii \cite{O1}]\label{OlSQ} Every non-virtually cyclic hyperbolic group $G$ is SQ-universal, that is every countable group embeds into a quotient of $G$.
\end{theorem}

A version of the following very useful theorem was stated in \cite{GrHyp} in a slightly wrong form. It was corrected and proved in full generality in \cite{Ols}. We give a not the strongest formulation here.

\begin{theorem}\rm{(}Olshanskii \cite[Theorem 2]{Ols}\rm{).}\label{OlG} Let $G$ be a hyperbolic group and $H_1, \ldots H_n$ be  non-virtually cyclic subgroups of $G$ which do not normalize non-trivial finite subgroups of $G$, $M$ a finite subset of $G$.
Then there exists a non-virtually cyclic hyperbolic quotient $G_1$ of the group $G$ such that the natural map $G\to G_1$ is surjective on each $H_1,\ldots,H_n$ and injective on $M$.
\end{theorem}

It is easy to deduce from Theorem \ref{OlG} that every noncyclic torsion-free hyperbolic group has a quotient which is a torsion-free Tarski monster (i.e. it is torsion-free and all proper subgroups are cyclic); every non-virtually cyclic hyperbolic group has an infinite torsion quotient; every noncyclic torsion-free hyperbolic group has an infinite quotient all of whose proper subgroups are finite (see \cite{GrHyp}). These groups with ``extreme" properties are obtained as inductive limits of hyperbolic groups and surjective homomorphisms
$$G\to G_1\to\ldots \to G_n\to \ldots  $$
similar to \ref{ss:grgm}.

It is also easy to deduce that every two non-virtually cyclic hyperbolic groups $H_1$, $H_2$ have a common non-virtually cyclic quotient \cite[5.5.A]{GrHyp} (con\-sider the free product $G=H_1/E(H_1)*H_2/E(H_2)$ where $E(H)$ is the maximal finite normal subgroup of $H$). It was observed by Osin, that it implies, by induction, that there exists a finitely generated non-virtually cyclic infinite group $G_\infty$ that is a common quotient of all non-virtually cyclic  hyperbolic groups. One can easily see that $G_\infty$ has Kazhdan property (T), is generated by an element of order 2 and an element of order 3, and also by two elements of orders $p$ and $q$ for any two primes with $pq\ge 6$ (since $G_\infty$ is a quotient of the free product $\Z/p\Z *\Z/q\Z$). In particular $G_\infty$ is not residually finite, and in fact does not have any non-trivial finite homomorphic image $G_0$ (otherwise $G_0$ would contain an element of order $p>
|G_0|$ which is impossible). Note that $G_\infty$ is not unique, and depends, in particular, on how we enumerate all hyperbolic groups. Still since the set of all finite hyperbolic group presentations is recursively enumerable, we can assume, using \cite{Ols}, that $G_\infty$ has a recursive presentation.

Theorem \ref{OlG} also allows one to construct a finitely generated infinite torsion group (and even with all proper subgroups cyclic), see \cite{GrHyp}. It does not allow to get a bound on the exponents of elements of that group. It was done in papers by Olshanskii \cite{OlPer} (odd exponents) and S.V. Ivanov and Olshanskii \cite[Theorem A, Lemma 19]{IO} (even exponents) where the following remarkable theorem was proved.

\begin{theorem}[Olshanskii \cite{OlPer}, S.V.Ivanov, Olshanskii \cite{IO}]\label{io} For every non-cyclic tor\-sion-free hyperbolic group $G$ there exists an odd integer $m=m(G)>0$ and a integer $n=n(G)>0$, such that the following hold. (a) The quotient groups $G/G^m$, $G/G^{2^n}$ are infinite\footnote{Recall that $G^k$ denotes the (normal) subgroup of $G$ generated by all $k$th powers of elements of $G$; hence $G/G^k$ has exponent dividing $k$.}. (b) The word and conjugacy problems are solvable in $G/G^{m}, G/G^{2^n}$.
\end{theorem}

That result was the strongest in the very important series of results on the bounded Burnside problem including the celebrated Novikov-Adian's solution for odd exponents \cite{NA, Ad}, and extremely involved S.V.Ivanov's solution for even exponents \cite{Iv}.

\subsubsection{The small cancelation once more.}\label{ss:tscom} The main ingredient in the proofs of Theorems \ref{OlSQ}, \ref{OlG} and the results in \cite{GrRand} is small cancelation over hyperbolic groups. The following coarse version of small cancelation from Olshanskii-Osin-Sapir \cite{OOS} is close to the ones used in papers \cite{Del, Ols, O1}.

\begin{df}\label{piece}
Let $H$ be a group generated by a set $X$. Let $\mathcal R$ be a
symmetrized set of reduced words in $X^{\pm 1}$ (that is a set of words closed under taking cyclic shifts and inverses). For $\e
>0$, a subword $U$ of a word $R\in \mathcal R$ is called a {\it
$\e $-\index{Piece!$\e$-piece}piece}  if there exists a word $R^\prime \in \mathcal R$
such that:

\begin{enumerate}
\item[(1)] $R\equiv UV$, $R^\prime \equiv U^\prime V^\prime $, for
some $V, U^\prime , V^\prime $; \item[(2)] $U^\prime = YUZ$ in $H$
for some words $Y,Z$ such that $\max \{ |Y|, \,|Z|\} \le \e $;
\item[(3)] $YRY^{-1}\ne R^\prime $ in the group $H$.
\end{enumerate}
Note that if $U$ is an $\e$-piece, then $U'$ is an $\e$-piece as
well.
\end{df}

\begin{df}\label{SC} Let $\varepsilon \ge 0$, $\mu\in (0,1)$, and $\rho
>0$.
We say that a symmetrized set $\mathcal R$ of words over the
alphabet $X^{\pm 1}$ satisfies the {\em \index{Small cancelation condition!$C(\varepsilon , \mu ,\rho
)$} small cancelation condition $C(\varepsilon , \mu ,\rho
)$} over a hyperbolic group $H=\la X\ra$, if
\begin{enumerate}
\item[($C_1$)] All words from $\mathcal R$ label geodesics in the Cayley graph of $H$ corresponding to $X$;
\item[($C_2$)] $|R|\ge \rho $ for any $R\in \mathcal R$;
\item[($C_3$)] The length of any $\varepsilon$-piece contained in any word $R\in \mathcal R$ is smaller than $\mu |R|$.
\end{enumerate}
\end{df}

Figure \ref{scc} shows the difference between the usual small cancelation (see Definition \ref{sc}) and the coarse version of it. The key feature of the coarse version of small cancelation is the coarse \index{Greendlinger lemma!coarse}Greendlinger Lemma (see \cite[Lemma 4.6]{OOS}). The role of that lemma is very similar to the role of the usual Greendlinger Lemma (Theorem \ref{greend}) - see \cite{Ols, O1, OOS} for more details.

\begin{figure}[H]

\unitlength .55mm 
\linethickness{0.4pt}
\ifx\plotpoint\undefined\newsavebox{\plotpoint}\fi 
\begin{picture}(246.625,56.375)(10,0)
\qbezier(47.75,38)(47.125,55.625)(27,55.75)
\qbezier(169.25,38.5)(168.625,56.125)(148.5,56.25)
\qbezier(27,55.75)(9.625,55.875)(7.75,35.5)
\qbezier(148.5,56.25)(131.125,56.375)(129.25,36)
\qbezier(7.75,35.5)(5.75,18)(25.75,15.5)
\qbezier(129.25,36)(127.25,18.5)(147.25,16)
\qbezier(25.75,15.5)(46.625,13)(47,24.5)
\qbezier(147.25,16)(168.125,13.5)(168.5,25)
\multiput(47.25,24)(.03125,1.84375){8}{\line(0,1){1.84375}}
\multiput(168.75,24.5)(.03125,1.84375){8}{\line(0,1){1.84375}}
\qbezier(47.75,38.5)(48.25,46.875)(88.75,52.75)
\qbezier(174.75,39)(175.25,47.375)(215.75,53.25)
\qbezier(88.75,52.75)(109.375,54.125)(114.5,32)
\qbezier(215.75,53.25)(236.375,54.625)(241.5,32.5)
\qbezier(114.5,32)(119.625,8.875)(80.25,9.25)
\qbezier(241.5,32.5)(246.625,9.375)(207.25,9.75)
\qbezier(80.25,9.25)(56.75,9.625)(47.25,24.5)
\qbezier(207.25,9.75)(183.75,10.125)(174.25,25)
\multiput(174.75,39)(-.0333333,-1){15}{\line(0,-1){1}}
\multiput(174.25,24)(-.3666667,.0333333){15}{\line(-1,0){.3666667}}
\put(168.75,39){\line(1,0){5.75}}
\put(44.75,31){\line(0,-1){.25}}
\put(47.375,31.75){\vector(0,1){.07}}\multiput(47.25,28.5)(.03125,.8125){8}{\line(0,1){.8125}}
\put(32.375,55.25){\vector(-1,0){.07}}\multiput(34.75,55)(-.3166667,.0333333){15}{\line(-1,0){.3166667}}
\put(72.25,49.75){\vector(4,1){.07}}\multiput(68.5,49)(.16666667,.03333333){45}{\line(1,0){.16666667}}
\put(168.75,31.5){\vector(0,1){.07}}\put(168.75,28.5){\line(0,1){6}}
\put(148.5,56){\vector(-1,0){.07}}\multiput(152.5,55.75)(-.5333333,.0333333){15}{\line(-1,0){.5333333}}
\put(174.5,31.75){\vector(0,1){.07}}\put(174.5,29){\line(0,1){5.5}}
\put(199.5,50.375){\vector(4,1){.07}}\multiput(194.75,49.5)(.18269231,.03365385){52}{\line(1,0){.18269231}}
\put(43,31.25){\makebox(0,0)[cc]{$U$}}
\put(26.75,52.5){\makebox(0,0)[cc]{$V$}}
\put(74.5,47.5){\makebox(0,0)[cc]{$V'$}}
\put(43.5,6.25){\makebox(0,0)[cc]{Standard piece}}
\put(170.5,8.5){\makebox(0,0)[cc]{$\e$-piece}}
\put(165.75,31.5){\makebox(0,0)[cc]{$U$}}
\put(148.75,52.75){\makebox(0,0)[cc]{$V$}}
\put(179.25,32.5){\makebox(0,0)[cc]{$U'$}}
\put(201.5,45.5){\makebox(0,0)[cc]{$V'$}}
\put(171.75,42){\makebox(0,0)[cc]{$Y$}}
\put(171.25,20.5){\makebox(0,0)[cc]{$Z$}}
\end{picture}

\caption{The standard and coarse definitions of a piece.}\label{scc}
\end{figure}
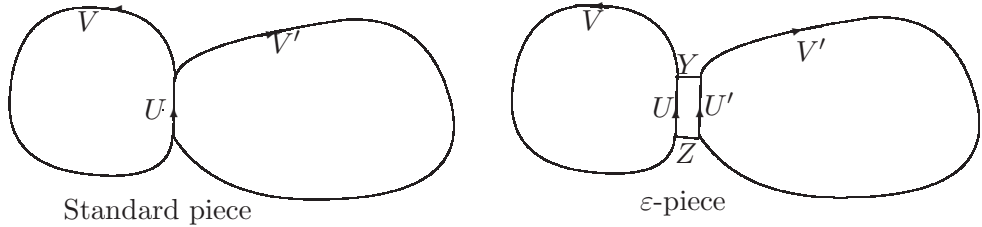

\subsubsection{Combination theorems for hyperbolic groups.}\label{ss:ctfhg}

Another way to construct new hyperbolic groups is to use various combination theorems. Suppose that a finitely generated group $G$ acts by automorphisms  on a simplicial tree $T$  and there are no invariant proper subtrees for this action. From the Bass-Serre theory \cite{Serre}, the group $G$ is a multiple HNN-extension of the stabilizer of a vertex with associated subgroups isomorphic to stabilizers of edges if and only if the action is transitive on vertices; and $G$ is an {\em \index{Amalgamated product}amalgamated product} of stabilizers of two vertices with the stabilizer of an edge as an  amalgamated subgroup if and only if the action has two orbits of vertices and two orbits of edges. In general, $G$ can be composed from the stabilizers of vertices using HNN-extensions and amalgamated products. A statement saying that some property $P$ holds for $G$ provided it holds for stabilizers of vertices of the action of $G$ on $T$ (plus some condition on the stabilizers of edges) is called a {\em \index{Combination theorem}combination theorem}. For example, the statement ``if all stabilizers of vertices are torsion-free, then $G$ is torsion-free" is such a theorem.

A very general combination theorem for hyperbolic groups is proved by Bestvina and Feighn \cite{BF1,BF2}. There are several other versions proved in \cite{Git,KM1, MiOl} and other papers. Here we present the ``greatest common divisor" of all these statements. It was proved in each of these papers, and reflects the nature of all combination theorems.

\begin{theorem}\label{ct} If $G$ is a hyperbolic group, $A$ and $B$ isomorphic virtually cyclic subgroups, $\phi\colon A\to B$ is an isomorphism between these subgroups, then the HNN-extension $H =\HNN(G; A,B, \phi)$ is hyperbolic if and only if for every $g\in G$ the set $A\cap gBg\iv$ is finite.
\end{theorem}

\subsubsection{The subquadratic isoperimetric inequality}\label{ss:tsii}

The ``if part" of Theorem \ref{t:lin} can be strengthened significantly: one can prove that even a quadratic isoperimetric inequality with small enough coefficient of the square implies hyperbolicity \cite{GrHyp}.

\begin{theorem}\rm{(}Gromov \cite[6.8]{GrHyp}, Olshanskii \cite{OlHyp1}, Bowditch \cite{Bow1}, Papasoglu \cite{P0}\rm{)}.\label{Bow1} There exists a universal constant $\upsilon$ such that the following holds. Let $G=\la X\mid R\ra$ be a group presentation, $p=\max(|r|, r\in R)$. If the Dehn function $f$ of $G$ satisfies $$f(n)< \frac{\upsilon}{p^2}n^2$$ for all sufficiently large $n$, then $G$ is hyperbolic.
\end{theorem}

One can actually estimate the universal constant $\upsilon$ by analyzing the proof of Bowditch \cite{Bow2}, for example. The estimate would be far from optimal, though. Stefan Wenger informed us that one can take any $\upsilon < \frac{1}{2}$: it follows from his paper \cite{Wen1}. It is an interesting question to find exact value of $\upsilon$ or at least a good approximation. Note that for the free Abelian group $\Z^2$ with the standard presentation $\la x,y\mid [x,y]=1\ra$, the Dehn function is $f(n)=\frac 1{16}n^2$ (it essentially follows from \ref{ss:poe}) and $p=4$.  So $\upsilon$ cannot be bigger than 1.

Bowditch  \cite{Bow1} proved actually a much more general result about all geodesic metric spaces with area functions satisfying Conditions ($B_1$) and ($B_2$) from \ref{ss:tdob} (the role of $1/p^2$ is played by the constant $k$ from ($B_2$)). Wenger's results in \cite{Wen1} are also much more general and apply to arbitrary geodesic metric spaces as well. In fact he proves that the isoperimetric inequality of the Euclidean plane, $$\mathrm{Area}(l)\le \frac1{4\pi}l^2$$ is optimal in the sense that every geodesic metric space with stronger isoperimetric inequality (for some natural area function) is hyperbolic. The crucial tool in his proof is metric currents as in \ref{ss:tdok}.

\subsubsection{The Cartan-Hadamard type theorem}\label{ss:tchtt}

Cartan-Hadamard type theorems are ``local-to-global" statements in group theory: we deduce information about the global structure of a group $G$ by looking at one (sufficiently large) ball in the Cayley graph of $G$.\footnote{Recall that the original \index{Cartan-Hadamard theorem}Cartan-Hadamard theorem states that the universal cover of a Riemannian maniford of non-negative sectional curvature (a local property) is diffeomorphic to a Euclidean space (a global property).}

\begin{theorem}
\rm{(}See Gromov \cite[6.8M]{GrHyp},  Bowditch \cite[Theorem
8.1.2]{Bow2}\rm{)}.\label{t:CH} There are  constants $C_1$, $C_2$, and
$C_3$ with the following property. Let $\Gamma$ be the Cayley graph of a finitely presented group where the length of every defining relation is at most $d$ and
every ball of radius $C_2d$ in $\Gamma$ is $C_1d$-hyperbolic.
Then $\Gamma$ is $C_3d$-hyperbolic.
\end{theorem}

An analogous result for metric spaces is also true (see \cite{GrHyp}). The proof in \cite{GrHyp} is based on the fact that Theorem \ref{Bow1} remains true even if restrict the assumption to relatively short loops: if every relatively short loop in $\Gamma$ has quadratic area with sufficiently small coefficient, then the group is hyperbolic.

This theorem and its variations are very important in the theory of hyperbolic groups. A version of it from \cite{GrMes} is used in the paper by Gromov and Delzant \cite{DG} (see also Coulon \cite{Cou}) to provide possibly the most conceptually easy proof of the theorem of Novikov-Adian \cite{NA} and Olshanskii \cite{OlPer}. Another version of Theorem \ref{t:CH} for CAT(0) spaces is proved in Bridson and Haefliger \cite{BH}. A yet another version was recently proved by Shalom and Tao \cite{ST}: it significantly strengthens the polynomial growth result of Gromov \cite{Gromov} (mentioned in \ref{ss:tacatcp}).

\begin{theorem} [Shalom, Tao \cite{ST}] For some constant $C$, the following holds for every finitely generated group $G$, and all $d >0$.  If there is some $R_0 > \exp(\exp(Cd^C))$ for which the number of elements in a ball of radius $R_0$ in a Cayley graph of $G$ is bounded by $R_0^d$, then $G$ has a finite index subgroup which is nilpotent (of nilpotency class $<C^d$).
\end{theorem}

\begin{prob} It would be interesting to find a common cause of all these (and other) theorems of Cartan-Hadamard type. Perhaps a recent model theoretic approach of Hrushovski will give such a common cause (see \cite[Section 7]{Hru}).
\end{prob}

\subsubsection{Lacunary hyperbolic groups}\label{ss:lhg} It was observed by M. Kapovich and Kleiner (see the Appendix to \cite{OOS}) that Theorem \ref{t:CH} almost immediately implies the following stronger version of Theorem \ref{hypascon}.

\begin{theorem}[See \cite{OOS}]\label{t:KK} Let $G$ be a finitely presented group with one asymptotic cone an $\R$-tree. Then $G$ is hyperbolic (and hence all asymptotic cones of $G$ are $\R$-trees).
\end{theorem}

Indeed, if one of the asymptotic cones $\Con^\omega(G,(d_i))$ is an $\R$-tree, then the ball in the Cayley graph of $G$ or radius $d_i$ is $o(d_i)$-hyperbolic for $\omega$-almost all $i$, and it remains to apply Theorem \ref{t:CH} since $d_i\to \infty$.

Theorem \ref{t:KK} is far from true for infinitely presented groups. The first example of a finitely generated infinitely presented (hence non-hyperbolic) group was constructed by Thomas and Velickovic in \cite{TV}. Their group is given by an infinite small cancelation presentation such that the lengths of relators form a very lacunary (sparse) sequence of numbers. As an almost direct application of Greendlinger's Lemma (see Theorem \ref{greend}), one can prove that any cone $\Con^\omega(G, (d_i))$ is an $\R$-tree if $d_i$ is far from the length of relators, and is not an  $\R$-tree if $d_i$ is the length of the relator number $i$.

Accordingly, in \cite{OOS}, we introduced a class of groups some of whose asymptotic cones are $\R$-trees, and called these {\em \index{Group!lacunary hyperbolic}lacunary hyperbolic groups}. It is easy to deduce from Theorem \ref{t:CH}, that every lacunary hyperbolic groups is an inductive limit of hyperbolic groups and surjective homomorphisms. A converse statement is also true if one imposes certain restrictions on the ``injectivity radia" of these surjective homomorphisms. Tarski monsters, infinite finitely generated torsion groups, etc. can be constructed as lacunary hyperbolic groups. Surprisingly, some amenable (see \ref{ss:shotp}) non-virtually cyclic groups turned out to be lacunary hyperbolic as  well. For more details and results see \cite{OOS}.

\subsection{Nilpotent groups}\label{ng}

\subsubsection{The asymptotic cones and the $c+1$-problem}\label{ss:tacatcp} Let $G$ be a {\em \index{Group!nilpotent}nilpotent group} of class $c$ (that is the $c$th member of its \index{Lower central series}lower central series, $G_c$, is the identity; $G_n$ is defined by $G_0=G, G_{n+1}=[G_n,G]$, $n=0,1,\ldots$). Every finitely generated nilpotent group is finitely presented. Algorithmic properties of nilpotent groups have been studied extensively, and this class can be considered relatively ``tame". Note that the word, conjugacy and even the isomorphism problem are solvable for nilpotent groups (Grunewald, Segal \cite{GrS80,GrS81}). Still some natural algorithmic problems are known to be undecidable in some nilpotent groups. For example, there exists a finitely generated nilpotent group with undecidable {\em \index{Endomorphic reducibility problem}endomorphic reducibility} problem: given two elements of a group find out if one of them is an endomorphic image of another (Romankov \cite{Rv77}). In particular, for some finitely generated nilpotent group, the problem of solvability of systems of equations in that group is undecidable. Indeed, $G=\la x_1,\ldots,x_m\mid r_1=1,\ldots,r_n=1\ra$, is a finitely presented group, then the endomorphic reducibility of $u$ to $v$ is equivalent to the solvability of the following system of equations with respect to unknowns $w_1,...,w_m$:

$$\left\{\begin{array}{l} r_1(w_1,\ldots, w_m)=1\\
\ldots\\
r_n(w_1,\ldots,w_m)=1\\
u(w_1,\ldots,w_m)=v
\end{array}\right.
$$

Among nilpotent groups, particularly nice ones are graded nilpotent groups (sometimes also called {\em homogeneous}). Recall that a nilpotent group is called {\em \index{Group!graded nilpotent}graded} if for every $x\in G_k\setminus G_{k+1}, y\in G_l\setminus G_{l+1}$, the commutator $[x,y]$ belongs to $G_{k+l}\setminus G_{k+l+1}$ or is equal to 1  (in a non-graded nilpotent group, the commutator $[x,y]$ can belong to any $G_n$, $n\ge k+l$). With every nilpotent group $G$, one can associate a graded nilpotent group $G^{gr}$ as follows $G^{gr}$ is generated by the union of all sets $G_k/G_{k+1}$ subject to  all possible  relations of the form $[xG_k,yG_l]=[x,y]G_{k+l}$.

The growth function of a nilpotent group is always equivalent to a polynomial and its degree can be computed explicitly using the following formula:

\begin{theorem}[Guivarc'h \cite{Gui1}, Bass \cite{Bass}]\label{Bass} If $G$ is a finitely
generated nilpotent group, and $\gamma_S(n)$ its growth function for
some generating set $S$,  then there are constants $A,B>0$ such that
$An^d\leq\gamma_S(n)\leq Bn^d$ for all $n\geq 1$, where
$d=d(G)=\sum_{h\geq 1}hr_h$, $r_h$ is the torsion-free rank of the Abelian group
$G_h/G_{h+1}$.
\end{theorem}

In fact by the famous result of Gromov \cite{Gromov}, groups having nilpotent subgroups of finite index are precisely the finitely generated groups whose growth functions are bounded by polynomial. The paper \cite{Gromov} was the first paper where asymptotic cones of nilpotent groups were studied. It was proved later by Pansu \cite{Pa}
that all asymptotic cones of any nilpotent group are isometric to a graded Lie group canonically associated to $G$. Here is the description. Let $\mathrm{tor}(G)$ be the finite normal subgroup of $G$ generated by all elements of finite order (the fact that $\mathrm{tor} (G)$ is finite can be found, say, in \cite{Hall}). The nilpotent group $\bar{G}=G/\mathrm{tor}(G)$ is without torsion, hence it can be embedded, according to Malcev \cite{Ma55}, as a uniform (co-compact) lattice in a nilpotent Lie group. The corresponding graded nilpotent group $\hat{G}$ is also a Lie group. That group equipped with the so-called  {\em \index{Carnot-Carath\' eodory metric}Carnot-Carath\' eodory} metric\footnote{In order to define that metric, consider the tangent space $T_x$ at every point $x$ of $\hat{G}$. It is isomorphic to the Lie algebra $\mathfrak{g}=\hat G/\hat G_1\oplus \hat G_1/\hat G_2\oplus\ldots$ (with the natural commutator product) corresponding to $\hat G$, and we can identify $T_x$ with $\mathfrak{g}$ in a smooth way with respect to $x$. The elements of the first summand, $\hat G/\hat G_1$, generate $\mathfrak{g}$. The corresponding directions in $\hat G$ are called {\em \index{Horizontal path}horizontal}. A {\em horizontal path} on $\hat G$ is any path that has horizontal direction at every point. Every two points of $\hat G$ are connected by a horizontal path, and the length of a shortest such path connecting two points is the Carnot-Carath\' eodory distance between these two points. }(see, for example, \cite{GrCarnot}) is isometric to all asymptotic cones of $G$. In particular, all asymptotic cones of a nilpotent group are simply connected, and by Theorem \ref{Gro}, every nilpotent group has polynomial isoperimetric function. That result was strengthened by Gersten  \cite{Ger1}: he proved  that
$G$ admits a polynomial isoperimetric inequality of degree $2^h$, where $h$ is the
{\em \index{Hirsch length of a nilpotent group}Hirsch length of $G$} i.e. the sum of torsion-free ranks of all Abelian groups $G_i/G_{i+1}$. The degree was improved to $2\cdot 3^c$ by Conner \cite{Conner},
and then to $2c$ by Hidber \cite{Hid}. One can deduce the bound $c+1+\varepsilon$ for every $\varepsilon>0$ by the following argument.  Pittet \cite{Pi}, proved that a lattice in a simply
connected graded nilpotent Lie group of class $c$ admits a polynomial
coarse area function $A_\delta$ of degree $c + 1$.  Hence by the result of Pansu \cite{Pa}, every asymptotic cone of any nilpotent group of class $c$ has isoperimetric function $n^{c+1}$. It remains to apply Theorem \ref{pap2} of Papasoglu. It is not clear how to remove $\varepsilon$ from the estimate. The methods of Dru\c tu \cite{Drutuquadr} are based on the fact that the asymptotic cones in her situation were buildings. The first proof of the polynomial bound of degree $c+1$ was outlined by Gromov in \cite[5.A5]{Gr1}. A complete combinatorial proof was found in \cite{GHR}.

\begin{theorem}[Gersten, Holt, Riley \cite{GHR}]\label{GHS} Every finitely generated nilpotent group of class $c$ has isoperimetric function that is polynomial of degree $c+1$.
\end{theorem}

The isoperimetric inequality $n^{c+1}$ is the best possible
bound for some nilpotent groups of class $c$. For example if $G$ is a free nilpotent
group of class $c$ with at least 2 generators, then its Dehn
function is polynomial of degree $c + 1$ (see \cite{BMS} or \cite{G3} for the lower
bound and \cite{Pi} for the upper bound, see also \cite{BP} for other examples of nilpotent groups with maximal possible Dehn functions). Moreover, the following general theorem is proved in \cite{BMS, Gr1}.

\begin{theorem}[Gromov \cite{Gr1}, Baumslag, Miller, Short \cite{BMS}] Let $G$ be a nilpotent group of class $c$ and let $G_{c-1}$ be the $(c-1)$st
term of its lower central series.
If $G_{c-1}$ contains an infinite cyclic subgroup; then the quotient group $G/G_{c-1}$ is a nilpotent group of class $c-1$ and has the Dehn function $n^{c}$.
\end{theorem}

In particular, $\Z^2$ is the factor-group of the Heisenberg group $H_3$ by its derived subgroup which contains an infinite cyclic group. Hence we get another proof that  the Dehn function of $\Z^2$ is $n^2$.

\subsubsection{The Heisenberg groups}\label{thg} Nevertheless $n^{c+1}$ is not the smallest isoperimetric function for many
nilpotent groups:

\begin{theorem}[Allcock \cite{Alc}, Olshanskii-Sapir \cite{OSHeis}]\label{alc}
The
$2n+1$-dimen\-sion\-al integral Heisenberg groups admit quadratic isoperimetric
functions for $n > 1$; these groups are all nilpotent of class 2.
\end{theorem}

\begin{df} Recall that the 3-dimensional Heisenberg group has presentation $H_3=\la x,y,c\mid [x,y]=c, [c,x]=[c,y]=1\ra$. Its center is generated by $c$. The $2n+1$-dimensional Heisenberg group $H_{2n+1}$ is the $n$th central power of $H_3$. In general, the {\em \index{Central power of a group}central power} of a group $G$ is the factor-group of the $n$th direct power $G\times \ldots\times G$ by the normal (moreover, central) subgroup consisting of all vectors $(c_1,\ldots, c_n)$ with $c_1c_2\ldots c_n=1$, where all $c_i$ belong to the center of $G$.
\end{df}

The proof of Theorem \ref{alc} from \cite{OSHeis} is relatively simple but it has some features in common with very complicated computations of Dehn functions (say, \cite{GubaF, Young2, Young}), so we provide some details here.

In order to describe the proof of Theorem \ref{alc}, let $n=2$, that is let us consider the group $$\begin{array}{ll}H_5= & \la x_1, x_2, y_1, y_2, c\mid [x_1,x_2]=[y_1,y_2]=c, [x_i, y_j]=1, \\ \, & [c,x_i]=[c,y_i]=1, i,j\in \{1,2\}\ra.\end{array}$$

Consider a word $w$ in the generators $x_1,x_2, y_1, y_2$. Suppose that $w=1$ in $H_5$. We need to show that applying the relations of $H_5$ we can reduce the word to 1 so that the number of steps is $O(|w|^2)$.  First notice that we can use the commutativity relations $[x_i,y_j]=1$ and rewrite
$w$ as $w_xw_y$ where $w_x$ is a word in $x_1, x_2$, $w_y$ is a word in $y_1,y_2$. We need $O(|w|^2)$ steps to do that. Now the proof consists of several reductions.

{\bf Reduction 1.} Note that if $u(x_1,x_2)$ is a product of commutators, then $w(x_1,x_2)=w(y_1,y_2)$ in $H_5$. We prove that using at most $O(|u|^2)$ steps we can transform $u(x_1,x_2)$ to $u(y_1,y_2)$. Therefore we can restrict ourselves to words $u(x_1,x_2)$ in two variables only (but we use the other two generators for auxiliary transformations).

{\bf Reduction 2.} Every word $u(x_1,x_2)$ can be drawn on the $(x_1,x_2)$-planar square grid: $x_1$ labels horizontal edges, $x_2$ labels vertical edges (similar to \ref{vkat}). A word $u$ is a product of commutators if and only if it labels a loop. Then it is easy to see that in $H_5$, $u$ is equal to $c^a$ for some $a$ called the {\em \index{Symplectic area of a loop on the square grid}symplectic area of the region spanned by the curve}. In particular $u$ equals $1$ in $H_5$ if and only if its symplectic area is $0$.

{\bf Reduction 3.} Now let us define a symplectic area of any word $u(x_1,x_2)$, not necessarily a product of commutators. Let $k_i$, $i=1,2$ be the sum of exponents of $x_i$ in $u$. Then the word $C(u)=ux_1^{-k_1}x_2^{-k_2}$ is a product of commutators. Then the  {\em symplectic area} of $u$ is, by definition, the symplectic area of $C(u)$.

{\bf Reduction 4.} A {\em rectangular word} $R(p,q)$ where $p, q>0$ are integers is, by definition $[x_1^p,x_2^q]$. Its symplectic area is $pq$. Our goal is for every word $u(x_1,x_2)$, reduce the word $C(u)$ to a rectangular word of the same area in number of steps quadratic in $|u|$. If $u=1$ in $H_5$, then the rectangular word is empty, and we are done.

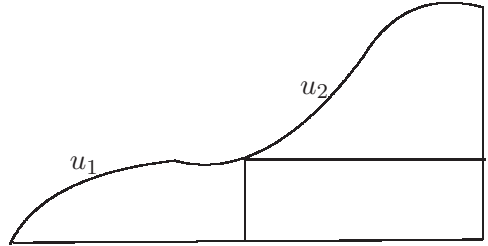
\begin{figure}[H]
\unitlength .6mm 
\linethickness{0.4pt}
\ifx\plotpoint\undefined\newsavebox{\plotpoint}\fi 
\begin{picture}(109.75,58.75)(0,0)
\qbezier(4.5,1.75)(11.75,17)(41,20.25)
\qbezier(41,20.25)(61.625,14.75)(82.75,43.25)
\qbezier(82.75,43.25)(92,58.75)(109.25,54.25)
\put(109.25,54.25){\line(0,-1){.5}}
\put(109.25,53.75){\line(0,-1){51}}
\multiput(109.25,2.75)(-4.5217391,-.0326087){23}{\line(-1,0){4.5217391}}
\multiput(5.25,2)(-.0333333,.0333333){15}{\line(0,1){.0333333}}
\put(56.5,20.75){\line(0,-1){18.5}}
\put(56.5,20.5){\line(1,0){53.25}}
\put(21,19.25){\makebox(0,0)[cc]{$u_1$}}
\put(72,36){\makebox(0,0)[cc]{$u_2$}}
\end{picture}

\caption{The division in halves trick} \label{p12}
\end{figure}

{\bf Reduction 5.} We are using the ``division in halves" trick shown on Figure \ref{p12}. The word $u$ is represented as a product of two subwords $u_1, u_2$ with $|u_i|\approx |u|/2$. Note that the loop corresponding to the word $C(u) $ on Figure \ref{p12} is decomposed into three loops: the loops corresponding to $C(u_1), C(u_2)$ and a rectangular word $R$ with sides $<|u|$. Then in order to reduce $C(u)$ to a rectangular word, we need to be able to reduce $C(u_1)$ to a rectangular word $R_1$, reduce $C(u_2)$ to a rectangular word $R_2$, and then reduce the product of three rectangular words $R_1RR_2$ to a rectangular words. The first two tasks can be done arguing by induction on the length of the words. Hence it remains to consider products of rectangular words. That is done by cutting and pasting (using Reduction 1).

Note that Reduction 1 is crucial in this proof. Here we use the fact that $H_5$ has two extra generators $y_1, y_2$ which we use as a temporary storage. Say, in order to transform $u(x_1,x_2)v(x_1,x_2)$ to $v(x_1,x_2)u(x_1,x_2)$ where $u,v$ are products of commutators, we first transform $u(x_1,x_2)$ into $u(y_1,y_2)$ then use commutativity of generators and transform $u(x_1,x_2)v(y_1,y_2)$ to $v(y_1,y_2)u(x_1,x_2)$ and then transform $v(y_1,y_2)$ back into $v(x_1,x_2)$. This transformation is used in many places in the proof.

Note that $H_3$ has cubic Dehn function (proved in \cite{Eautomatic}).

\subsubsection{Other central products.}\label{ss:ocp}
The central powers of nilpotent groups of class 2 have very small Dehn functions. The following result is mentioned in our paper with A. Olshanskii \cite{OSHeis}. Its proof is essentially the same as the proof for the Heisenberg group $H_{2n+1}$. Another proof is given in R. Young \cite{Young1}. Young's approach is in a sense a combination of approaches from \cite{Alc} and \cite{OSHeis}.

\begin{theorem}\rm{(}\cite[Remark on page 927]{OSHeis}\rm{)}. Let $G$ be any finitely generated nilpotent group of class 2. Then
\begin{enumerate}
\item[(a)] The central square of $G$ admits isoperimetric function $n^2\log n$.
\item[(b)] Some central power of $G$ has quadratic isoperimetric function.
\end{enumerate}

\end{theorem}

We were unable to show that the central square of every nilpotent group of class 2 has quadratic isoperimetric function. More concretely,  let $FN(2,10)$ be the \index{Group!free nilpotent of class 2}free nilpotent group of class $2$ generated by $a_1,b_1,\ldots, a_5,b_5$ (it is given by all the defining relations of the form $[[x,y],z]=1$, where $x,y,z\in \{a_1,\ldots, a_5, b_1,\ldots, b_5\}$). Let $G_{10}$ be the factor-group of $FN(2,10)$ by the central cyclic subgroup generated by $[a_1,b_1]\ldots [a_5,b_5]$. Then we were not able to prove that the central square of $G_{10}$ has quadratic Dehn function (this example and the question can be found in \cite[Section 5]{Young1}).

\subsubsection{Wenger's result}\label{ss:wr} Recently S. Wenger \cite{Wen2} showed that in fact central squares of nilpotent groups of class 2 often do not admit  quadratic isoperimetric inequality. In particular, he proved

\begin{theorem}[Wenger \cite{Wen2}]\label{Wen} The Dehn function of the central square of $G_{10}$ is strictly greater than quadratic.
\end{theorem}

The idea of the proof is the following. Let $H$ be the central square of $G_{10}$. Suppose that $H$ has quadratic Dehn function. Then by Theorem \ref{Wen2}, in the asymptotic cone of $H$, every Lipschitz loop $\gamma$ must bound an integral 2-current of area $c|\gamma|^2$ for some constant $c$. Wenger gets a contradiction by finding a loop that does not bound an integral 2-current at all. In the proof, he essentially uses Pansu's description of asymptotic cones of nilpotent groups as nilpotent Lie groups with a Carnot-Carat\'eodory metric (see \cite{Pa} and \ref{ss:tacatcp}). Assuming that his loop bounds an integral 2-current $\zeta$, he approximates $\zeta$ by singular Lipschitz chains. He then uses the fact that every  Lipschitz map from a disc $D^2$ to a Carnot-Carath\'eodory space is differentiable almost everywhere (Pansu \cite{Pa}).

Therefore Wenger obtained the first example of a finitely generated nilpotent group whose Dehn function is not equivalent to a polynomial (compare with Guivarc'h-Bass' Theorem \ref{Bass} about the growth functions).

\begin{prob} Find a combinatorial proof of Theorem \ref{Wen}. Is the Dehn function of the central square of $G_{10}$ equivalent to $n^2\log n$?
\end{prob}

\subsubsection{Other nilpotent groups}\label{ss:yr} Note also that R. Young \cite{Young2} constructed nilpotent groups of arbitrary class with quadratic Dehn functions.

\subsection{Semihyperbolic groups}\label{s:sg}

Semihyperbolic groups were defined by Gromov \cite{GrHyp} and with more precision by Alonso and Bridson  \cite{AB}. Let $X$ be a metric space. Let $A\ge 1,B\ge 0$ be two numbers. A {\em \index{Discrete path}discrete $(A,B)$-path} in $X$ is any $(A,B)$-quasi-isometry $f$ from an interval of natural numbers $[0,T_f]$ into $X$. We shall extend $f$ to the  whole $\N$ by $f(n)=f(T_f)$ for $n\ge T_f$. Suppose that for some $A,B$ we can choose a discrete $(A,B)$-path $p_{a,b}$ between any pair of points $a,b\in X$ in such a way that for every four points $a,b, a', b'$ and every $t\in \N$ we have the {\em \index{Combing condition}combing condition}: \begin{equation}\label{cco}\dist(p_{a,b}(t),p_{a',b'}(t))\le C_1\max(\dist(a,a'),\dist(b,b'))+C_2\end{equation} for some constants $C_1, C_2$. Then we say that $X$ is {\em \index{Semihyperbolic metric space}semihyperbolic}. It is easy to prove that every hyperbolic metric space is semihyperbolic ({\bf exercise:} prove that using asymptotic cones) and that a metric space that is quasi-isometric to a semihyperbolic space is semihyperbolic space itself. A finitely generated group $G$ is called {\em \index{Group!semihyperbolic} semihyperbolic} if its Cayley graph with respect to some (equivalently, any) finite generating set is semihyperbolic and $p_{ga,gb}=gp_{a,b}$ for every $g,a,b\in G$. In that case we can consider the language of words labeling the paths $p_{1,b}$ for all $b\in G$. If that language is regular, then the semihyperbolic group is called {\em \index{Group!bi-automatic}bi-automatic}. If the language is regular, but we only assume that  the combing property (\ref{cco}) holds provided $a=a'$, then the group is called {\em \index{Group!automatic}automatic} \cite{Eautomatic}.

\subsubsection{CAT(0)-spaces} \label{cat0} Let $X$ be a geodesic metric space, $ABC$ be a geodesic triangle in $X$ with side lengths $p,q,r$. We can construct the triangle $A'B'C'$ in the Euclidean space $\R^2$ with the same side lengths. If $a,b$ are points on the sides $AB$, $BC$ of the triangle $ABC$, let $a',b'$ be the corresponding points on the sides $A'B'$ and $B'C'$ of $A'B'C'$. If for every choice of the triangle $ABC$ and points $a,b$ we have $\dist(a,b)\le \dist(a',b')$, then $X$ is called a {\em \index{CAT(0) metric space}CAT(0) space} \cite{Gr1,BH}. In every CAT(0) space there exists unique geodesic $p_{x,y}$connecting any two points $x,y$. These geodesics satisfy the combing property (\ref{cco}) (see \cite{BH}). Hence every CAT(0) metric space is semihyperbolic.  The class of CAT(0) spaces includes
all hyperbolic spaces, complete simply connected Riemannian manifolds all of
whose sectional curvatures are non-positive (that is, {\em \index{Hadamard manifold}Hadamard manifolds}), including symmetric spaces of non-compact type,  Euclidean buildings (see \cite{Brown2}).

Every group acting properly discretely and co-compactly on a CAT(0)-space is semihyperbolic. This implies that many non-hyperbolic groups are semihyperbolic. For example:

\begin{enumerate}
\item All Coxeter groups \cite{AB} (the fact that these groups are automatic was proved in \cite{BrHo}).
\item \index{Group!right angled Artin}Right angled Artin groups\footnote{A right angled Artin group is a group given by a presentation  of the form $\la a_1,\ldots,a_n\mid a_ia_j=a_ja_i \hbox{ for some } i,j \in \{1,\ldots,n\}\ra$} and many other Artin groups (see \cite{Char}) including the direct product of two free groups (see \ref{ss:qeatpc}).
\item The mapping class group\footnote{The \index{Mapping class group}mapping class group of a surface $S$ is the group of all isotopy classes of orientation preserving self-homeomorphisms
of $S$. This group is finitely presented for every $S$.} of a surface with genus $g$ and $n$ punctures provided $3g-3+n\ge 2$ (the fact that these groups are automatic was proved by Mosher \cite{Mosh}, the bi-automaticity was proved by Hamenst\"adt \cite{Ham}).
\item All co-compact lattices in $\SL_n(\R)$ (and many other Lie groups) because the corresponding symmetric space $\SL_n(\R)/\SO_n(\R)$ is CAT(0).
\item All lattices in $\SO_{n,1}(\R)$ \cite{BH}.
\end{enumerate}

\subsubsection{Semihyperbolic metric spaces have quadratic isoperimetric inequality}\label{ss:smshqii} Indeed, consider every loop $\gamma$ of length $l$ and point $a$ on that loop. Subdivide the loop into a sequence of subarcs of length $\delta$ (for some constant $\delta$) by points $b_1,\ldots, b_n$, and connect $a$ with $b_1,\ldots, b_n$ by the $(A,B)$-quasi-geodesics $p_{a,b_i}$. The length of each of these quasi-geodesics is at most $O(l)$. Then by the combing property (\ref{cco}) for every $t\in \N$ the distance between $p_{a,b_i}(t)$ and $p_{a,b_{i+1}}(t)$ is bounded by some constant $C=C(\delta)$. Thus we obtain a collection of points $p_{a,b_i}(t)$, $t\in \N$, containing the points $b_1,\ldots, b_n$, having mesh $O(1)$. The number of points in that set is at most $O(l^2)$. Hence the Gromov coarse area function $A_\delta(l)$ is bounded by $O(l^2)$ for a sufficiently large $\delta$.

A very similar proof shows that automatic groups have quadratic isoperimetric functions too.

\subsubsection{The conjugacy problem in semihyperbolic groups}\label{ss:tcpisg} Another key property of semihyperbolic group is the following

\begin{theorem}[Alonso-Bridson \cite{AB}]\label{ab} Every semihyperbolic group has solvable  conjugacy problem.
\end{theorem}

\proof Let $G$ be a semihyperbolic group, $u,v\in G$. Suppose that $x\iv ux=v$ for some $x\in G$, and the length $|x|$ of $x$ is the smallest possible. Clearly it is enough to bound $|x|$ as some recursive function in $|u|+|v|$. We show that one can take an exponential function $\exp (c(|u|+|v|)$ for some constant $c$. Indeed, let $x=a_1a_2\ldots a_m$ be a product of generators and their inverses. Note that in the Cayley graph of $G$ we have two paths $p_1, p_2$ labeled by $x$, $p_1$ connecting $1$ and $x$, and $p_2$ connecting $u$ and $ux$. Since $xv=ux$, the initial points of $p_1$ and $p_2$ are at distance $|u|$ from each other and the terminal points are at distance $|v|$ from each other. By the combing property (\ref{cco}), then there exist paths $q_i$, $i=1,\ldots,m$ connecting $a_1\ldots a_i$ with $ua_1\ldots a_i$ of length $\le C(|u|+|v|)$ for some constant $C$. If $m$ is larger than the number of different words in the generators of $G$ of length $C(|u|+|v|)$ (which is $\exp (c(|u|+|v|)$ for some constant $c$), then two of these paths, say $q_i, q_{i+j}$ have the same labels. Then let $y=a_1\ldots a_ia_{i+j+1}\ldots a_m$. It is easy to see that $y\iv uy=v$ which contradicts the minimality of $x$. \endproof

Geometrically this proof can be described as follows. Cut a mimimal annular diagram $\Delta$ by a simple path $p$ connecting the boundary components. Using the combing property (\ref{cco}), for every vertex of $p$ we can find a (non-null-homotopic) loop in $\Delta$ based at that point and having length at most $C$ times the sum of lengths of the boundary components.  We can assume that these loops do not cross each other.  If $p$ is too long, two of these loops will have the same label. Cutting off the annular subdiagram bounded by these two loops from $\Delta$, we obtain another, smaller, annular van Kampen diagram with the same boundary labels (see also the proof of Lemma \ref{lmnlogn} below).

\section{Dehn functions and algorithmic properties of groups}\label{dfaapog}

\subsection{The zoo of groups with quadratic Dehn functions.}\label{s:tzogwqdf}

We have already mentioned that automatic and semihyperbolic groups have quadratic isoperimetric functions. In fact the set  of finitely presented groups with quadratic isoperimetric functions is quite diverse.

\begin{theorem}\label{th:qdf} The following groups have quadratic Dehn functions:
\begin{itemize}
\item[(1)] Dru\c tu's examples of {\em \index{Group!polycyclic}polycyclic}\footnote{A group $G$ is called {\em polycyclic} if it has a series of subnormal subgroups $G=G_0\unrhd G_1\unrhd\ldots \unrhd G_n=\{1\}$ with cyclic factors $G_i/G_{i+1}$.}non-nilpotent groups from \cite{Drutuquadr} (see also
\cite[Section 5.A.9]{Gr1}),
\item[(2)] Cornulier-Tessera examples of non-polyciclic metabelian groups with quadratic Dehn functions (some of these groups contain the Baumslag-Solitar group $\mathrm{BS}(1,2)$) \cite{CT}.
\item[(3)]
The Heisenberg groups of rank $\ge 5$, sufficiently large central power of any (infinite) finitely generated nilpotent group of class 2, some finitely generated  nilpotent groups of arbitrary nilpotency class  (Alcock \cite{Alc}, Olshanskii-Sapir \cite{OSHeis}, Young \cite{Young1, Young2}).
\item[(4)] All extensions of finitely generated free groups by cyclic groups (Bridson, Groves \cite{BrGr}).
\item[(5)] The Stallings group \cite{Stal} (the kernel of the homomorphism from the direct product of three free groups of rank 2 to $\Z$ that sends all generators to the generator of $\Z$) (Dison, Elder, Riley, Young \cite{DERY})\footnote{Note that Stallings' group was the first example of a finitely presented group which is not of type $FP_3$. Thus by \cite{DERY} and a previous results of Riley \cite{Rie} and Papasoglu \cite{Pap1} this group is the first example of a finitely presented group  with simply  connected but not 2-connected asymptotic cones.}
\item[(6)] The groups ${\mathrm SL}(n,\Z)$, $n\ge 5$ (Young \cite{Young}).
\item[(7)] The R. Thompson group $F$ (Guba \cite{GubaF}) \footnote{The R.Thompson group $F$ consists of all piecewise-linear increasing bijective functions $[0,1]\to [0,1]$ whose derivatives have finitely many dyadic break points and the slopes of linear pieces are powers of 2. The operation: composition of functions. That remarkable group has a presentation with two generators and two defining relations \cite{CFP}.}.
\end{itemize}
\end{theorem}

The proofs of the quadratic isoperimetric inequality for each of these classes of groups is quite complicated and quite different.

\subsubsection{}\label{Dru} Dru\c tu's groups from Part (1) of the theorem are uniform lattices in solvable Lie groups of class 2. She proves quadratic isoperimetric inequality for these Lie groups using asymptotic cones (these are horoballs in direct products of Euclidean buildings) and Theorem \ref{pap2}.

\subsubsection{} Nilpotent groups from part (2) have been discussed already in Section \ref{ng}.

\subsubsection{} The proof for free-by-cyclic groups is very long and compicated. Here is a short description of that proof  sent to us by Daniel Groves.

Let $G = F_n \rtimes_{\phi} \mathbb Z$ be a free-by-cyclic group, with
natural HNN-presentation
\[ \mathcal P = \langle x_1, \ldots,  x_n , t \mid tx_i t\iv =
\phi(x_i) \rangle   .  \] Let $\Delta$ be a minimal diagram over this presentation. Consider the maximal $t$-bands in $\Delta$. By Lemma \ref{ms}, every maximal $t$-band starts and ends on the boundary of $\Delta$. Every HNN-cell
has a top side (corresponding to a letter $x_i$) and a bottom side (corresponding to
$\phi(x_i)$).  Passing from the bottom to the top corresponds to
applying the automorphism $\phi$.  It is easy to see that a pair of HNN-cells cannot be
glued along their bottoms in a minimal $\Delta$, so each 2-cell
can be traced backwards through $\Delta$ via the edge on the bottom sides of HNN-cells
until it hits an edge on the boundary of $\Delta$.  If such an edge $e \in
\partial \Delta$ is fixed, then all the 2-cells corresponding to $e$
form a single {\em color} (and there are fewer than $| \partial \Delta |$
colors). The proof involves a detailed investigation of the two
transverse structures coming from bands and colors.  The
evolution of a color and of different colors is closely related to
the dynamics of words in $F_n$ under iterates of $\phi$, although the
interaction of bands complicate this substantially.

\subsubsection{}\label{ss:dery} The proof of quadratic isoperimetric inequality of the Stallings group is purely combinatorial. The authors in \cite{DERY} present an algorithm of creating a van Kampen diagram with quadratic area and a given boundary label.

\subsubsection{} \label{emp}The problem of whether the Dehn function of $SL_n(\Z)$ is quadratic for large enough $n$ was open for quite some time. The group $\SL_2(\Z)$ is hyperbolic (in fact virtually free), so its Dehn function is linear. It was proved in \cite{Eautomatic} that $SL_3(\Z)$ has exponential Dehn function. The fact that the Dehn function of $\SL_n(\Z)$ is at most exponential for every $n$ was stated in Gromov \cite[5.A7]{Gr1} and fully proved by
Leuzinger \cite{Leu}. The book \cite{Eautomatic} contains a statement attributed to Thurston that for $n\ge 4$ the Dehn function is quadratic. Recently Robert Young proved in \cite{Young} that the Dehn function is quadratic for $n\ge 5$. The case $n=4$ is still open.

To explain Young's proof, we need some preliminaries. Note that $SL_n(\R)$ is quasi-isometric to the {\em \index{Symmetric space}symmetric space} $\SL_n(\R)/\SO_n(R)$, and the word metric in $\SL_n(\Z)$ is quasi-isometric to the restriction of the Riemannian metric on $\SL_n(\R)$ to $\SL_n(\Z)$ by Lubotzky-Mozes-Raghunathan \cite{LMR}. The action of $\SL_n(\Z)$ on the symmetric space $\SL_n(\R)/\SO_n(R)$ is not co-compact. The typical case is when $n=2$. In that case the symmetric space $\SL_2(\R)/\SO_2(\R)$ is just the Hyperbolic plane $\HH^2$, say the upper half of the complex plane with the hyperbolic metric. The group $\SL_2(\R)$ acts on $\HH^2$ by the M\"obius transformations: $$\left(\begin{array}{cc} a &b \\ c& d\end{array}\right)\cdot z=\frac{az+b}{cz+d}.$$ Then the action $\SL_2(\Z)$ on $\HH^2$ has a fundamental domain which is an ideal  triangle (with one vertex at infinity). Consider the set $H_\infty$ of all $z\in \HH^2$ with imaginary part at least $1$ (the number 1 can be replaced by any real number $> 0$). It is a {\em \index{Horoball}horoball} in $\HH^2$. Note that $H_\infty$ is stable under the action of the subgroup of $\SL_2(\R)$ consisting of upper uni-uppertriangular matrices from $\SL_2(\R)$ of the form $\left(\begin{array}{cc} 1 &b \\ 0& 1\end{array}\right)$. The orbit of $H_\infty$ under the action of the group $\SL_2(\Z)$ consists of horoballs too. The union of these horoballs is the {\em \index{Symmetric space!thin part of}thin} part of the symmetric space $\HH^2$. The complement of the thin part is the {\em \index{Symmetric space!thick part of}thick part}. The thick part is clearly invariant under the action of $\SL_2(\Z)$ and the action is co-compact. A similar picture exists if $n\ge 3$. The symmetric space $\SL_n(\R)/\SO_n(\R)$ has a horoball $H_\infty$. The union of the horoballs $\SL_n(\Z)\cdot H_\infty$ is called the {\em thin part} of the symmetric space, its complement is the {\em thick part}, and the action of $\SL_n(\Z)$ on the thick part is co-compact. Thus it is natural to consider ``equivalent objects": the group  $\SL_n(\Z)$ and the thick part of the corresponding symmetric space.

These preliminaries should be enough to understand a short description of the proof from \cite{Young} sent to us by Robert Young. Every word $w$ that is equal to 1 in $\SL_n(\Z)$ corresponds to a loop in
the symmetric space $\SL_n(\R)/\SO_n(\R)$. The basic idea of the proof is to use fillings of loops in the
symmetric space $\SL(p;\R)/\SO(p)$ (which has quadratic area function being semihyperbolic) as templates for fillings loops in the Cayley graph of
$\SL_n(\Z)$.  Fillings which lie in the thick part of
$\SL_n(\R)/\SO_n(\R)$ correspond directly to fillings in $\SL_n(\Z)$, but in
general, an optimal filling of a curve which is contained in the thick part may have to
go deep into the thin part. The proof uses both geometric and algebraic points of view.
Different horoballs in the thin part correspond to different parabolic
subgroups of $\SL_n(\Z)$ (each of these  subgroups is conjugate to some group of block upper triangular matrices), so one key step was to develop geometric techniques to cut the quadratic filling of a loop in the symmetric space into pieces (subdiscs) with each piece lying in one horoball. The word corresponding to the boundary of each piece is a product of three words, each of which represents an
element of the parabolic subgroup (these elements are determined by
the geometry of the filling).  Each of these three words can be written as a
product of a block diagonal matrix and a unitriangular matrix.
The uni-triangular
matrix can be written as a product of boundedly many elementary matrices (one non-zero off-diagonal entry) where the off-diagonal entry may be of exponential size. The elementary matrices  can be replaced by relatively short (linear length) words in $\SL_n(\Z)$ ({\em shortcuts}) using the technique
from Lubotzky-Mozes-Raghunathan \cite{LMR}. For this purpose Young includes the elementary matrices in some solvable subgroups  with quadratic Dehn functions where the cyclic subgroups generated by these elements are exponentially distorted. This is not very mysterious. We have already seen (see \ref{ss:poe}) that the cyclic group generated by $a$ is exponentially distorted in $\mathrm{BS}(1,2)$. The group $\mathrm{BS}(1,2)$ is metabelian and has the following matrix representation $b=\left(\begin{array}{ll}2 &0\\ 0& 1\end{array}\right)$, $a=\left(\begin{array}{ll}1 &1\\ 0 & 1\end{array}\right)$. Note that $a$ is an elementary matrix in this representation. Of course as we have seen (Example \ref{example2}), the Baumslag-Solitar group has exponential Dehn function. But Dru\c tu \cite{Drutuquadr}, and later Cornulier and Tessera \cite{CT} have shown that metabelian groups of higher ranks can have quadratic Dehn functions. The key property of Young's solvable subgroups is the freedom of choosing them. This is similar to having two commuting copies of the Heisenberg group $H_3$ in $H_5$ in the proof from \ref{thg} and this is why Young's proof does not work in $\SL_4(\Z)$: there is simply not enough room to place several commuting solvable subgroups with quadratic Dehn functions. The shortcuts allow Young to reduce the problem of filling the original loop to the filling problem for loops in some subgroups of block-diagonal matrices (which are just direct products of matrix groups of smaller sizes). Repeatedly applying these geometric and combinatorial techniques breaks the original word into pieces which lie in smaller and smaller subgroups of $\SL_n(\Z)$ and ultimately to a quadratic filling.

\subsubsection{}\label{ss:F} The result that the R.Thompson group $F$ has quadratic Dehn function \cite{GubaF} was quite unexpected. It was conjectured by Gersten \cite{GerF} that the Dehn function of $F$ is exponential. It was proved to be at most $\exp(\log^2n)$ in our paper with Guba \cite{GSF}. For some time we thought that $\exp(\log^2{n})$ is equivalent to the Dehn function of $F$ because we had a concrete sequence of loops in the Cayley graph of $F$ for which we could not find fillings with smaller areas. A big breakthrough was achieved by Guba in \cite{GubaF0} where a polynomial isoperimetric function was found. Still it was hard to believe that the Dehn  function is quadratic because the polynomial fillings used in \cite{GubaF0} seemed optimal.

Recall that the R. Thompson group $F$ has infinite presentation \begin{equation}\label{eq6} \la x_0,x_1,\ldots\mid x_i\iv x_jx_i =x_{j+1} \forall i<j\ra.\end{equation}  It is also generated by $x_0,x_1$ subject to 2 defining relations. It turned out that instead of dealing with a finite presentation of $F$ one is better off working with the infinite presentation obtained from (\ref{eq6}) by considering only those relations where $j-i\le 5$. It is not difficult to prove that this (infinite) presentation defines $F$ and the quadratic isoperimetric inequality for this presentation is equivalent to a quadratic inequality for a finite presentation of $F$. (In defining an isoperimetric inequality for the infinite presentation, one needs to modify the notion of length of a word; Guba uses a notion of {\em complexity} of a word which is the length plus the difference between the highest and the lowest indices of letters occurring in the word.)  Guba's proof is in a sense similar to the proof from \cite{OSHeis} outlined in \ref{thg}. He also introduces normal forms of elements in $F$ (in the case of $H_5$ these were the rectangular words $R(p,q)$), and considers the process of reducing a word to its normal form rather than reducing a word that is equal to 1 in $F$ to the empty word. Using a trick similar to the ``division in halves trick" from \ref{thg}, Guba reduces the problem to finding the filling area of words of the form $uv$ where $u,v$ are normal forms. These are called the {\em triangular words} in \cite{GubaF}. He then transforms these words to words of different shapes (rectangular, vertical, horizontal). In order to do that  Guba uses various elaborate cutting and pasting techniques. Some of these  amount to dividing a word into three parts which are of approximately (but not quite) equal lengths. Instead of three pieces as on Figure \ref{p12}, Guba gets 8 pieces, and the following recurrent formula for the isoperimetric function $f(n)$:
$$f(n)\le 8f(n/3)+O(n^2). $$ The fact that the coefficient of $f(n/3)$ in that formula is smaller than $9=3^2$ plays crucial role. It turns out that the reduction procedure described in \cite{GubaF} not always decreases the complexity of a word.  In order to make increases rare enough, one needs to cut the words into three parts of not quite equal lengths and the required freedom is allowed by the fact that $8<9$.

\subsubsection{Rips' question}\label{ss:rc} Note that the following problem formulated by Rips in 1994 is still open.

\begin{prob}[Rips] Is $F$ automatic?
\end{prob}

It is known \cite{GSF} that $F$ admits a regular language of (unique) normal forms. But that set of normal forms does not satisfy the combing property (\ref{cco}).

Also note that in \cite{GubaFTV}, Guba proved that the simple R.Thompson groups $T$ and $V$ \cite{CFP} have polynomial isoperimetric functions (these groups are finitely presented by \cite{Hig}). Exact Dehn functions of these groups are unknown. Simple finitely presented groups with quadratic Dehn functions (and even bi-automatic) are known (see Burger and Mozes \cite{BM1,BM2}).

\begin{prob} Is $F$ semihyperbolic?
\end{prob}

\subsection{The conjugacy problem for groups with quadratic Dehn functions}\label{tcp}

Semihyperbolic groups and all the groups mentioned in Theorem \ref{th:qdf} have one thing in common: the conjugacy problem in any of these groups is solvable. For semihyperbolic groups the proof is easy (see Theorem \ref{ab}).  The proofs for the different classes of groups mentioned in Theorem \ref{th:qdf} are quite different, and sometimes very complicated.

\subsubsection{}\label{ss:for} For polycyclic (in particular, nilpotent) groups, solvability of the conjugacy problem was proved by Formanek \cite{For} using representation of these groups by matrices.

\subsubsection{}\label{ss:nos} For finitely presented metabelian groups, the solvability of conjugacy problem was proved by Noskov \cite{Nos} using the Magnus representation of metabelian groups by matrices over commutative rings, and some very non-trivial commutative algebra (finding groups of units of commutative rings, in particular).

\subsubsection{}\label{ss:fb} For free-by-cyclic groups it was proved by Bogopolski, Martino, Maslakova and Ventura \cite{BMMV} using some deep properties of the dynamics of free group automorphisms, in particular, Bestvina-Feighn-Handel train tracks \cite{BFH}.

\subsubsection{}\label{ss:br3} For the Stallings group, the solvability of  conjugacy problem is proved by Bridson in \cite{Br3}. He applies the following nice result (we weaken the formulation a little bit in order not to introduce new terminology).

\begin{theorem}[Bridson \cite{Br3}] Let $G$ be a bi-automatic group and let $N\lhd G$ be a normal subgroup. If the membership problem for finitely generated subgroups of $G/N$ is solvable, then $N$ has a solvable conjugacy problem.
\end{theorem}

Recall that the {\bf input} of the \index{Membership problem for finitely generated subgroups}membership problem for finitely generated subgroups of a group $G$ is a tuple of elements of $G$ (given as words in generators) $u_1,\ldots, u_n,v$. The {\bf output} is ``yes" if $v$ belongs to the subgroup generated by $u_1,\ldots, u_n$.

\subsubsection{}\label{ss:gru} For the group ${\mathrm SL}(n,\Z)$, the solvability of conjugacy problem was proved independently by Grunewald \cite{Gru} and Sarkisyan \cite{Sar}. The diversity of methods used in these papers is reflected by the following paragraph taken from Roger Lyndon's review of \cite{Gru} in Math Reviews: ``The arguments, given explicitly, are long and intricate but appear for the most part to use hardly more than standard ideas from the theory of linear groups and their modules, from number theory, and, from the theory of groups, coset representations and the Reidemeister-Schreier theorem."

\subsubsection{}\label{ss:t9} The conjugacy problem in the R.Thompson group $F$ was  first solved by Guba and myself in \cite{GSDG} using the diagram group representation of $F$. More recently a different solution was found by Belk and Matucci \cite{BeMa}.

\subsubsection{Rips' problem}\label{ss:rp2} These results justify the following general problem which was formulated in 1994 as a conjecture (before most of the above results confirming this conjecture appeared):

\begin{prob}[E. Rips] Is the conjugacy problem solvable in every finitely presented group with quadratic Dehn function?
\end{prob}

The fact that this problem is difficult  is shown by the diversity of methods used in proving solvability of the conjugacy problem mentioned above, and the following ``quasi-proof" from our paper with A. Olshanskii \cite{OSnlogn}.

In order to formulate the statement, we shall need the well known constructive version of a limit of a
sequence of numbers. In fact we are going to use that definition
only in the case when the limit is $0$.

\begin{df} {\rm Let $g\colon \N\to \R$ be a function. We say that the
{\em \index{Constructive limit of a sequence}constructive limit} of $g(n)$ as $n\to \infty$ is $0$ if for every
integer $A>0$ there exists $N=N(A)$ such that for every $n>N$,
$|f(n)|\le 1/A$, and the function $N(A)$ is recursive. In that case
we shall write $\lim_{n\to\infty}^c g(n)=0.$ It is easy to see that
$\lim_{n\to\infty}^c g(n)=0$ if and only if there exists an
increasing recursive function $f(n)$ such that $g(k)\le \frac{1}{n}$
for every $k\ge f(n)$.} \label{dfcon}
\end{df}

\begin{ctt} \label{ct1}Let $d(n)$ be the Dehn
function of a finite group presentation $P$. Suppose that
$\lim_{n\to\infty}^c \frac{d(n)}{n^2\log n}=0$. Then $P$ has
decidable conjugacy problem.
\end{ctt}

\proof We shall need the following lemma.

\begin{lemma}\label{lmnlogn} Let $\Delta$ be a minimal
annular diagram with boundary components $p, p'$ over a finite group
presentation $P$. Let $x$ be a shortest path connecting $p$ and
$p'$. Then the area of $\Delta$ is at least $C|x|\log |x|$ for some
constant $C$ depending on $P$.
\end{lemma}

\proof Consider the following construction. Let $p_0=p$ (considered
as a cyclic path) be the inner contour of the diagram $\Delta$.
Suppose that we have constructed a cyclic path $p_i$ surrounding the
hole of the diagram in $\Delta$ such that $p_i$ does not have common
vertices with $p'$. Let $K_i$ be the annulus bounded by $p_0$ and
$p_i$. Let $M_{i+1}$ be the set of cells of $\Delta$ outside $K_i$
that have common vertices with $p_i$. Then let $K_{i+1}$ be the
minimal annular subdiagram of $\Delta$ with simple contours that
contains $K_i$ and all cells from $M_{i+1}$. Let $p_{i+1}$ be the
outer contour of $K_i$ (the inner contour of $K_i$ is $p=p_0$).

It follows that every edge of the path $p_{i+1}$ belongs to the
contour of one of the cells of $M_{i+1}$. Hence every vertex of
$p_{i+1}$ can be connected with a vertex of $p_i$ by a path such
that

(0) the length of the path is bounded by a constant,

(1) it can be connected with $p_0$ by a path of length at most
$O(i)$ and

(2) the number of cells in $M_{i+1}$ is at least $O(|p_{i+1}|)$.

From (1), it follows that the number of subdiagrams $K_i$ is
$O(|x|)$. Furthermore, more than a half of the paths $p_i$ have
length at most $\log_c |x|$ where $c$ is, say, four times the number
of letters in the alphabet of the presentation $P$. Indeed,
otherwise we would have two paths $p_i$ and $p_j$, $i\ne j$ with the
same labels, and we could remove the annular subdiagram between
$p_i$ and $p_j$ reducing the area of $\Delta$ (that would contradict
the minimality of $\Delta$), see an analogous trick in the proof of Theorem \ref{ab} above.

From (2), it follows that at least half of the subsets $M_i$ contain
at least $O(\log |x|)$ cells each. Since these sets do not
intersect, the number of cells in $\Delta$ is at least $O(|x|\log
|x|)$.
\endproof

Now the ``quasi-proof" of the Quasi-Theorem \ref{ct1} proceeds as
follows. Suppose that $P$ is a finite presentation with
undecidable conjugacy problem. Suppose that the constructive limit
of $\frac{d(n)}{n^2\log n}$ is $0.$ Then, in particular, $d(n)$ is
bounded from above by a recursive function, and $P$ has solvable
word problem.

Note that if an annular diagram $\Delta$ with boundary labels $u$ and
$v$ has a simple path $x$ with label $t$ connecting the boundary components,
then we can cut $\Delta$ along $x$ and obtain a disc \vk diagram
with boundary label $t^{\pm 1}ut^{\mp 1}v\iv$. So if $|t|$ is
recursively bounded in terms of $|u|$ and $|v|$ (for every $u$ and
$v$ that are conjugate modulo $P$) then the conjugacy of $u$ and $v$
can be algorithmically verified.

Pick an increasing recursive function $f(n)$ with
$\frac{d(3k)}{Ck^2\log k}< \frac1n$ for every $k>f(n)$ where $C$
is the constant from Lemma \ref{lmnlogn} (as in Definition
\ref{dfcon}). Since the conjugacy problem for $P$ is undecidable,
there exists a minimal annular diagram $\Delta$ with contours $p$,
$p'$ such that any path in $\Delta$ connecting $p$ and $p'$ has
length at least $f(|p|+|p'|)$. Let $n=|p|+|p'|$. Let $x$ be a
shortest path connecting $p$ and $p'$. Thus \begin{equation}|x|\ge
f(n)\label{xn},\end{equation} and so

\begin{equation}\label{eqcon}\frac{d(3|x|)}{C|x|^2\log |x|}<\frac1n.\end{equation}

Since $x$ is a shortest path connecting $p$ and $p'$, $x$ is simple.
Let us cut $\Delta$ along $x$ and obtain a disc diagram $\Gamma$
with boundary label $zuz\iv v\iv$ where $z$ is the label of $x^{\pm
1}$. By Lemma \ref{lmnlogn}, the area of $\Gamma$ is at least
$C|x|\log|x|$.

Now we can take an integer $m$ between $\frac{|x|}{n}-1$ and
$\frac{|x|}{n}$. We attach $m$ copies of $\Gamma$ consecutively to
each other along the sides labeled by $z$ to get a \vk diagram
$\Pi$ with boundary label $zu^mz\iv v^{-m}$. The
perimeter $r$ of $\Pi$ is between $2|x|$ and $3|x|$, and the area
is $m$ times the area of $\Delta$. So, by Lemma \ref{lmnlogn}, the
area of $\Pi$ is at least $\frac{C|x|^2\log |x|}{n}$. By
(\ref{eqcon}), we can deduce that the area of $\Pi$ is bigger than
$d(r)$. This contradicts the definition of Dehn function of a
group presentation.
\endproof

\begin{rk}\label{rk91} The only gap in the preceding argument is contained in the last
phrase. We cannot guarantee that
$\Pi$ has minimal area among all diagrams with the same boundary
label, and, in principal, the area of a minimal diagram with this
boundary label may be even quadratic in terms of the perimeter
$r$. Still we do not know any groups for which this proof does not
work. Note that we do not need $\Pi$ to be minimal: only that the
minimal diagram with the same boundary label does not have too few
cells compared to $\Pi$. Also note that $\Pi$ cannot be ``obviously" non-minimal. For example, it is easy to see that $\Pi$ cannot contain two cells that share an edge and are mirror images of each other (if such a pair of cells existed, we would replace their union by a diagram with no cells, reducing the number of cells). Also we have freedom of choosing $u, v$, $\Delta$ and $x$. We do not need $x$ to be a minimal length path
connecting the boundary components of $\Delta$. We only need that
the area of $\Delta$ exceeds $O(|x|\log |x|)$ divided by a
recursive function in $n$ (depending only on the presentation). In
addition, the number $m$ should only be $O(\frac{|x|}{|u|+|v|})$.
Thus Quasi-theorem \ref{ct1} seems true for a very large class of
groups and possibly for all finitely presented groups.
\end{rk}

\subsubsection{The sharp bound of $n^2\log n$}\label{ss:tsbn} It is shown in \cite{OSnlogn} that Quasi-theorem \ref{ct1} is true for groups that are
multiple HNN-extensions of free groups (for the definition see \ref{hnne}).

\begin{theorem}[See \cite{OSnlogn}]\label{quadconj} Let $d(n)$ be the Dehn
function of a multiple HNN-extension of a free group. Suppose that
$\lim_{n\to\infty}^c \frac{d(n)}{n^2\log n}=0$. Then $P$ has
decidable conjugacy problem.
\end{theorem}

The bound $n^2\log n$ in Theorem \ref{quadconj} turned out to be sharp.

\begin{theorem}[See \cite{OSnlogn}]\label{quadconj1} There exists a finitely presented  group that
\begin{itemize}
\item  is a multiple HNN-extension of a free group,
 \item has undecidable conjugacy problem,
 \item has Dehn function $n^2\log n$.
\end{itemize}
\end{theorem}

It is pointed out in \cite{OSnlogn} and \cite{BrGr} that Theorem \ref{quadconj} and the results of \cite{BrGr} give an alternative prove of solvability of the conjugacy problem in free-by-cyclic groups (originally proved in \cite{BMMV}).

\subsubsection{The conjugacy problem in automatic groups}\label{ss:tcpiag} It is still unknown whe\-ther the conjugacy problem is decidable in every automatic group although it is decidable in every bi-automatic group (by Theorem \ref{ab} since these groups are semihyperbolic).

\begin{prob} Can the quasi-proof above and, specifically, Remark \ref{rk91} be applied to prove that automatic groups have decidable conjugacy problem? Is it possible to apply the quasi-proof to groups from Theorem \ref{th:qdf} in order to obtain a unified proof of the solvability of conjugacy problem for these groups?
\end{prob}

\subsection{A description of isoperimetric functions $\ge n^4$}\label{s:adoifg}

It is most probable that the class of Dehn functions
$\ge n^2\log n$ is as wide as the class of time functions of Turing
machines with the same restriction. The next theorem confirms that conjecture in the case of
Dehn functions $\ge n^4$.

\subsubsection{Superadditivity}\label{ss:s} We call a function $f$ {\em \index{Superadditive function}superadditive} if for all natural
numbers $m, n$ we have $f(m+n)\ge f(m)+f(n)$. The problem of whether all Dehn functions of finitely presented groups are superadditive (up to equivalence) is still open. We conjectured in \cite{SBR} that the answer is ``yes".  In \cite{freeproducts} Guba and myself proved that every free product $G*H$ where $G$ and $H$
are non-trivial finitely presented groups, has a superadditive
Dehn function. For example, for every group $G$ the free product
$G*{\Z}$ has a superadditive Dehn function (here ${\bf Z}$ is
the infinite cyclic group). In fact the Dehn function of $G*\Z$ is the superadditive closure of the Dehn function of $G$, i.e. the smallest superadditive function exceeding the Dehn function of $G$. Thus if there was an example of a finitely presented group $G$ with non-superadditive Dehn function, it would be also an example of a group $G$ whose Dehn function is not equivalent to the Dehn function of $G*\Z$. Note that a presentation of $G*\Z$ is obtained from the presentation of $G$ by adding one new generator and no extra relations. The problem of superadditivity of all Dehn functions is similar in nature to the problem of whether the quasi-proof from \ref{tcp} works for all groups with quadratic Dehn functions. A counterexample is hard to imagine, let alone construct.

\subsubsection{The theorem}\label{ss:tt2} \begin{theorem} \rm{(}See Sapir-Birget-Rips \cite[Theorem 1.2]{SBR}\rm{)}. \label{th1}
1. Let ${\mathcal D}_4$ be the set of all Dehn functions $d(n)\ge n^4$
of finitely presented groups. Let ${\mathcal T}_4$ be the set of time
functions $t(n)\geq n^4$ of arbitrary (non-deterministic) Turing machines. Let ${\mathcal
T}^4$ be the set of superadditive functions which are fourth powers
of time functions. Then $${\mathcal T}^4\subseteq {\mathcal D}_4\subseteq
{\mathcal T}_4.$$

2. For every non-deterministic Turing
machine $\ttt$ with time function $T(n)$ such that $T^4(n)$ is superadditive there exists a finitely
presented group $G$ with Dehn function $T^4(n)$ and the isodiametric
function $T^3(n)$ such that the problem recognized by $\ttt$ reduces polynomially to the word problem in $G$ (see \ref{ropta}).
\end{theorem}

Note that we do not know if

\begin{itemize}
\item Every time function of a non-deterministic Turing machine is equivalent to a superadditive function.
\item The square root of a time function of a non-deterministic Turing machine is equivalent to the time function of a non-deterministic Turing machine.
\end{itemize}

If one or both of these statements were true, it would be (obviously) possible to simplify the formulation of Theorem \ref{th1}.

\subsection{NP-complete and coNP-complete groups}\label{s:npcg}

\subsubsection{NP-complete groups}\label{ss:npcg} Now let us take a non-deterministic Turing machine $M$ recognizing some NP-complete problem, say, the exact bin packing  Problem \ref{bin}. Let $G$ be the group from Theorem \ref{th1} that corresponds to $M$. Then $G$ has polynomial Dehn function, hence its word problem is in NP (by \ref{sub:witness}). On the other hand since the problem recognized by $M$ polynomially reduces to the word problem in $G$, the word problem in $G$ is NP-complete. Hence we get

\begin{cy}\rm{(}See \cite[Corollary 1.1]{SBR}\rm{)}. \label{cyNP1} There exists a finitely presented group with NP-complete word problem.
\end{cy}

\subsubsection{CoNP-complete groups}\label{s:cg} Theorem \ref{th1} only treats the ``yes" part of the word problem. Van Kampen diagrams are not witnesses of the ``no" part (i.e. inequalities $w\ne 1$). Nevertheless we have the following

\begin{theorem}[Birget \cite{Bir1}]\label{Birget} There exists a finitely presented group $G$ whose word problem is coNP-complete.
\end{theorem}

Birget reduces a known coNP-complete problem, namely the {\em \index{Circuit equivalence problem}circuit equivalence problem}\footnote{A \index{Boolean circuit}boolean circuit \cite{Vol} is a ``device" built of elementary circuits connected by wires; each elementary circuit represents a basic operation from logic: AND, OR, NOT and FORK. Every circuit has input wires and output wires. Every time we send signals 0 or 1 through input wires, we get a combination of signals 0 or 1 coming through the output wires. Hence every circuit represents a boolean function.
The circuit equivalence problem asks, given two (acyclic) boolean circuits, do they have the same input-output function?} to the word problem in a concrete finitely presented subgroup of the Thompson-Higman group $V$ \cite{CFP} (denoted by $G_{2,1}$ in Higman \cite{Hig}).

That group, as all Thompson groups, consists of (partial) transformations of the set of infinite binary words. One of the simple key ideas is that if $G$ is a group of transformations of a set $X$, then a witness to the inequality $f\ne g$ for $f,g\in G$ is just an element $x\in X$ such that $f(x)\ne g(x)$. We have mentioned already in \ref{ropta} that the group $V$ has been used before in constructing finitely presented groups with undecidable word problem by McKenzie-Thompson \cite{McT}. They used $V$ to simulate a construction of arbitrary recursive functions from simple building blocks (somewhat similar to boolean circuits).

\subsection{The isoperimetric spectrum}\label{s:tis}

The {\em \index{Isoperimetric spectrum}isoperimetric spectrum} consists of all numbers $\alpha\ge 1$ such that $n^\alpha$ is equivalent to the Dehn function of a finitely presented group. By Theorem \ref{Bow1}, the intersection of the isoperimetric spectrum with the open interval $(1,2)$ is empty. The next theorem gives an almost complete description of all numbers in the isoperimetric spectrum that are $\ge 4$.
\begin{theorem}[Sapir \cite{SBR}]\label{alpha0} The {\em isoperimetric spectrum} contains all
numbers $\alpha\ge 4$ whose $n$-th binary approximation\footnote{That is a rational diadic number $\beta=a/2^n$ which is within $O(2^{-n})$ from $\alpha$} can be computed by a
deterministic Turing machine in time less than $2^{2^n}$.  On the other hand, Theorem \ref{2} implies that if
$\alpha$ is in the isoperimetric spectrum then the $n$-th binary approximation of
$\alpha$ can be computed in time $\le 2^{2^{2^n}}$.
\end{theorem}

The difference in the number of $2$'s in these
expressions, is the difference between $P$ and $NP$ in Computer
Science (if $P=NP$ then there should be two $2$'s in both
expressions).

All ``constructible" numbers (rational numbers, algebraic numbers,
values of relatively easy computable analytic functions at rational points such as $e+2$, $\pi+1$, $10\sin \frac34$, etc.) which are $\ge 4$ satisfy the condition of Theorem \ref{alpha0}. In fact, these and most other familiar numbers can be computed by very fast quasilinear algorithms (see J. Borwein  and P. Borwein \cite{BoBo}).

The intersection of the isoperimetric spectrum with the interval $[2,4]$ is less known, although it is known that it is dense, contains all rational and many transcendental numbers from that interval \cite{Bri, BB, BBFS}. The advantage of the {\em snowflake} construction and its modifications employed in these papers is that one obtains explicit and often quite short presentations of groups with prescribed Dehn functions of the form $n^\alpha$ while the groups from \cite{SBR} are obtained by simulating Turing machines computing binary approximations of $\alpha$ (and so the number of generators and relators are usually quite large and writing explicit presentations is not a feasible task).

\subsection{Groups with polynomial-non-recursive and quadratic-non-re\-cur\-sive Dehn functions}\label{pnrdf}

Let $F_{\hbox{low}}$ and $F_{\hbox{high}}$ be two classes of increasing functions $\N\to \N$.  We say that a function $h\colon\N\to \N$ is an $F_{\hbox{low}}\mathrm{-non-}F_{\hbox{high}}$-function, if for some $f\in F_{\hbox{low}}$, $h(n)\le f(n)$ for  infinitely many $n$,  and for every $g\in F_{\hbox{high}}$, $h(n)>g(n)$
for infinitely many $n$. For example, a function $h$ is polynomial-non-recursive if $h$ is not bounded above by any recursive function, but is smaller than some fixed polynomial $f(n)$ for infinitely many values of $n$. Similarly we can define the class of functions which are quadratic-non-recursive, etc.

\subsubsection{Polynomial-non-recursive}\label{pnr}
\begin{theorem}[Olshanskii, Sapir \cite{OSnq}]\label{th2} There exists a finitely presented group with poly\-nomial-non-recursive Dehn function.
\end{theorem}

\begin{proof} We shall sketch a proof showing how to deduce Theorem \ref{th2} from the results of \cite{SBR}. It shows a general strategy of estimating Dehn functions of groups from \cite{SBR} corresponding to the so called $S$-{\em machines} (see Section \ref{msmarc}). Let $M$ be the Turing machine from Theorem \ref{re}, that is $M$ has undecidable halting problem and infinitely many $h$-good numbers where $h(n)=\exp\exp n$. In \cite{SBR}, it is shown
how for every Turing machine $M$, one can construct a group $G=G(M)$. The fact that $M$ has undecidable halting problem, and \cite[Propositions 4.1, 12.1]{SBR} imply that
$G$ has undecidable word problem. The presentation of $G$ from \cite{SBR} shows that $G$ is obtained from some special multiple HNN-extension of a free group (called an $S$-{\em machine}) with free letters $\theta_1,\ldots, \theta_m$ by imposing one relation which is called the {\em hub} (see \ref{s:twp}).

To estimate the Dehn function of $G$, we need to estimate the areas of van Kampen diagrams over the presentation of $G$ from
\cite{SBR}. Let $\Delta$ be a van Kampen diagram over that presentation with boundary label $w$ of length $l$. We
assume that $\Delta$ is minimal.

By \cite[Proposition 11.1]{SBR} we can do the so called {\em \index{Snowman decomposition}snowman decomposition} of $\Delta$, to obtain
three sequences of diagrams $E_1$,\ldots, $E_s$,
$\Gamma_1$,\ldots, $\Gamma_{s-1}$, $\Pi_1$,\ldots, $\Pi_{s-1}$, with the
following properties:

\begin{enumerate}
\item[(S1)] $E_s = \Delta$.
\item[(S2)] $E_1$ contains no hubs, i.e. $E_1$ is a diagram over $G'$.
\item[(S3)] For every $i=1,\ldots, s-1$, we have that $E_{i+1}$ is a union of subdiagrams
$\Gamma_i, E_i, \Pi_i$ having no common cells, and $\Pi_i$ is either empty or contains exactly one hub cells and no $\theta_j$-edges on its boundary.
\end{enumerate}

This implies that we can glue together the diagrams
$E_1, \Gamma_1,\ldots,\Gamma_{s-1}$, $\Pi_1,\ldots,\Pi_{s-1}$
so as to obtain
a diagram with boundary label $w$. Therefore in order to find
the upper bound  for the area of $\Delta$ we need to estimate
the sum of the areas of
$E_1, \Gamma_1,\ldots,\Gamma_{s-1}, \Pi_1,\ldots,\Pi_{s-1}$.

By \cite[Lemma 11.21]{SBR} for every $i=1,\ldots, s-1$,
$$|\partial(E_i)|<|\partial(E_{i+1})|.$$ Therefore $$|\partial(E_i)|\le |w|, \ i=1,\ldots, s,$$
and $s<|w|$.

The diagram $E_1$ and for every $i$ the diagram $\Gamma_i$
contain no hubs.
So by \cite[Lemma 8.1]{SBR},
$$\area(\Gamma_i)\le
c_1|\partial(\Gamma_i)|^3$$
and $$\area(E_1)\le
c_1|\partial(E_1)|^3$$
for some constant $c_1$.
We also know that $|\partial(E_1)|\le |w|$ and by \cite[Lemma 11.21]{SBR}
$|\partial(\Gamma_i)|\le c_2|\partial(E_{i+1})|\le c_2|w|$,
\ for $i=1,\ldots, s-1$ and some constant $c_2$.
Therefore
\begin{equation*}
\begin{array}{l}
\area(E_1)+\area(\Gamma_1)+ \ldots +\area(\Gamma_{s-1}) \ \le \\
c_1|\partial(E_1)|^3+
c_1|\partial(\Gamma_1)|^3 +
c_1|\partial(\Gamma_2)|^2+ \ldots \\
+c_1|\partial(\Gamma_{s-1})|^3 \ \le\\
c_1c_2s |w|^3\le c_3|w|^4.
\end{array}
\label{eqq1}
\end{equation*}
for some constant $c_3$.

It remains to estimate the area of $\Pi_i$. Let $\Pi$ be one of the $\Pi_i$, $i=1,\ldots, s-1$. Let $W$ be the word written on the boundary of $\Pi$. Note that $|W|\le c_5l$ for some constant $c_5$ because the union of $E_i, \Gamma_i, \Pi_i$ is $E_{i+1}$ by property $(S3)$ and the estimates of the length
$|\partial(E_i)|, |\partial(\Gamma_i)|$ above.

We can assume that $\Pi$ is a minimal diagram. Then by \cite[Propositions 4.1]{SBR} some subword $w$ of the boundary label of $\Pi$ is an accepted configuration of $M$ and the area of $\Pi$ is bounded from above by $c_6 \cost(w)^2$ for some constant $c_6$. Moreover $\cost(w)$ is either bounded by $c_7 |w|$ or does not exceed $c_8\cost(w')$ for some input configuration $w'$ with $|w'|\le |w|$ and some constants $c_7,c_8$. Now suppose $\frac{b}{2c_5}<l< \frac{b}{c_5}$ for some $h$-good number $b\gg 1$. Then $|w'|\le |w|<b$. Hence $\cost(w')< \log b\le l$. Thus the area of $\Pi$ is bounded by $c_9l^2$ for some constant $c_9$. Thus the area of $\Delta$ is bounded by some polynomial in $l$. Since the set of $h$-good numbers of $M$ is infinite by the choice of $M$, the group $G$ has polynomial-non-recursive Dehn function. \end{proof}

\subsubsection{Quadratic-non-recursive}\label{ss:qnr} The following theorem which is much stronger than Theorem \ref{th2} was proved in \cite{OSnq}.

\begin{theorem}[Olshanskii \cite{OSnq}]\label{nq} There exists a finitely presented group with quadratic-non-recursive Dehn function.
\end{theorem}

Comparing with \ref{pnr}, the proof of Theorem \ref{nq} requires significantly new ideas. It is easy to see that the proof from \ref{pnr} does not give quadratic upper bound for infinitely many values of $n$ (for instance, the $S$-machines usually have cubic Dehn functions \cite{SBR}, so one cannot hope to bound the area of even $E_1$ by a quadratic function). The proof of Theorem \ref{nq} involves
\begin{itemize}
\item A new construction of an $S$-machine simulating the Turing machine from Theorem \ref{re}.
\item A new way of decomposing a diagram over the presentation of the $S$-machine plus the hub. Several parameters of diagrams and several cutting and pasting operations on diagrams are introduced. One of the parameters is the area, another one is the perimeter, and there are several others. Every cutting operation reduces certain expression involving the parameters which makes the induction possible.
\end{itemize}

\section{Higman embedding}\label{hem}

\subsection{A characterization of groups with word problem in NP and other time complexity classes}\label{s:acogw}

The fundamental result of Higman \cite{HigEmb} gave an algebraic characterization  of groups with recursively enumerable word problem (i.e. recursively enumerable set of words that are equal to 1 in the group).

\begin{theorem}[Higman \cite{HigEmb}]\label{Hig} The word problem in a finitely generated group $G$ is recursively enumerable if and only if $G$ is a subgroup of a finitely presented group.
\end{theorem}

There are several different proofs of Theorem \ref{Hig}. See, for example,  Rotman's book \cite{Ro} and  Manin's book \cite{Manin} (where you can also find an interesting discussion of the philosophical significance of that theorem). See also \ref{te1} below.

After \cite{HigEmb}, there were several results showing that embedding into a finitely presented group can preserve or even improve the algorithmic properties of the group. First results have been obtained by Cla\-pham \cite{Cla} and Valiev \cite{Va} (see \cite{OScol} for the history of these results): they
proved that the solvability (even r.e. degree) of the word problem
and the level in the polynomial hierarchy of the word problem is
preserved under some versions of Higman embedding.

By Theorem \ref{MO} the Dehn function of a finitely presented group is recursive if and only if the group has decidable word problem. Moreover by Proposition \ref{Dehn1} for every finitely presented  group $G$ with Dehn function $T(n)$ there exists a nondeterministic Turing machine $M(G)$ which solves the word problem in $G$ and has time function equivalent to $T(n)$. Roughly speaking, this Turing machine takes a word $w$ over the generators of $G$ and just inserts relators of $G$. It stops and accepts $w$ when it gets 1.
Clearly this machine solves the word problem in every finitely generated subgroup of $G$ as well. Therefore if a finitely generated group $G$ is a subgroup of a finitely presented group with polynomial isoperimetric function then the word problem in $G$ is in NP (i.e. it can be solved by a non-deterministic Turing machine with polynomial time function).

The drawback is that the word problem in a group $G$ can be easy to solve but the Dehn  function of $G$ can be huge. A typical example is the Baumslag-Solitar group $\mathrm{BS}(1,2)=\langle a,b\ | \ bab\iv=a^2\rangle$. This group has exponential Dehn function (see Example \ref{example2}). But we have mentioned in \ref{emp} that this group is a subgroup of $\GL(2, {\bf Q})$ so the word problem in $\mathrm{BS}(1,2)$ can be solved in at most quadratic time: it is easy to see that the word problem of every finitely generated  group of matrices over
a field of rational numbers can be solved in at most quadratic time by a deterministic Turing machine. In fact it is possible to solve the word problem there in time $n(\log n)^2(\log\log n)$ (this follows from the fact that the product of two $n$-digit numbers can be computed in time $n\log n\log\log n$ \cite{Knuth}).

Nevertheless, the following theorem shows that every finitely generated group $G$  with word problem in NP can be embedded into a finitely presented group $H$ with polynomial isoperimetric function. Thus if we can solve the word problem in $G$ using a very clever and fast Turing machine, then we can use the simple-minded but almost as fast  Turing machine $M(H)$ to solve the word problem in $G$.

\begin{theorem}[Birget, Olshanskii, Rips, Sapir \cite{BORS}] Let $G$ be a finitely generated group with word problem solvable by a non-deterministic Turing machine with time function $\le T(n)$ such that $T(n)^4$ is superadditive (that is $T(m+n)^4\ge T(m)^4+T(n)^4$ for every $m,n$). Then $G$ can be embedded into a finitely presented group $H$ with isoperimetric function equivalent to $n^2T(n^2)^4$. In particular, the word problem of a finitely generated group is in NP if and only if  this group is a subgroup of a finitely presented group with polynomial isoperimetric function.
\label{bors}
\end{theorem}

For matrix groups over rationals, Theorem \ref{bors} implies that every such group is embedded into a finitely presented group with Dehn function at most $n^{10+\varepsilon}$ for every $\varepsilon>0$. Young's result from \cite{Young} (see \ref{emp}) implies that for matrix groups over $\Z$ one can reduce the exponent $10+\varepsilon$ to $2$.

\begin{prob} Is it true that every finitely generated matrix group over a field (in particular, the field of rational numbers) embeds into a finitely presented group with quadratic Dehn function?
\end{prob}

Theorem \ref{bors} immediately implies

\begin{cy}\label{cyNP} The word problem of a finitely generated group $G$ is in NP if and only if $G$ embeds into a finitely presented group with polynomial isoperimetric function.
\end{cy}

By Fagin's Theorem \ref{ct:fagin} this corollary gives an algebraic characterization of groups where the word problem admits an algebraic description.

Using Corollary \ref{cyNP} (and the proof of Theorem \ref{bors}), in order to embed a finitely generated group with word problem in NP into a finitely presented group with polynomial isoperimetric function, one needs first construct a Turing machine which solves the word problem, then convert it into an $S$-machine, then convert the $S$-machine into a group.  As a result the group we construct will have a relatively complicated set of relations. In some particular cases like the free Burnside groups $B_{m,n}$, the free solvable groups, etc., it is possible to modify this construction and get simple presentations of groups with polynomial isoperimetric functions where these groups embed \cite{OSfree}. For the Baumslag-Solitar group $\mathrm{BS}(1,2)$ it was done in \cite{OSfree} too, and an even  much simpler presentation was found by Cornulier and Tessera \cite{CT} (see also Theorem \ref{th:qdf}, part (2) ).

Theorem \ref{bors} is also interesting from the logic point of view. One can consider a group as a logical system where the defining relations
are axioms, and the inference rules are constructing step by step van
Kampen diagrams. The van Kampen diagrams are then {\em proofs} of their boundary
labels. Then the Dehn function becomes the syntactic complexity of proofs. The computational
complexity of the word problem is the complexity in
the metaworld of the group theory. The embedding of groups becomes a conservative
extention of theories.

With this vocabulary, the result means that the complexity of proofs in
the outer world of groups (that is, in its metamathematical semantics)
after appropriate conservative extension to a finitely axiomatized theory becomes the complexity of proofs
in the syntactical sense (up to a polynomial correction).

In this formulation, one can ask a similar question for any logical system. Several results of this type for general logical systems can be found in \cite{Kraj}.

By Theorem \ref{Gro} the class of groups with polynomial isoperimetric functions is a subclass of the class of groups with simply connected asymptotic cones. Hence we formulate

\begin{prob} \label{prob67} Does every finitely generated group with word problem in NP embed into a group all of whose asymptotic cones are simply connected? Equivalently (by Theorem \ref{bors}): does every finitely presented group with polynomial isoperimetric function embed into a group all of whose asymptotic cones are simply connected?
\end{prob}

Note that by Theorem \ref{Gro}, groups with simply connected asymptotic cones have linear isodiametric functions. On the other hand, for every $k\ge 4$ there are groups with Dehn function $n^k$ and isodiametric function $n^{{\frac 34}k}$ by \cite{SBR}. Thus if the embedding as in Problem \ref{prob67} exists, it must distort lengths at least polynomially. The embeddings in \cite{BORS} distort lengths linearly (see \cite{BORS}), so a quite different construction should be used.

\subsection{The space functions of Turing machines and the filling functions of groups}\label{s:tsfot}

The results about space functions of Turing machines and the FFFL functions of groups \cite{OlFFFL} are similar to the results about time functions and isoperimetric functions obtained in \cite{SBR,BORS}. On the one hand the situation is simpler because it  is easier to control the space function than the time function of a machine. Also the issues with superadditivity does not appear in this situation. Indeed, if we want to check if two words are equal to 1 using a Turing machine, then the time spent will be the sum of times needed to checking each word (hence we need superadditivity of the Dehn function) but the storage space is approximately the maximal of the spaces needed for each word because after we proved that one of these words is 1, we can clean the tape and start proving that the other word is equal to 1 (so superadditivity is not needed). On the other hand, the descriptions are more finalized which requires more detailed consideration of geometry of van Kampen diagrams. If in \cite{SBR, BORS}, for example, we just cut the diagram into pieces and estimate the number and the areas  of the pieces, here one needs to carefully examine the lengths of the cuts. Also working with time and Dehn functions, we could modify Turing and $S$-machines by adding artificial tapes which can potentially by very long, but serve as a ``temporary storage space". Dealing with space and FFFL functions, we cannot do that which adds to the complexity of the situation.

Nevertheless the following complete description of FFFL functions of groups was obtained by Olshanskii.

\begin{theorem}[Olshanskii \cite{OlFFFL}] \label{realiz} Every space function $f(n)$ of a Turing machine is equivalent (as in \ref{ss:ttasf}) to the FFFL function of some finitely presented group. \end{theorem}

Since space functions of Turing machines  can be ``arbitrarily complicated" (similar to time functions), we obtain finitely presented groups with ``arbitrary complicated" FFFL functions.

Also similar to the isoperimetric spectrum one can define the {\em \index{Space spectrum}space spectrum} of finitely presented groups as the set of all $\alpha$ such that  $n^{\alpha}$ is equivalent to the FFFL function of a finitely presented group. Since the problem with non-determinism vs determinism is not an issue, we have a complete description of the space spectrum.
\begin{theorem}[Olshanskii \cite{OlFFFL}, compare with Theorem \ref{alpha0}] \label{alpha} For a real number $\alpha\ge 1,$ the function $n^{\alpha}$  is equivalent to the FFFL function of a finitely presented group if and only if the $n$th binary approximation of $\alpha$ is  computable by a deterministic Turing machine using space $\le 2^{2^n}.$
\end{theorem}

An analog of Theorem \ref{bors} also was obtained in \cite{OlFFFL}. Note that by Proposition \ref{propos}, the space complexity of the word problem in a finitely presented group $G$ does not exceed the FFFL function of $G.$ It follows from \cite{MO}, \cite{CMO} that the converse statement fails. An easier example is given by the Baumslag 1-relator \index{Group!Baumslag}group from \cite{B} $G=\langle a,b\mid (aba^{-1})b(aba^{-1})^{-1}=b^2\rangle$. Its FFFL function is not bounded from above by any multi-exponential function (see Gersten \cite{G2} and Platonov \cite{Pl}) while the space complexity of the word problem for $G$ is polynomial (see A. G. Miasnikov, A. Ushakov, and Dong Wook Won \cite{MUW}).

\begin{theorem}[Olshanskii \cite{OlFFFL}] \label{main1} Let $G$ be a finitely generated group such that the word problem for $G$ is decidable by a deterministic Turing machine with space function $f(n)$. Then $G$ embeds into a finitely presented group with FFFL function equivalent to $f(n).$ \end{theorem}

Similarly to Corollary \ref{cyNP} one immediately deduces

\begin{cy}[Olshanskii \cite{OlFFFL}] \label{pspace} The word problem of a finitely generated group  $G$ belongs to $PSPACE$ if and only if $G$ embeds into a  finitely presented group with polynomial FFFL function. \end{cy}

\subsection{Collins' problem}\label{s:cp}

The conjugacy problem turned out to be much harder to preserve under
embeddings. We have already mentioned the results of Collins and Miller \cite{CM} and Gorjaga and
Kirkinski\u\i  \cite{GK}:  even subgroups of index 2 of
finitely presented groups do not inherit solvability or
unsolvability of the conjugacy problem (see \ref{ss:tcpaq}).

In 1976 D. Collins \cite{KT} posed the following question (problem
5.22): {\em Does there exist a version of the Higman embedding
theorem in which solvability of the conjugacy
problem is preserved?} The solution is given by the following two theorems.

\begin{theorem}[Olshanskii, Sapir \cite{OScol}]\label{th7}
A finitely generated group $H$ has solvable conjugacy problem if and
only if it is Frattini embedded into a finitely presented group $G$
with solvable conjugacy problem.
\end{theorem}

\begin{theorem}[Olshanskii, Sapir \cite{OScol1}]\label{th8} Every countable recursively presented
group with solvable word and power problems is embeddable into a
finitely presented group with solvable conjugacy and power problem.
\end{theorem}

Recall that a subgroup $H$ of a group $G$ is {\em \index{Frattini embedded subgroup}Frattini embedded} in
$G$ if every two elements of $H$ that are conjugate in $G$ are also
conjugate inside $H$. It is clear that if $H$ is Frattini embedded into $G$ and $G$ has solvable conjugacy problem, then $H$ also has solvable conjugacy problem. We say that $G$ has solvable {\em \index{Power problem}power
problem} if there exists an algorithm which, given $u, v$ in $G$
says if $v=u^n$ for some $n\ne 0$.

Theorem \ref{th8} is a relatively easy application of Theorem \ref{th7}.

In the same Problem 5.22 of \cite{KT} Collins asked whether a version of Higman embedding can preserve the recursively enumerable hierarchy (the recursive sets are at the bottom of that hierarchy). The affirmative answer is also given in \cite{OScol}.

The construction in \cite{OScol} is much more complicated than in
\cite{BORS} or \cite{OlFFFL}. First we embed $H$ into a finitely presented group
$H_1$ preserving the solvability of the word problem (say, one can use \cite{BORS}). Then we use
the Miller $S$-machine $M(H_1)$ (see \ref{ss:tmm}) to solve the word problem in $H$. In
order to overcome technical difficulties, we needed certain parts of
words appearing the computation to be always positive\footnote{A word is called {\em \index{Word!positive}positive} if it does not contain inverses of the alphabet letters.}. The standard
positivity checkers do not work because they are $S$-machines as
well, and can insert negative letters. So we used some ideas from
the original Boone-Novikov proofs. That required introducing new
generators, $x$-letters (in addition to the $a$-, $q$-, and
$\theta$-letters in $S$-machines) and Baumslag-Solitar relations (as in Example \ref{example2}). In
addition, to analyze the conjugacy problem in $G$, we had to
consider annular diagrams which are more complicated than \vk disc
diagrams. Different types of annular diagrams (spirals, roles, etc.)
required different treatment.

We do not have any reduction of the complexity of the conjugacy
problem in $H$ to the complexity of the conjugacy problem in $G$. In
particular, solving the conjugacy problem in $G$, in some cases
required solving systems of equations in free groups (i.e. the
Makanin-Razborov algorithm, see \ref{ss:oeifg}).

\subsection{A finitely presented non-amenable group without free subgroups}\label{s:afpnagwfs}

One of the most important applications of our versions of Higman
embeddings so far was the construction of a finitely presented
counterexample to the von Neumann problem, i.e. a finitely presented
non-amenable group without non-Abelian free subgroups \cite{OSamenab}.

\subsubsection{Short history of the problem}\label{ss:shotp} Suppose that a group $G$ acts on a metric space $X$ by isometries. We say that the action is paradoxical if $X$ can be decomposed as a disjoint union $A_1\sqcup \ldots \sqcup A_m\sqcup B_1\sqcup\ldots \sqcup B_n$ ($m,n>1$) such that for some $g_1,\ldots, g_m, h_1,\ldots, h_n$ of $G$ we have $X=g_1A_1\cup \ldots \cup g_mA_m=h_1B_1\cup\ldots \cup h_nB_n$. Hausdorff \cite{Haus}
proved that the action of $SO(3)$ on the unit sphere minus a countable set of points is paradoxical:
one can subdivide the 2-sphere minus a countable
set of points into 3 parts $A$, $B$, $C$, such that the union of any two of these parts can be obtained by
rotating the third part (hence we can cut a sphere into a finite number of pieces, rotate these pieces and obtain two spheres of the same radius). Banach and Tarski \cite{BT} generalized
Hausdorff's result and proved what is now known as Banach-Tarski paradox.
Von Neumann
\cite{vN} noticed that the cause of the paradoxes is the structure of the group $G$. In particular the cause of the Hausdorff paradox is that $\SO_3(\R)$ contains a non-Abelian free subgroup. He called the groups that cannot act paradoxically {\em \index{Group!amenable}amenable}. Von
Neumann showed that the class of amenable groups contains Abelian
groups, finite groups and is closed under taking subgroups,
extensions, and infinite unions of increasing sequences of groups.
Day \cite{day} and Specht \cite{Specht} showed that this class is
closed under homomorphic images. The class of groups without
non-Abelian free subgroups  is also closed under these operations and
contains Abelian and finite groups.

The problem of existence of non-amenable groups without non-Abelian
free subgroups probably goes back to von Neumann and became known as
the ``von Neumann problem" in the fifties.

First counterexamples to the von Neumann problem were constructed by
Olshanskii \cite{OlAmen}. He proved that the Tarski monsters, both
torsion-free and torsion (see \cite{Olbook}), are not amenable. Later
Adian \cite{Adian} showed that the non-cyclic \index{Free Burnside group}free Burnside group of
odd exponent $n\ge665$ with at least two generators, that is the
group given by the presentation $$\langle a_1,\ldots, a_m\mid u^n=1,
\hbox{ for all group words $u$ in the alphabet }
\{a_1,\ldots, a_m\} \rangle,$$ is not amenable. Recently L\"uck and Osin \cite{LO} constructed the first example of residually finite non-amenable torsion group.

All these examples are not finitely presented. For the Tarski monsters and Burnside groups this is because these groups are not hyperbolic but are inductive limits of hyperbolic groups (see \ref{ss:qohg}), a similar argument applies to the groups of L\"uck and Osin.
The question about existence of finitely presented counterexample to von Neumann's problem was explicitly formulated by Grigorchuk
in \cite{KT} and by J. Cohen in \cite{Cohen}. There exists one
finitely presented group without 
non-Abelian free subgroups for which the problem of amenability is
non-trivial: it is the  R.Thompson's group $F$ (for the definition of $F$
look in Section \ref{td}). The question of whether $F$ is not
amenable was formulated by R. Thompson in the beginning of the 70s (unpublished), and, in print, by R. Geoghegan in 1979. A considerable
amount of work has been done to answer this question but it is still
open.

\subsubsection{The result}\label{ss:tr5}

Together with A. Olshanskii, we proved the following theorem.

\begin{theorem}[Olshanskii, Sapir \cite{OSamenab}] \label{main} For every sufficiently
large odd $n$, there exists a finitely presented group \label{g1}
${\mathcal G}$ which satisfies the following conditions.
\begin{enumerate}
\item\label{22} ${\mathcal G}$ is an HNN-extension of a finitely generated
infinite group of exponent $n$.
\item\label{33} ${\mathcal G}$ is an extension of a group of
exponent $n$ by an infinite cyclic group.
\item\label{44} ${\mathcal G}$ contains a subgroup isomorphic to the free
Burnside group of exponent $n$ with $2$ generators.
\item\label{55} ${\mathcal G}$ is a non-amenable finitely presented group
without free non-cyclic subgroups.
\end{enumerate}
\end{theorem}

By a theorem of Adian \cite{Adian},
part (\ref{44}) implies that ${\mathcal G}$ is not amenable. Thus parts
(\ref{22}) and (\ref{44}) imply part (\ref{55}).

\subsubsection{The proof}\label{ss:tpr}

Let us present the main ideas of our construction. We first embed
the free Burnside group \label{bmn}$B(m,n)=\langle {\mathcal B}\rangle$
of odd exponent $n\gg 1$ with $m>1$ generators $\{b_1,\ldots, b_m\}={\mathcal
B}$ into a finitely presented group ${\mathcal G}'=\langle {\mathcal C}\mid
{\mathcal R}\rangle$ where ${\mathcal B}\subset {\mathcal C}$. This is done as
in \cite{BORS, OSsur} using an explicitly constructed $S$-machine. Then we take a copy \label{a11}${\mathcal
A}=\{a_1,\ldots, a_m\}$ of the set ${\mathcal B}$, and a new generator ${
t}$, and consider the group given by generators ${\mathcal C}\cup {\mathcal
A}\cup \{t\}$ and the following three sets of relations.

\begin{enumerate}
\item[(1)] \label{rrel} the set ${\mathcal R}$ of the relations of
the finitely presented group ${\mathcal G}'$;
\item[(2)] \label{urel}($u$-relations) $y=u_y$, where $u_y, y\in {\mathcal C},$
is a certain word in ${\mathcal A}$ satisfying the small cancelation condition $C'(\lambda)$ for a very small $\lambda$;
\item[(3)] \label{rhorel} (HNN-relations) $t\iv a_it =b_i,
i=1,\ldots, m$; these relations make $\la {\mathcal A}\ra$ a conjugate of
its subgroup of exponent $n$ (of course, the group $\la {\mathcal A}\ra$
gets factorized).
\end{enumerate}
The resulting group ${\mathcal G}$ is obviously generated by the set
${\mathcal A}\cup \{{t}\}$ and is an HNN-extension of its
subgroup $\langle {\mathcal A}\rangle$ with the stable letter ${t}$.
Every element in $\la {\mathcal A}\ra$ is a conjugate of an element of
$\langle {\mathcal B}\rangle$, so $\la {\mathcal A}\ra$ is an $m$-generated
group of exponent $n$. This immediately implies that ${\mathcal G}$  is
an extension of a group of exponent $n$ (the union of increasing
sequence of subgroups ${t}^s\la {\mathcal A}\ra {t}^{-s},
s=1,2,\ldots )$ by a cyclic group. It turns down that $\la {\mathcal A}\ra$ contains a copy of
the free Burnside group $B(2,n)$. Thus the group is non-amenable by Adian \cite{Adian} (see \cite{OSamenab} and \cite{SaICM} for details).

\section{Methods: $S$-machines and related constructions}\label{msmarc}

Several theorems discussed in this survey have been proved using $S$-machines, a natural blend of Turing machines and groups (see Theorems \ref{quadconj1}, \ref{th1}, \ref{alpha0}, \ref{th2}, \ref{nq}, \ref{bors}, \ref{realiz},  \ref{main1}, \ref{th7}, \ref{main}). Here we shall present a short introductions to $S$-machines.

\subsection{The definition}\label{s:td} There are several ways to look at $S$-machines introduced in \cite{SBR}. One can view them as a version of Turing machines (see \ref{tm}), as rewriting systems (as in \cite{SBR}), as groups that are certain multiple HNN-extensions of free groups (see \cite{OSnlogn}) or as semigroups of partial 1:1-transformations of the set of admissible words as in \cite{OSamenab}. For different applications, one needs different points of view.
Probably the easiest way to introduce $S$-machines is by defining them as multiple HNN-extension of a free
group.

\subsubsection{The Miller machine} \label{ss:tmm} Let us start with an
example that we call the {\em \index{Miller machine}Miller
machine}. This important example is due to C. Miller \cite{Mi1}. Let $G=\la
X\mid R\ra$ be a finitely presented group. The Miller machine is the
group $M(G)$ generated by $X\cup \{q\}\cup \{\theta_x\mid x\in
X\}\cup\{\theta_r\mid r\in R\}$ subject to the following relations
$$\theta x=x\theta,\quad \theta_x xq=qx\theta_x,\quad \theta_r q=qr\theta_r$$
where $\theta$ is any letter in $\Theta=\{\theta_x\mid x\in
X\}\cup\{\theta_r\mid r\in R\}$. Clearly, this is an HNN-extension
of the free group $\la X, q\ra$ with free letters $\theta\in
\Theta$. The main feature of $M(G)$ discovered by Miller is that
{\em $M(G)$ has undecidable conjugacy problem provided $G$ has
undecidable word problem.} In fact it is easy to see that $qw$ is
conjugated to $q$ in $M(G)$ if and only if $w=1$ in $G$.

To see that $M(G)$ can be viewed as a version of a Turing machine, consider any word
$uqv$ where $u, v$ are words in $X\cup X\iv$. If we conjugate $uqv$
by $\theta_r$, we get the word $uqrv$ because $\theta_r
q=qr\theta_r$ and $\theta_r$ commutes with $u$ and $v$ (here and
below we do not distinguish words that are freely equal). Hence
conjugation by $\theta_r$ amounts to executing a command $[q\to
qr]$. Similarly, conjugation by $\theta_x$ amounts to executing a
command $[q\to x\iv qx]$. If $u$ ends with $x$, then executing this
command means moving $q$ one letter to the left. Thus conjugating
words of the form $uqv$ by $\theta$'s and their inverses, we can
move the ``head" $q$ to the left and to the right, and insert
relations from $R$.

The work of the Miller machine $M(G)$ can be drawn in the form of a
van Kampen diagram (see Figure \ref{fig1}) that we call a {\em \index{Trapezium}trapezium}. It consists of horizontal $\theta$-bands. The bottom side of
the boundary of the trapezium is labeled by the first word in the
computation ($uqv$), the top side is labeled by the last word in the
computation ($q$), the left and the right sides are labeled by the
{\em history of computation}, the sequence of $\theta$'s and their
inverses corresponding to the commands used in the {\em computation}
$uqv\to\ldots \to q$. The words written on the top and bottom sizes of
the $\theta$-bands are the intermediate words in the computation. We
shall always assume that they are freely reduced.

\begin{center}
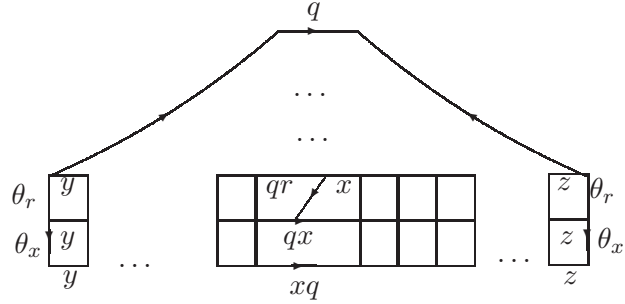
\begin{figure}[H]
\unitlength .7mm 
\linethickness{0.4pt}
\ifx\plotpoint\undefined\newsavebox{\plotpoint}\fi 
\begin{picture}(122.75,59)(0,0)
\put(55.25,10.75){\framebox(19.75,8.5)[cc]{}}
\put(55.25,19.25){\framebox(19.75,8.5)[cc]{}}
\put(75,10.75){\framebox(7.25,8.5)[cc]{}}
\put(75,19.25){\framebox(7.25,8.5)[cc]{}}
\put(82.25,10.75){\framebox(7.25,8.5)[cc]{}}
\put(82.25,19.25){\framebox(7.25,8.5)[cc]{}}
\put(89.5,10.75){\framebox(7.25,8.5)[cc]{}}
\put(89.5,19.25){\framebox(7.25,8.5)[cc]{}}
\put(111,11){\framebox(7.25,8.5)[cc]{}}
\put(111,19.5){\framebox(7.25,8.5)[cc]{}}
\put(48,10.75){\framebox(7.25,8.5)[cc]{}}
\put(48,19.25){\framebox(7.25,8.5)[cc]{}}
\put(16,10.75){\framebox(7.25,8.5)[cc]{}}
\put(16,19.25){\framebox(7.25,8.5)[cc]{}}
\put(32.25,10.75){\makebox(0,0)[cc]{$\ldots$}}
\put(104.25,11.75){\makebox(0,0)[cc]{$\ldots$}}
\put(64.5,6){\makebox(0,0)[cc]{$xq$}}
\put(63.25,15.5){\makebox(0,0)[cc]{$qx$}}
\put(20,8){\makebox(0,0)[cc]{$y$}}
\put(19.5,15.5){\makebox(0,0)[cc]{$y$}}
\put(114.25,16){\makebox(0,0)[cc]{$z$}}
\put(115,8.5){\makebox(0,0)[cc]{$z$}}
\put(12,14.75){\makebox(0,0)[cc]{$\theta_x$}}
\put(122.75,15.5){\makebox(0,0)[cc]{$\theta_x$}}
\put(64.38,10.75){\vector(1,0){.07}}\put(62,10.75){\line(1,0){4.75}}
\put(64.63,19.25){\vector(1,0){.07}}\put(62.75,19.25){\line(1,0){3.75}}
\put(16,15.25){\vector(0,-1){.07}}\put(16,17.5){\line(0,-1){4.5}}
\put(118.13,15.63){\vector(0,-1){.07}}\multiput(118.25,17.75)(-.03125,-.53125){8}{\line(0,-1){.53125}}
\put(65.5,23.5){\vector(-2,-3){.07}}\multiput(68.5,27.75)(-.03370787,-.04775281){178}{\line(0,-1){.04775281}}
\put(11.25,24){\makebox(0,0)[cc]{$\theta_r$}}
\put(121,25){\makebox(0,0)[cc]{$\theta_r$}}
\put(19.5,25.75){\makebox(0,0)[cc]{$y$}}
\put(72,25.75){\makebox(0,0)[cc]{$x$}}
\put(113.75,26.25){\makebox(0,0)[cc]{$z$}}
\put(59.75,25.25){\makebox(0,0)[cc]{$qr$}}
\put(66,34.75){\makebox(0,0)[cc]{$\ldots$}}
\put(65.5,43.25){\makebox(0,0)[cc]{$\ldots$}}
\put(67,55.25){\vector(1,0){.07}}\put(59.5,55.25){\line(1,0){15}}
\put(66.25,59){\makebox(0,0)[cc]{$q$}}
\put(38.75,39.75){\vector(3,2){.07}}\qbezier(16,27.75)(39.75,38)(59.5,55.25)
\put(95.25,39.75){\vector(-3,2){.07}}\qbezier(118,27.75)(94.25,38)(74.5,55.25)
\end{picture}
\caption{Trapezium of the Miller machine for a deduction
$uqv\to\ldots \to q$. Here $u=y\ldots x$, $v=\ldots z$.} \label{fig1}
\end{figure}
\end{center}

Viewed as  a Turing machine, the Miller machine has one tape and one state letter.
\subsubsection{The formal definition}  Like Turing machine (see \ref{tm}) general
{\em \index{$S$-machine}$S$-machines} can have many tapes and many state letters. Here is a
formal definition.

Let $F(Q,Y)$ be the free group generated by two sets of letters
$Q=\sqcup_{i=1}^N Q_i$ and $Y=\sqcup_{i=1}^{N-1} Y_i$ where $Q_i$ are
disjoint and non-empty (below we always assume that $Q_{N+1}=Q_1$,
and $Y_N=Y_0=\emptyset$).


The set $Q$ is called the set of $q$-letters, the set $Y$ is called
the set of $a$-letters.

In order to define an HNN-extension, we consider also a collection
$\Theta$ of $N$-tuples of $\theta$-letters. Elements of $\Theta$ are
called {\em \index{$S$-machine!rule of}rules}. The components of $\theta$ are called {\em
brothers} $\theta_1,\ldots, \theta_N$. We always assume that all
brothers are different. We set $\theta_{N+1}=\theta_1$.

With every $\theta\in \Theta$, we associate two sequences of
elements in $F(Q\cup Y)$: $B(\theta)=[U_1,\ldots, U_N]$,
$T(\theta)=[V_1,\ldots, V_N]$, and a subset $Y(\theta)=\cup Y_i(\theta)$
of $Y$, where $Y_i(\theta)\subseteq Y_i$.

The words $U_i, V_i$ satisfy the following restriction:

\begin{itemize}
\item[(*)] For every $i=1,\ldots, N$, the words $U_i$ and $V_i$ have the form
$$U_i=v_{i-1}k_iu_i, \quad V_i=v_{i-1}'k_{i}'u_{i}'$$ where $k_{i}, k_{i}'\in
Q_{i}$, $u_{i}$ and $u_{i}'$ are words in the alphabet $Y_{i}^{\pm
1}$, $v_{i-1}$ and $v_{i-1}'$ are words in the alphabet
$Y_{i-1}^{\pm 1}$.
\end{itemize}

Now we are ready to define an $S$-machine $\sss$ by generators and
relations. The generating set $X$ of the $S$-machine $\sss$ consists
of all $q$-, $a$- and $\theta$-letters. The relations are:

$$U_i\theta_{i+1}=\theta_i V_i,\,\,\,\, i=1,\ldots, s, \qquad \theta_j a=a\theta_j$$
for all $a\in Y_j(\theta)$.

Every $S$-rule $\theta=[U_1\to V_1,\ldots, U_s\to V_s]$ has an inverse
$\theta\iv=[V_1\to U_1,\ldots, V_s\to U_s]$; we set
$Y_i(\theta\iv)=Y_i(\theta)$.

\begin{rk} {\rm Every $S$-machine is indeed an HNN-extension of
the free group $F(Y,Q)$ with finitely generated associated
subgroups. The free letters are $\theta_1$ for every $\theta\in
\Theta$. We leave it as an exercise to find the associated
subgroups.}
\end{rk}

\subsubsection{Turing machines as $S$-machines}\label{ss:tmasm} Every Turing machine $T$
(see the definition in \ref{tm}) can be considered as an $S$-machine $S'(T)$
in the natural way: the generators of the free group are all tape
letters and all state letters. The commands of the Turing machine
are interpreted as rules of the $S$-machine. The main problem in
that conversion is the following: there is a much bigger freedom in
applying $S$-rules than in executing the corresponding commands of
the Turing machine. Indeed, a Turing machine is in general not
{\em \index{Turing machine!symmetric}symmetric} (i.e. if $[U\to V]$ is a command of the Turing
machine then $[V\to U]$ is usually not) while every $S$-machine is
symmetric. Another - more important - difference is that Turing machines work only with
positive words, and $S$-machines work with arbitrary group words. This, for example, means that the rules $[aqb\to cqd]$ and $[q\to a\iv cqdb\iv]$ are equivalent (the corresponding group presentations are equivalent). Hence an $S$-machine is blind: it decides which rule to apply based solely on the state letters of the configuration, it does not ``see" the content of the tape  (unlike the Turing machine which takes into account the tape letters observed by the head). Hence the language accepted by $S'(T)$ is usually much bigger than
the language accepted by $T$.

Nevertheless, it can be proved that if $T$ is symmetric, and a
computation $w_1\to w_2\to\ldots $ of the $S$-machine $S'(T)$ involves
only positive words, then that is a computation of $T$.

This leads to the following idea of converting any Turing machine
$T$ to an $S$-machine $S(T)$. First we construct a symmetric
Turing machine $T'$ that is equivalent to $T$ (recognizes the same
language). That is a fairly standard Computer Science trick (see
\cite{SBR}).

The second step is to compose the $S$-machine $S'(T')$ with a
machine that checks positivity of a word. That machine starts
working after every step of $S'(T')$. That is if an application of a
rule of $S'(T')$ gives a non-positive (reduced) word, then the
checking machine does not allow the machine $S'(T')$ to proceed to
the next step.

There are several checking machines. One of them - the {\em \index{Adding machine}adding
machine} - is very simple but its time function is exponential (see
\cite{OSnlogn}). Another one is very complicated but it has a
quadratic time function (it was first constructed in \cite{SBR} and used in \cite{SBR,BORS}).

\subsection{The conjugacy problem}\label{s:tcp}

If a Turing machine $M$ has undecidable halting problem, then the $S$-machine $S(M')$ (where $M'$ is a symmetric Turing machine that is equivalent to $M$) has undecidable conjugacy problem \cite{SBR}. It is easy to see that the Dehn function of any $S$-machine is at most cubic. The adding machine used for checking positivity of admissible words exponentially slows down the machine. A nice consequence  is that the Dehn function of  the resulting $S$-machine is $n^2\log n$ \cite{OSnlogn} which is the minimal possible for multiple HNN-extensions of free groups having undecidable conjugacy problem by \ref{ss:tsbn}. Note that although the upper bound of $n^2\log n$ is easy to anticipate, a proof of it involves very non-trivial combinatorics of van Kampen diagrams (some Vassiliev-type invariants of chord diagrams are used, see \cite{OSnlogn}).

\subsection{The word problem} \label{s:twp} To obtain a group with undecidable word problem and, more generally, with a prescribed  Dehn function, we do the following. Take an $S$-machine $\sss$, with many tapes. In order to make the number of tapes large enough, just glue the copies of the same $S$-machine side-by-side so that the trapezia of the new $S$-machine are glued from copies of trapezia of the old $S$-machine as in \ref{ss:rp2}. Let $h$ be the accept word of the $S$-machine (all tapes are empty, all state letters are in the accept state). Let the presentation of the group $G(\sss)$ be obtained from the presentation of $\sss$ by adding one {\em \index{Hub}hub} relation $h=1$. Now if $W$ is an admissible word accepted by $\sss$, we can take the corresponding trapezium $\Delta$, identify the left and right sides of it (labeled by the history of computation) to obtain an annulus with $W$ on the outside boundary component and $h$ on the inside boundary component, then glue the cell corresponding to the hub relation to the inside boundary component to obtain a disc $\Delta'$. This van Kampen diagram shows that if $W$ is accepted by $\sss$, it is equal to 1 in $G(\sss)$. The fact that if $W=1$ in $G(\sss)$ then $W$ is accepted by $\sss$ is proved by some small cancelation argument using the fact that the number of tapes is large. Thus, for example, if $\sss$ has undecidable halting problem, then $G(\sss)$ has undecidable word problem. Proving that the Dehn function of $G$ polynomial in terms of the time function of $\sss$ is much harder \cite{SBR} and requires the snowman decomposition discussed in \ref{pnr}.

\subsection{The embedding}\label{te1} Let now $P=\la X\mid R\ra$ be a finitely generated group with recursively enumerable set of relators $R$. Take an $S$-machine $\sss$ that recognizes the language $R$. We can assume that $\sss$ has sufficiently many tapes, and in tape 1 (between the state letters $k_1,k_2$) we have the input word in the input configuration. Consider now another $S$-machine $\sss'$ which is a copy of $\sss$ (with copies of the state letters) but it does not do anything in the input sector, so during the work of $\sss'$ the input sector is always empty. We can also assume that the state letters of the input configurations in $\sss, \sss'$ do not appear in the middle of (reduced) computations, the input state letters of $\sss$ and $\sss'$ coincide, all other state letters are different.  Now let $H$ be the amalgamated product of $G(\sss)$ and $G(\sss')$ with the amalgamated subgroup generated by the input state letters and all tape letters. Let $W(u)$ be the input configuration of $\sss$ corresponding to the input word $u$. It has subword $u$ between state letters $k_1$ and $k_2$. Let $W(u)'$ be the word $W(u)$ with this occurrence of $u$ deleted. Then $W(u)=1$ in $G(\sss)$(and thus $u\in R$) and $W'(u)=1$ in $G(\sss')$. Hence $u=1$ in $H$. Therefore every $u\in R$ is equal to 1 in $H$ and there exists a homomorphism from $P$ to $H$. It can be proved that this homomorphism is injective. This proves the Higman embedding theorem. Proving that $H$ has isoperimetric  function that polynomially depends on the time function of $\sss$ is of course more complicated (see \cite{OSsur}). Again it involves cutting diagrams into pieces which are treated separately.

\section{Open problems}\label{op}

Here we collect the open problems mentioned in the text. It may be more convenient for the reader to have all the problems in one place.

\begin{prob}[See \ref{ss:tcc}] Is there (time or space) complexity bound for the word problem in a finitely presented residually finite group. More precisely, what is the highest class in the time complexity hierarchy \cite{GJ} where the word problem in every finitely presented residually finite group belongs? An even more bold question: is the word problem of every finitely presented residually finite group in NP?
\end{prob}

\begin{prob} [D. Cohen \cite{Coh}, also attributed to Stallings, see \ref{ss:tif}] Is it true that for every finitely presented group $G$ there exists a constant $a>1$ such that $$f(n)\le a^{g(n)}$$ where $f, g$ are, respectively, the Dehn function and the isodiametric function of $G$.
\end{prob}

\begin{prob}[See \ref{ss:tchtt}] Find a common cause of the theorems of Cartan-Hadamard type from \ref{ss:tchtt}. \end{prob}

\begin{prob}[See \ref{ss:wr}] Find a combinatorial proof of Theorem \ref{Wen}. Is the Dehn function of the central square of $G_{10}$ equivalent to $n^2\log n$?
\end{prob}\begin{prob}[Rips, see \ref{ss:rc}] Is $F$ automatic?
\end{prob}\begin{prob}[See \ref{ss:rc}] Is $F$ semihyperbolic?
\end{prob}\begin{prob}[E. Rips, see \ref{ss:rp2}] Is the conjugacy problem solvable in every finitely presented group with quadratic Dehn function?
\end{prob}

\begin{prob}[See \ref{ss:tcpiag}] Can Remark \ref{rk91} be applied to prove that automatic groups have decidable conjugacy problem?
\end{prob}

\begin{prob}[See \ref{s:acogw}] Is it true that every finitely generated matrix group over a field (in particular, the field of rational numbers) embeds into a finitely presented group with quadratic Dehn function?
\end{prob}

\begin{prob}[See \ref{s:acogw}] Does every finitely generated group with word problem in NP embed into a group all of whose asymptotic cones are simply connected? Equivalently (by Theorem \ref{bors}): does every finitely presented group with polynomial isoperimetric function embed into a group all of whose asymptotic cones are simply connected?
\end{prob}

\medskip
{\bf About the list of references.} After each reference in the bibliography we put numbers of sections where we cite this reference. Each of the section numbers is the highest section number preceding the reference. Say, Section number 4.1 means that the reference appears in Section 4.1, before 4.1.A.

\printindex

\end{document}